\RequirePackage{amsmath}
\documentclass[runningheads]{llncs}
\usepackage[T1]{fontenc}
\newcommand{\compdia}{{\vcenter{\hbox{\scalebox{0.6}{$\Diamond$}}}}_E}
\newcommand{\compbox}{{\vcenter{\hbox{\scalebox{0.6}{$\Box$}}}}_E}
\newcommand{\compdiadelta}{{\vcenter{\hbox{\scalebox{0.6}{$\Diamond$}}}}_\Delta}
\newcommand{\compboxdelta}{{\vcenter{\hbox{\scalebox{0.6}{$\Box$}}}}_\Delta}

\newcommand{\bluebf}[1]{\textcolor{blue}{\textbf{#1}}}
\newcommand{\blue}[1]{\textcolor{blue}{#1}}

\renewcommand{\P}{\mathcal{P}}

\newcommand{\F}{\mathbb{F}}

\renewcommand{\emph}{\textbf}

\newcommand{\Prop}{\mathsf{Prop}}

\newcommand{\nomi}{\mathbf{i}}
\newcommand{\nomj}{\mathbf{j}}
\newcommand{\nomk}{\mathbf{k}}

\newcommand{\cnomm}{\mathbf{m}}
\newcommand{\cnomn}{\mathbf{n}}

\newcommand{\marginnote}[1]{\marginpar{\raggedright\tiny{#1}}}

\newcommand{\val}[1]{[\![{#1}]\!]}
\newcommand{\descr}[1]{(\![{#1}]\!)}
\renewcommand{\phi}{\varphi}

\usepackage{lscape}
\usepackage{pdflscape}

\usepackage[margin = 2cm]{geometry}

\usepackage{stmaryrd}

\usepackage{cmll}

\usepackage{comment}

\usepackage{amscd}
\usepackage{amssymb}
\usepackage{amsmath}
\usepackage{mathrsfs}
\usepackage{color}
\usepackage{xspace}
\usepackage{bussproofs}
\EnableBpAbbreviations

\usepackage{etoolbox,refcount}
\usepackage{multicol}
\usepackage{tikz}

  \usepackage{txfonts}

\usepackage{enumerate}

\usepackage{centernot}

\usepackage{mathtools}

\usepackage{hyperref}

\usepackage{cleveref}

\usepackage{relsize}

\usepackage{multirow}

\usepackage{array}
\newcolumntype{C}[1]{>{\centering\arraybackslash}p{#1}}
\newcolumntype{L}[1]{>{\arraybackslash}p{#1}}

\usepackage{amsmath}

\usepackage{multirow}
\usepackage{colortbl}
\usepackage{amssymb}
\usepackage{graphicx}
\begin{document}
\title{Modal reduction principles: a parametric shift to graphs}
%
%
\author{Willem Conradie\inst{1,5}\orcidID{0000-0001-9906-4132} \and
Krishna Manoorkar\inst{2,4}\orcidID{0000-0003-3664-7757} \and
Alessandra Palmigiano\inst{2,3}\orcidID{0000-0001-9656-7527}\and
 Mattia Panettiere\inst{2,4}\orcidID{0000-0002-9218-5449}}%
\authorrunning{Conradie et al.}
%
\institute{School of Mathematics, University of the Witwatersrand, Johannesburg
 \and
Vrije Universiteit, Amsterdam, The Netherlands\\
\and
Department of Mathematics and Applied Mathematics, University of Johannesburg, South Africa
\and
KPMG, Amstelveen, The Netherlands
\and 
National Institute for Theoretical and Computational Sciences (NITheCS), South Africa
}
\maketitle              
\begin{abstract}
Graph-based frames have been introduced as a logical
framework which internalizes an inherent boundary to knowability (referred to as ‘informational entropy’), due, e.g., to perceptual,  evidential or linguistic limits.  They also support the interpretation of lattice-based (modal) logics as
hyper-constructive logics of evidential reasoning. Conceptually, the present paper proposes graph-based frames as a formal framework suitable for generalizing
Pawlak’s rough set theory to a setting in which inherent limits to knowability exist and need to be considered. Technically, the present paper establishes
systematic connections between the first-order correspondents of Sahlqvist modal
reduction principles on Kripke frames, and on the more general relational environments of graph-based and polarity-based frames. This work is part of a research line aiming at: (a)
comparing and inter-relating the various (first-order) conditions corresponding to
a given (modal) axiom in different relational semantics; (b) recognizing when
first-order sentences in the frame-correspondence languages of different 
relational structures encode the same “modal content”; (c) meaningfully transferring   relational properties  across different semantic contexts. The present paper
develops these results for the graph-based semantics, polarity-based semantics, and all Sahlqvist modal reduction principles.

As an application, we study well known modal axioms in rough set theory (such as those corresponding to  seriality, reflexivity, and transitivity) on graph-based frames and show that, although these axioms
correspond to different first-order conditions on graph-based frames, their intuitive meaning is retained. This allows us to introduce the notion of hyperconstructivist approximation spaces as  the subclass of graph-based frames defined by the first-order conditions   
corresponding to the same modal axioms defining classical generalized approximation spaces, and to transfer the properties and the intuitive understanding
of different approximation spaces to the more general framework of graph-based
frames. The approach presented in this paper provides a base for systematically
comparing and connecting various formal frameworks
in rough set theory, and for the transfer of insights across different frameworks.
\end{abstract}
\section{Introduction}

\paragraph{Generalizations of Rough Set Theory.} The contributions of the present paper lie within, and are motivated by, several 
research lines in  Rough Set Theory \cite{pawlak1984rough,pawlak1998rough,pawlak2007rudiments} 
aimed at generalizing Pawlak's original setting both  for extending its scope, and for making the theory more modular and systematic. Instrumental to these generalizations were several cross-fertilizations with 
other research lines such as (two-valued \cite{blackburn2002modal} and many-valued \cite{fitting1991many,fitting1992many,bou2011minimum,godo2003belieffunctions}) modal logic, and Formal Concept Analysis \cite{ganter2012formal}. 
Since Orlowska's work identifying the modal logic S5 as the logic of information spaces \cite{orlowska1994rough}, and building on the insight that formalizing the indiscernibility relation  as an equivalence relation might be too restrictive \cite{zakowski1983approximations,rasiowa1984rough},
Vakarelov \cite{vakarelov2005modal,vakarelov1991model} introduces a framework in which lower and upper approximations are associated with different indiscernibility relations; 
Yao and Lin \cite{yao1996-Sahlqvist} account systematically for various notions of generalized  approximation spaces, each  defined in terms of certain first-order properties (e.g.~reflexivity, symmetry, transitivity, seriality, Euclideanness) corresponding (via the well known Sahlqvist theory in modal logic \cite{sahlqvist1975completeness,van1984correspondence}) to  various well known modal axioms (which are, more specifically, {\em modal reduction principles} \cite{van1976mrp}, cf.~Definition \ref{ssec:mrp}); moreover, in the same paper, this approach is systematically  extended  to the setting of graded and probabilistic approximation spaces, via graded \cite{fine1972so} and probabilistic modal logic \cite{fagin1990logic}; using duality-theoretic and correspondence-theoretic techniques,
  Balbiani and Vakarelov \cite{balbiani2001first} introduce a sound and complete decidable poly-modal logic of weak and strong indiscernibility and complementarity for information systems. Also building on insights and results stemming from duality and representation theory for modal logic,
 Banerjee and Chakraborty \cite{banerjee1994rough,banerjee1996rough,banerjee2004algebras} study  the mathematical framework of rough set theory from a more abstract viewpoint grounded on the notion of {\em rough  algebras}, i.e.~modal algebras the non-modal reducts of which are more general than Boolean algebras, and hence  allow access to a base logic which, being {\em paraconsistent}, is strictly more general  than classical propositional logic. Based on this algebraic perspective,  proof-theoretic frameworks have also been developed for the logics of rough algebras \cite{saha2014algebraic,saha2016algebraic,ma2018sequent,greco2019proper,greco2019logics,icla_algebraic,icla_relational}. 

\paragraph{Unifying RST and FCA.} Closely related 
to these developments are proposals of unifying frameworks for rough set theory and formal concept analysis \cite{kent1996rough,yao2004comparative,yao2004concept,cattaneo2016connection,yao2006unifying,yao2016rough,formica2018integrating,yao2020three}, both in the two-valued and in the many-valued setting. This unification is particularly sought after for modelling classification problems characterized by partial or incomplete information.

Drawing on insights developed in the  research lines discussed above, in \cite{CONRADIE2021371}, 
a  framework unifying rough set theory and formal concept  analysis is proposed, based on duality-theoretic and algebraic insights stemming from the mathematical theory of modal logic \cite{ConPal12,CoGhPa14,ConPalSou}. The main tool exploited in \cite{CONRADIE2021371} is an embedding, defined as in  \cite{moshier2016relational}, from the class of sets to the class of formal contexts (cf.~\cite[Definition 3.3]{CONRADIE2021371}) which {\em preserves} the Boolean (powerset) algebra associated with each set, in the sense that the concept lattice of any formal context in the image of the embedding  is order-isomorphic to the powerset algebra of its pre-image (cf.~\cite[Proposition 3.4]{CONRADIE2021371}). This embedding can be extended (cf.~\cite[Proposition 3.7]{CONRADIE2021371}) to a {\em complex-algebra-preserving} embedding from the class of Kripke frames to the class of {\em enriched formal contexts}, or {\em polarity-based frames}.\footnote{A polarity (or formal context) is a triple $\mathbb{P} =(A, X, I)$ where $A$ and $X$ are two sets, and $I\subseteq A \times X$. For the modal signature $\tau = \{\Box, \Diamond\}$, a {\em polarity-based frame} is a structure $\mathbb{F} = (\mathbb{P}, R_{\Box}, R_{\Diamond})$, where $\mathbb{P}$ is polarity, and $R_\Box \subseteq A \times I, R_\Diamond \subseteq X \times I$ are relations that satisfy certain compatibility conditions (cf.~\cite[Definition 2]{TarkPaper2017}). } These  structures have been defined in \cite{conradie2016categories,TarkPaper2017} in the context of a research program aimed at developing the logical foundations of categorization theory, and consist of  formal contexts enriched with additional relations, each   defined on the basis of duality-theoretic considerations (discussed in \cite{Conradie2020NondistributiveLF}), and supporting the interpretation of a normal modal operator. In \cite{CONRADIE2021371}, the study of the preservation properties of the  embedding mentioned above  leads to the understanding that a connection can be established (cf.~\cite[Sections 3.4 and  4]{CONRADIE2021371}) between the first-order conditions on Kripke frames and those on enriched formal contexts corresponding to the  modal axioms considered in the literature in rough set theory  for generalizing approximation spaces. 

\paragraph{Rough concepts and beyond: towards parametric correspondence.} Besides motivating the identification of  {\em conceptual approximation spaces} (cf.~\cite[Section 5]{CONRADIE2021371}) as the subclass of enriched formal contexts  serving as  a unifying environment for rough set theory and formal concept analysis, 
%
these developments naturally elicit the questions of whether the connection found and exploited in \cite{CONRADIE2021371}, between the first-order conditions corresponding to a finite number of well known modal axioms in different semantic settings, can be systematically extended to significantly large classes of modal axioms, and also to different semantic settings.
Concrete examples pointing to the existence of such systematic connections  
cropped up, not only in relation with the semantics of enriched formal contexts (aka polarity-based frames) for lattice-based  modal logics (sometimes referred to as `non-distributive' modal logics), but also for an alternative   relational semantic setting for the same logics (namely,  {\em graph-based frames}), and have been discussed in 
\cite{TarkPaper2017,graph-based-wollic,CONRADIE2021371,conradie2016categories,vanBenthem2001}.
In \cite{conradie2022modal}, both questions are answered positively:  the above-mentioned connection is generalized from a finite set of modal axioms  to the class of {\em Sahlqvist modal reduction principles} \cite{vanBenthem:Reduction:Principles,van1976mrp} (cf.~Section \ref{ssec:mrp}) and is also extended  to  two-valued and many-valued Kripke frames, and  two-valued and many-valued enriched formal contexts. These results initiate a line of research which aims at making correspondence theory not just (methodologically) unified \cite{CoGhPa14}, but also
(effectively) parametric \cite{conradie2022modal}.

\paragraph{Extending parametric correspondence: the case of graph-based frames.} The contributions of the present paper include positive answers to these same questions  for Sahlqvist  modal reduction principles, relative to Kripke frames,   graph-based frames, and polarity-based frames, pivoting on {\em graph-based frames}. 
Graph-based frames \cite{conradie2015relational,graph-based-wollic,Conradie2020NondistributiveLF,craig2015tirs} (cf.~Definition \ref{def:graph:based:frame:and:model}) are relational structures based on graphs $(Z, E)$, where $Z$ is a nonempty set and $E\subseteq Z\times Z$ is a reflexive relation. Graph-based frames provide complete semantics to non-distributive modal logic \cite{graph-based-wollic}, which fact, as also discussed in \cite{Conradie2020NondistributiveLF},  supports the  interpretation of this logic as the logic of {\em informational entropy}, i.e.~an inherent boundary to knowability due e.g.~to perceptual, theoretical, evidential or linguistic limits. In  graphs $(Z, E)$ on which these relational structures are based, the reflexive relation $E$  is interpreted as the {\em inherent indiscernibility} relation induced by informational entropy, much in the same style as  Pawlak's  approximation spaces in rough set theory. However,  $E$ is only required to be reflexive but, in general, neither transitive nor symmetric, which is in line with proposals in rough set theory \cite{wybraniec1989generalization,yao1996-Sahlqvist,vakarelov2005modal}, also discussed above, that indiscernibility does not need to be modelled by equivalence relations. Another  key difference is that,  rather than generating modal operators which associate any subset of $Z$ with its definable $E$-approximations,  $E$ generates a complete lattice (i.e.~the lattice of formal $E^c$-concepts, as shown in the picture below). Hence, the role played by  $E$ in generating the 
complex algebras of graph-based frames  is analogous to the role played by  the order relation of intuitionistic Kripke frames in generating the Heyting algebra structure of their associated complex algebras: both $E$ and the order ultimately support the interpretation of the propositional (i.e.~non-modal) connectives. The resulting propositional fragment of non-distributive modal logic  is sound (and complete) w.r.t.~general lattices, and is hence  much weaker than both classical and intuitionistic logic. Rather than being a flaw, this is a feature of this framework, which precisely captures and internalizes the correct reasoning patterns under informational entropy into the propositional fragment of the language.
\begin{figure}[ht]
\label{img:usualimage}
\begin{center}
\resizebox{12cm}{!}{
\begin{tikzpicture}
\draw[very thick] (7.5, 2.5) -- (6, 3.5) --
	(6, 4.5) -- (7.5, 5.5) -- (9, 4)  -- (7.5, 2.5);
	\filldraw[black] (7.5,2.5) circle (3 pt);
	\filldraw[black] (6,3.5) circle (3 pt);
	\filldraw[black] (6,4.5) circle (3 pt);
	\filldraw[black] (7.5,5.5) circle (3 pt);
	\filldraw[black] (9,4) circle (3 pt);
	\draw (7.5, 2.2) node {$(\varnothing,uvw)$};
	\draw (5.7, 3.2) node {$(w,uv)$};
	\draw (5.7, 4.8) node {$(vw,u)$};
    \draw (5.5, 3.5) node {{$V(p)$}};
	\draw (7.5, 5.8) node {$(uvw,\varnothing)$};
		\draw (5.7, 5.8) node {$V(p\vee q) =$};
	\draw (9.3, 3.6) node {$(u,w)$};
    \draw (9.1, 4.5) node {$V(q)$};
    \draw (4.3,4) node{\Huge{$\leftrightsquigarrow$}};
\draw[very thick] (1, 3.5) -- (0, 4.5) --
	(2, 3.5) -- (1, 4.5);
	\draw[very thick] (0, 3.5) -- (2, 4.5);
	\filldraw[black] (0,3.5) circle (3 pt); 
	\filldraw[black] (0,4.5) circle (3 pt); 
	\filldraw[black] (1,3.5) circle (3 pt); 
	\filldraw[black] (1,4.5) circle (3 pt); 
	\filldraw[black] (2,3.5) circle (3 pt); 
	\filldraw[black] (2,4.5) circle (3 pt); 
    \draw (-0.4,4.5) node {$Z$};
    \draw (-0.4,4) node {$E^c$};
    \draw (-0.4,3.5) node {$Z$};
	\draw (0,4.8) node {$u$};
	\draw (1,4.8) node {$v$};
	\draw (2,4.8) node {$w$};
	\draw (0,3.2) node {$u$};
	\draw (1,3.2) node {$v$};
	\draw (2,3.2) node {$w$};
    
     \draw (-1.5,4) node{\Huge{$\leftrightsquigarrow$}};
     \filldraw[black] (-5,4) circle (3 pt);
     \draw (-5,3.7) node {$u$};
  \draw (-5, 4.9) node {$q$};
   \draw (-5, 3.3) node {$p$};
	\filldraw[black] (-4,4) circle (3 pt);
	 \draw (-4,3.7) node {$v$};
	 
	 \draw (-4, 5.2) node {$p\vee q$};
	  \draw (-5, 5.2) node {$p\vee q$};
	  
   \draw (-4, 3.3) node {$p$};
	\filldraw[black] (-3,4) circle (3 pt);
	 \draw (-3,3.7) node {$w$};
 \draw (-3, 4.9) node {$p$};
  \draw (-3, 3.3) node {$q$};
	   
 \draw (-3, 5.2) node {$p\vee q$};
	\draw[very thick, ->] (-4.8, 4) -- (-4.2, 4);
	\draw[very thick, ->] (-3.8, 4) -- (-3.2, 4);
	\draw[very thick,  <-] (-5.1, 4.1)  .. controls (-5.7, 4.8) and  (-4.2, 4.8)  .. (-4.9, 4.1); 
\draw[very thick,  <-] (-4.1, 4.1)  .. controls (-4.7, 4.8) and  (-3.2, 4.8)  .. (-3.9, 4.1); 
\draw[very thick,  <-] (-3.1, 4.1)  .. controls (-3.7, 4.8) and  (-2.2, 4.8)  .. (-2.9, 4.1); 
\end{tikzpicture}
}
\end{center}
\caption{From left to right, a reflexive graph, its associated formal context, and its concept lattice. Valuations of atomic propositions and formulas to elements of the concept lattice dually correspond to
satisfaction/refutation  relations between  states of the graph and formulas. In the picture, formulas above (resp.~below) a  state of the reflexive graph are satisfied (resp.~refuted) at that state.}
\end{figure}
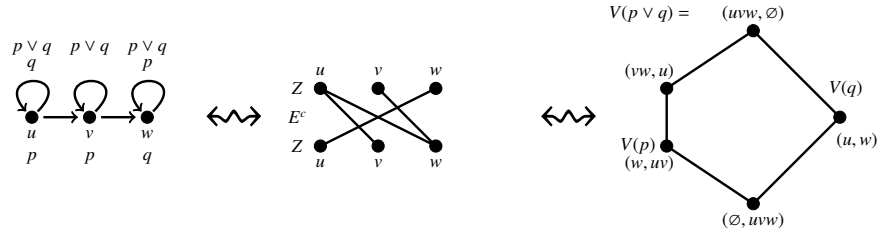

\paragraph{Inherent vs.~epistemic indiscernibility.} The environment of graph-based frames makes it possible to differentiate between inherent indiscernibility (modelled by $E$) and other (possibly more contingent, e.g.~subjective, or epistemic) types of indiscernibility (modelled by the additional relations which support the interpretation of the modal operators). 
 The conceptual difference between inherent and other types of indiscernibility is reflected into the different roles taken by these relations w.r.t.~the interpretation of the logical language: the elements of the concept lattice represented on the right-hand side of the picture do not represent definable approximations of predicates, but rather, they represent  `all there is to know' about predicates,  i.e.~the theoretical horizon to knowability, given the inherent boundary  encoded into $E$. In their turn, the elements of the concept lattices are approximated by the modal operators induced by the additional relations of the graph-based relational structures, and this approximation reflects the epistemic (subjective) indiscernibility of agents. 

\paragraph{Methodology for the shift from Kripke frames to graph-based frames.} The relation $E$ is intended as the generalization of the identity relation $\Delta$ on sets, and the embedding   from sets to reflexive graphs corresponding to this generalization, which is defined by the assignment $S\mapsto (S, \Delta_S)$, plays the same role in the present context as the role played by the embedding from sets to formal contexts discussed above, and displays properties analogous to those of that embedding (namely, it can be extended to a {\em complex algebra-preserving} embedding from Kripke frames to graph-based frames). Thus, conceptually, the environment of classical Kripke frames can be understood as the limit case of the present environment in which any two states are inherently indistinguishable only if they are identical. 


\paragraph{Graph-based frames and non-distributive logics as hyper-constructivist logics.} In \cite{Conradie2020NondistributiveLF}, it is argued that graph-based frames  support a view of non-distributive (modal) logics as {\em hyper-constructivist logics}, i.e.~logics in which the principle
of excluded middle fails at the meta-linguistic level.
Indeed, as illustrated by the picture above in the case of the proposition variable $q$, the {\em satisfaction} and  {\em refutation} set of a formula $\varphi$  on graph-based  models are not necessarily set-theoretic complements of one other; therefore, it is possible for a formula to be neither satisfied nor refuted at a state. The progressive generalizations, from classical propositional logic to intuitionistic logic, and from intuitionistic logic to 
the propositional fragment of non-distributive (modal) logic (hereafter: non-distributive logic) 
is reflected  in two ways: concerning relational models, the move from classical to intuitionistic logic is captured as the move from sets to preorders (i.e.~sets equipped with reflexive and transitive relations), and the move from intuitionistic to non-distributive logic is captured as the move from preorders to reflexive graphs (i.e.~transitivity is dropped); concerning the notion of truth, classical truth is replaced in intuitionistic logic with a finer-grained notion of {\em provable truth}, which makes the law of excluded middle fail at the level of the logical language, since at a given state there might be neither a proof of $\varphi$ nor a proof of $\neg\varphi$; however, at the metalogical level of satisfaction and refutation, if a formula is not satisfied at a given state of an intuitionistic model, then it is refuted at that state. 
\paragraph{Hyper-constructivist logics and evidential truth.} As discussed in \cite{Conradie2020NondistributiveLF}, the failure of the law of excluded middle at the metalogical level supported by the graph-based semantics suggests that the notion of truth supported by non-distributive (modal) logics is even more refined than intuitionistic truth, and can be intuitively construed as an {\em evidence-based truth}: that is, to refute a formula at a given state, the lack of positive evidence is not enough, but rather, there needs to be evidence {\em against}. 

\paragraph{Many-valued graph-based frames.} In \cite{Conradie2019/08},  the many-valued version is introduced of the graph-based semantics of  \cite{graph-based-wollic} for two axiomatic extensions of the basic normal non-distributive modal logic, and in particular their potential is explored for modelling situations in which informational entropy derives from the theoretical frameworks under which empirical studies are conducted. This has been proposed as a suitable framework 
 for modelling competing
theories in the empirical sciences.
In a similar vein, in \cite{CONRADIE2021115}, the many-valued graph-based semantic setting for non-distributive modal logics has been further extended to a multi-type setting, and has been proposed as a suitable formal framework for modelling socio-political competition.

For the readers' convenience, the relevant developments in  graph-based and polarity-based semantics for modal logic and of their applications are summarised in  Table \ref{tab:relatedwork}. 

\begin{table}[ht]

\begin{center}
\setlength\tabcolsep{2mm}.
\begin{tabular}{|l|l|l|}
    \hline
    Date &Development &Reference\\
    \hline
    2015
    &\begin{minipage}{0.75\textwidth}
    \vspace{2mm}
    Graph-based semantics for (non-distributive) modal logics 
    \vspace{2mm}
    \end{minipage}
    & \cite{conradie2015relational}\\ \hline
    2016
    &\begin{minipage}{0.75\textwidth}
    \vspace{2mm}
    Introduction of enriched RS-polarities  as semantic framework for an epistemic logic of categories and (formal) concepts. 
    \vspace{2mm}
    \end{minipage}
    &\cite{conradie2016categories}\\ \hline
    2017 
    &\begin{minipage}{0.75\textwidth}
    \vspace{2mm}
     Approach in \cite{conradie2016categories} generalized to arbitrary (hence not necessarily RS) polarities, adding a `common knowledge'-type construction and formalizing key concepts in categorization theory. 
    \vspace{2mm}
    \end{minipage}
    & \cite{TarkPaper2017}\\ \hline
    2019
    &\begin{minipage}{0.75\textwidth}
    \vspace{2mm}
    Complete axiomatization of modal logic of graph-based frames.\\
    Examples of correspondents of some well-known MRPs on graph-based frames; suggestive similarities noted, but no systematic theory of connection with classical Kripke frames.\\
    Introduction of the notion of parameterized composition ($E$-composition) in the context of graph-based frames.\\
    Interpretation of graph-based models for reasoning about informational entropy. 
    \vspace{2mm}
    \end{minipage}
    & \cite{graph-based-wollic}\\ \hline
    2019
    &\begin{minipage}{0.75\textwidth}
    \vspace{2mm}
    General correspondence theory for the class of  logics algebraically captured by varieties of normal lattice expansions(LE-logics), with  graph-based semantics as a particular example of associated relational semantics.\\
    Correspondents as pure quasi-inequalities to with standard translations can be applied. No general theory establishing systematic connections between  first-order correspondents of the same modal axioms across relational semantics, yet. 
    \vspace{2mm}
    \end{minipage}
    &\cite{CoPa:non-dist}\\ \hline
    2020
    &\begin{minipage}{0.75\textwidth}
    \vspace{2mm}
    Systematic study of interconnections between polarity-based and graph based-semantics. Methodology for extending both to accommodate arbitrary additional, normal  connectives. Study of possible interpretations of these semantic structures and arguable preservation of intuition/meaning of well-know epistemic axioms across these semantics and interpretations.
    \vspace{2mm}
    \end{minipage}
    &\cite{Conradie2020NondistributiveLF}\\  \hline
    2021
    &\begin{minipage}{0.75\textwidth}
    \vspace{2mm}
    Novel unification of Rough Set Theory and Formal Concept Analysis generalizing Pawlak’s original approximation with polarity-based frames.
    \vspace{2mm}
    \end{minipage}
    &\cite{CONRADIE2021371}\\ \hline
    2021
    &\begin{minipage}{0.75\textwidth}
    \vspace{2mm}
    Many-valued graph-based 
    semantics extended to a
    multi-type setting;  proposed as  formal framework for modelling
    socio-political competition.
    \vspace{2mm}
    \end{minipage}
    &\cite{Tark1}\\ \hline
    2022
    &\begin{minipage}{0.75\textwidth}
    \vspace{2mm}
    Introduction of many-valued graph-based frames and models and application to reasoning about competing theories.
    \vspace{2mm}
    \end{minipage}
    & \cite{Conradie2019/08}\\ \hline
    2022
    &\begin{minipage}{0.75\textwidth}
    \vspace{2mm}
    General theorems expressing the first-order frame correspondents of all Sahlqvist MRPs as inclusions of (parameterized) relational compositions, both over Kripke frames and polarity based frames in both crisp and many-valued cases together with systematic, effective, parametric translations between all of these. (The present paper does the same but for graph-based frames rather than polarity-based frames.) 
    \vspace{2mm}
    \end{minipage}
    & \cite{conradie2022modal}\\ \hline
\end{tabular}
\end{center}

\caption{ \label{table:developments}Summary of relevant related work on graph and and polarity based semantics.}
    \label{tab:relatedwork}
\end{table}

\paragraph{Main contributions.}   
Besides extending the network of systematic connections, established in \cite{conradie2022modal}, among various relational semantics for modal logic so as to include graph-based frames, the results of the present paper, which we summarise in Table \ref{table:results}, lay the foundations of a systematic generalization of Rough Set Theory to settings which require  forms of reasoning whose associated notion of truth is evidence-based in the sense indicated above. Mathematically, these settings can be formalized by environments whose naturally associated logics are  hyper-constructive.  

\paragraph{Comparison with  results on polarity-based frames.} As observed in \cite{Conradie2020NondistributiveLF}, although the polarity-based semantics and the graph-based semantics arise from the same duality-based methodology and they both provide complete semantics to  non-distributive (modal) logics, they support very different interpretations of what these logics are {\em about}, and   the  first-order languages canonically associated with these two classes of relational structures are also very different. This difference can be appreciated e.g.~when comparing the first-order correspondents on polarity-based frames and on graph-based frames of a simple and well known Sahlqvist modal reduction principle, as reported in the following table.\footnote{In any frame based on a polarity $(A, X, I)$, variables $x,y,z,\ldots$ range over $X$, and variables $a,b,c,\ldots$ range over $A$.}
   \begin{center}
   \begin{tabular}{|l|ll|}
   \hline
   Axiom & Correspondents &\\
   \hline

   & Kripke frames: & $\forall x\forall y\forall z \: ((xRy \ \& \ yRz) \Rightarrow xRz)$ \\
   $\Box p \vdash \Box\Box p$ & Polarity-based frames: & $\forall a\forall x (aIx \Rightarrow \forall y(\forall b(bR_\Box x \Rightarrow bIy) \Rightarrow aIy ) )$ \\
   & Graph-based frames: & $\forall z_1\forall z_2\forall z_3((\forall z_4(z_3Ez_4 \Rightarrow z_4 R_\Box z_1) \ \& \ z_2 R_\Box z_3 ) \Rightarrow z_2 R_\Box z_1 )$\\
   \hline
   \end{tabular}
   \end{center}
   In order to be able to systematically connect the first-order correspondents of any Sahlqvist modal reduction principle across these relational semantics, we need to establish a common ground between their associated first-order languages. Establishing this common ground is precisely what the core technical contribution of the present paper is about. Although, as indicated earlier, the methodology followed in the present paper is analogous to the one adopted in \cite{conradie2022modal}, transferring it to graph-based frames requires specific steps, which are presented in Sections \ref{sec: composing crisp} and \ref{sec:crisp}.

\begin{table}[h!] 
    \centering
    \begin{tabular}{|c|c|}
    \hline
    \begin{minipage}{0.20\textwidth}
    \centering
    \vspace{2mm}
    Hyper-constructivist\\ approximation spaces
    \vspace{2mm}
    \end{minipage}
    &
    \begin{minipage}{0.70\textwidth}
    \vspace{2mm}
    Introduction of hyper-constructivist approximation spaces (see Definition \ref{def:graph based approximation spaces}), i.e., approximation spaces based on graph based frames, where the law of excluded middle fails at the {\em frame level}. 
    \vspace{2mm}
    \end{minipage}\\
    \hline
    \begin{minipage}{0.20\textwidth}
    \vspace{2mm}\centering
    $E$-mediated and non $E$-mediated compositions
    \vspace{2mm}
    \end{minipage}
    &
    \begin{minipage}{0.70\textwidth}
    \vspace{2mm}
    Introduction and study of the fundamental properties of  three notions of relational compositions  (cf.~Definition \ref{def:relational composition}) which reflect the semantics of the compositions of connectives in graph-based frames (see Lemma \ref{lemma:properties of homogeneous composition}), and which  coincide with the usual composition of relations in the special case in which a graph-based frame is a Kripke frame (see Lemma \ref{lemma:compdia compbox delta}).
    \vspace{2mm}
    \end{minipage}\\
    \hline
    \begin{minipage}{0.20\textwidth}
    \vspace{2mm}\centering
    Relational languages\\ $\mathsf{GRel}_\mathcal{L}$ and $\mathsf{KRel}_\mathcal{L}$
    \vspace{2mm}
    \end{minipage}
    &
    \begin{minipage}{0.70\textwidth}
    \vspace{2mm}
    Introduction of suitable relational languages for Kripke frames and graph-based frames, which serve as the natural environment for the bridge between the first-order correspondents of modal reduction principles in the two semantics. Indeed, we show that such first-order correspondents can always be written as  `inclusions of relations', i.e.~term inequalities in these relational languages (see Lemmas \ref{lemma:molecular polarity} and \ref{lemma:molecular kripke}, and the discussions following the lemmas). These definitions hinge on the notions of composition discussed above.
    \vspace{2mm}
    \end{minipage}\\
    \hline
    \begin{minipage}{0.20\textwidth}\centering
    \vspace{2mm}
    Shifting of term inequalities in $\mathsf{KRel}_\mathcal{L}$
    \vspace{2mm}
    \end{minipage}
    &
    \begin{minipage}{0.70\textwidth}
    \vspace{2mm}
    Introduction of the notion of {\em shifting} between $\mathsf{KRel}_\mathcal{L}$-inequalities and $\mathsf{GRel}_\mathcal{L}$-inequalities   (cf.~Definition \ref{def:shifting}). This notion is exemplified in Table \ref{fig:MPS:Famous} for the best known mrps considered in Rough Set Theory. 
    \vspace{2mm}
    \end{minipage}\\
  \hline
    \begin{minipage}{0.20\textwidth}\centering
    \vspace{2mm}
    Parametric correspondence across Kripke frames and graph-based frames
    \vspace{2mm}
    \end{minipage}
    &
    \begin{minipage}{0.70\textwidth}
    \vspace{2mm}   
    Systematic connection between the correspondents of Sahlqvist mrps in Kripke frames and graph-based frames formulated and established  in terms of  the notion of shifting (see Theorem \ref{thm: sahlqvist shifting}).  
    \vspace{2mm}
    \end{minipage}\\
    \hline
    \begin{minipage}{0.20\textwidth}\centering
    \vspace{2mm}
    Lifting of term inequalities in $\mathsf{GRel}_\mathcal{L}$
    \vspace{2mm}
    \end{minipage}
    &
    \begin{minipage}{0.70\textwidth}
    \vspace{2mm}
    Introduction of the notion of {\em lifting} of $\mathsf{GRel}_\mathcal{L}$-inequalities to $\mathsf{PRel}_\mathcal{L}$-inequalities   (cf.~Definition \ref{def:lifting graph frames}), a relational language for polarity-based frames.
    \vspace{2mm}
    \end{minipage}\\
  \hline
    \begin{minipage}{0.20\textwidth}\centering
    \vspace{2mm}
    Parametric correspondence across graph-based and polarity-based frames
    \vspace{2mm}
    \end{minipage}
    &
    \begin{minipage}{0.70\textwidth}
    \vspace{2mm}   
    Systematic connection between the correspondents of Sahlqvist mrps in graph-based frames and polarity-based frames formulated and established  in terms of  the notion of lifting (see Theorem \ref{prop: sahlqvist lifting to polarities}). This extends the results in \cite{conradie2022modal}.  
    \vspace{2mm}
    \end{minipage}\\
    \hline
    \end{tabular}
    \caption{Main contributions of the present paper}
    \label{table:results}
\end{table}

\paragraph{Structure of the paper.} In Section \ref{sec:Preliminaries}, after introducing  notation  used throughout the paper (cf.~Section \ref{ssec:notation}), we collect preliminaries on the various (set-based and algebraic) generalizations of  Rough Set Theory (cf.~Sections \ref{ssec:approx spaces} and \ref{ssec:algebraic review}), on graph-based frames as semantic structures for general lattice-based modal logic (cf.~Section \ref{ssec:logics}), and on modal reduction principles (cf.~Section \ref{ssec:mrp}). 
In Section \ref{ssec:twocases studies}, we discuss two case studies in which informational entropy plays a relevant role and is accounted for in the setting of graph-based frames; 
in Section \ref{sec:Rough set theory on graph-based frames}, we introduce the notion of hyperconstructivist approximation spaces and discuss the methodology underlying this definition as an instance of the general transfer results presented in the following sections. 
In Section \ref{sec: composing crisp}, we introduce and study the properties of  several notions of compositions of relations in the settings of Kripke frames and of graph-based frames. Thanks to these definitions,   suitable propositional languages $\mathsf{GRel}_{\mathcal{L}}$ and $\mathsf{KRel}_{\mathcal{L}}$ can be introduced  (cf.~Section \ref{ssec:krel and grel}) in which first-order conditions on Kripke frames and on polarity-based frames can be represented.
In Section \ref{sec:crisp}, we show that  the correspondents of all Sahlqvist modal reduction principles on Kripke frames and on graph-based frames can be represented as term-inequalities of the languages $\mathsf{GRel}_{\mathcal{L}}$ and $\mathsf{KRel}_{\mathcal{L}}$. Represented in this way, a systematic connection can be established between them (cf.~Section \ref{ssec:shifting}), and with polarity-based frames (cf.~Section \ref{ssec:lifting}).
In Section \ref{sec:Applications to rough set theory}, we apply the main result of the previous section to describe the relationship between the first-order correspondents of well known modal axioms in Rough Set Theory 
in graph-based frames and in Kripke frames. We show that, based on this mathematical relationship, also the  intuitive meaning of these axioms can be systematically transferred across these different semantics. Finally, we conclude in Section \ref{sec:conclusions}.

\section{Preliminaries}\label{sec:Preliminaries}
The present section collects basic definitions and facts about  the various intersecting topics and results on which the contributions of the present paper  build. In Section \ref{ssec:notation}, we introduce the notation used throughout the paper. In Section \ref{ssec:approx spaces}, we collect the basic definitions  on approximation spaces and their modal logic. In Section \ref{ssec:algebraic review}, we review the literature on various algebraic approaches to rough set theory. In Section \ref{ssec:logics}, we collect basic definitions and facts on non-distributive modal logics and their graph-based frames. 
In Section \ref{ssec:mrp}, we collect preliminaries on inductive modal reduction principles and discuss a general and compact representation of their associated  output when running the algorithm ALBA \cite{CoPa:non-dist} on them. 

\subsection{Notational conventions and basic constructions}\label{ssec:notation}
We let $\Delta_U$ denote the identity relation on a set $U$, and we will drop the subscript when it causes no ambiguity. The superscript $(\cdot)^c$ denotes the relative complement of a subset of a given set. Hence, for any binary relation $R\subseteq U\times V$, we let $R^c\subseteq U\times V$ be defined by  $(u, v)\in R^c$ iff $(u, v)\notin R$.  For any such $R$ and any $U'\subseteq U$ and $V'\subseteq V$, we  let $R[U']: = \{v\in V\mid (u, v)\in R \mbox{ for some } u\in U'\}$ and $R^{-1}[V']: = \{u\in U\mid (u, v)\in R \mbox{ for some } v\in V'\}$, and write $R[u]$ and $R^{-1}[v]$ for $R[\{u\}]$ and $R^{-1}[\{v\}]$, respectively. Any such $R$ gives rise to the
  {\em semantic modal operators} $\langle R\rangle, [R]: \mathcal{P}(V)\to \mathcal{P}(U)$ s.t.~
  $\langle R\rangle W : = R^{-1}[W] $ and $ [R]W: = (R^{-1}[W^c])^c$ for any $W\subseteq V$.
For any  $T\subseteq U\times V$, and any $U'\subseteq U$  and $V'\subseteq V$, let
\begin{equation}\label{eq:def round square brackets}
\small
\begin{array}{ll}
T^{(1)}[U']:=\{v\mid \forall u(u\in U'\Rightarrow uTv) \}  & T^{(0)}[V']:=\{u\mid \forall v(v\in V'\Rightarrow uTv) \} \\
T^{[1]}[U']:=\{v\mid \forall u(u\in U'\Rightarrow uT^cv) \}  \quad & T^{[0]}[V']:=\{u\mid \forall v(v\in V'\Rightarrow uT^cv) \}
\end{array}
\end{equation}
Well known properties of these constructions (cf.~[\cite{davey2002introduction}, Sections 7.22-7.29]) are stated in the following lemma. 
 \begin{lemma} \label{lemma: basic} For any relation $T \subseteq U \times V$, and all $V'\subseteq V$, $U'\subseteq U$, $\mathcal{U}\subseteq \mathcal{P}(U)$, and $\mathcal{V}\subseteq \mathcal{P}(V)$,
\begin{enumerate}
\item $X_1\subseteq X_2\subseteq U$ implies $T^{(1)}[X_2]\subseteq T^{(1)}[X_1]$, and $Y_1\subseteq Y_2\subseteq V$ implies $T^{(0)}[Y_2]\subseteq T^{(0)}[Y_1]$.
\item $U'\subseteq T^{(0)}[V']$ iff  $V'\subseteq T^{(1)}[U']$.
 \item $U'\subseteq T^{(0)}[T^{(1)}[U']]$ and $V'\subseteq T^{(1)}[T^{(0)}[V']]$.
 \item $T^{(1)}[U'] = T^{(1)}[T^{(0)}[T^{(1)}[U']]]$ and $T^{(0)}[V'] = T^{(0)}[T^{(1)}[T^{(0)}[V']]]$.
 \item $T^{(0)}[\bigcup\mathcal{V}] = \bigcap_{V'\in \mathcal{V}}T^{(0)}[V']$ and $T^{(1)}[\bigcup\mathcal{U}] = \bigcap_{U'\in \mathcal{U}}T^{(1)}[U']$.
\end{enumerate}
\end{lemma}
All properties listed in Lemma  \ref{lemma: basic} hold with $(\cdot)^{[0]}$ and $(\cdot)^{[1]}$ replacing $(\cdot)^{(0)}$ and $(\cdot)^{(1)}$ respectively, by instantiating  $T: = T^c$, so we will refer to this lemma also when invoking these properties.

\begin{lemma}[\cite{conradie2022modal}, Lemma 3.10]
\label{lemma:properties of square bracket superscript}
For any relation $T \subseteq U \times V$, and any $W\subseteq V$,
\begin{center}
    $T^{[0]}[W^c] = (T^{-1}[W^c])^c = [T]W $ \ and \ $(T^{[0]}[W])^c = T^{-1}[W] = \langle T\rangle W$.
\end{center}
\end{lemma}
\subsection{Approximation spaces and their modal logic}\label{ssec:approx spaces}

 {\em Approximation spaces} have been introduced by Pawlak \cite{pawlak1982rough} as the basic formal environment for Rough Set Theory.  These are structures $\mathbb{X} = (S, R)$ such that $S$
is a nonempty set, and $R\subseteq S\times S$ is an equivalence relation. For any such $\mathbb{X}$ and any $Z\subseteq S$, the {\em upper} and {\em lower approximations} of $Z$ are respectively defined as follows:
\[\overline{Z}: = \bigcup\{R[z]\mid z\in Z\} \quad \mbox{ and }\quad \underline{Z}: = \bigcup\{R[z]\mid z\in Z\mbox{ and } R[z]\subseteq Z\}. \]
A {\em rough set} of $\mathbb{X}$ is a pair $(\underline{Z},\overline{Z})$ for any  $Z\subseteq S$ (cf.~\cite{banerjee1996rough}).
 Since $R$ is an equivalence relation, \[\overline{Z} = \langle R\rangle Z  \quad \mbox{ and }\quad \underline{Z} = [R] Z,\] with $\langle R\rangle$ and $[R]$  validating the axioms of the classical normal modal logic S5 (cf.~\cite{orlowska1994rough}). 
 
 As approximation spaces are exactly the Kripke frames for  the normal modal logic S5 \cite{orlowska1994rough}, this logic  completely axiomatizes Pawlak's rough set theory. Given a language for classical modal logic over a set $\Prop$ of proposition variables  and  a valuation $V:\Prop\to \mathcal{P}(S)$ on $\mathbb{X}$, the satisfiability on models $M = (\mathbb{X}, V)$ is defined as usual in modal logic. The extension $\val{\phi}\subseteq S$ of any formula $\phi$ is defined recursively as follows:
\begin{center}
\begin{tabular}{llll}
$M, w \Vdash p$                    & iff & $w \in V(p)$&$\val{p} = V(p)$\\
$M, w \Vdash \phi \lor \psi$    & iff & $ w \in \val{\phi} $ or $w \in\val{\psi}$& $\val{\phi \lor \psi} = \val{\phi} \cup \val{\psi}$\\
$M, w \Vdash \phi \land \psi$ & iff & $ w \in \val{\phi} $ and $w \in \val{\psi}$&$\val{\phi \land \psi} = \val{\phi} \cap \val{\psi}$\\
$M , w \Vdash \neg\phi$ & iff & $ w \notin \val{\phi}$&$\val{\neg\phi} = \val{\phi}^c$\\
$M, w \Vdash \Box\phi$ & iff & $wR u$ implies $u  \in \val\phi$ &$\val{\Box\phi }=  [R]\val{\phi}$\\
$M, w \Vdash \Diamond\phi$ & iff & $wRu$ for some $u \in \val{\phi}$ &$\val{\Diamond\phi} = \langle R\rangle\val{\phi}$\\
\end{tabular}
\end {center}
As is well known \cite{orlowska1994rough}, the modal logic S5 is the axiomatic extension of the basic normal modal logic $K$ with the following axioms, respectively corresponding to the  reflexivity, transitivity and symmetry of the indiscernibility relation $R$.
\[ \Box\phi \rightarrow \phi \quad\quad  \Box \phi \rightarrow \Box\Box\phi \quad\quad  \phi \rightarrow \Box\Diamond\phi.\]

Motivated by the idea that the indiscernibility relation does not need to be an equivalence relation, the following definition provides the environment for a generalized and more modular definition of approximation space.  

\begin{definition}[\cite{yao1998generalized}, Section 2]
\label{def: gen approx space}
  A {\em generalized approximation space} is a tuple $\mathbb{X}= (S,R)$ such that $S$  is a non-empty set, and $R \subseteq S \times S$ is any relation.  For any such $\mathbb{X}$ and any $Z \subseteq S$, the {\em upper} and {\em lower approximations} of $Z$ are respectively
  \[\overline{Z}: = \langle R \rangle Z=  \bigcup\{R^{-1}[z]\mid z\in Z\} \quad \mbox{ and }\quad \underline{Z}: = [R] Z=  \{s \in S \mid R[s]\subseteq Z\}. \]
The following  table lists the names of subclasses of generalized approximation spaces in terms of their defining first-order condition. 
\begin{center}
\begin{tabular}{|c|c|}
\hline
\textbf{Class} & \textbf{first-order condition}\\
\hline
     serial & $\forall s\exists t(R(s, t))$  \\
    \hline
     reflexive &  $\forall s(R(s, s))$\\
     \hline
     symmetric & $\forall s\forall  t(R(s, t)\Rightarrow R(t, s))$\\
        \hline
     transitive & $\forall s\forall t(\exists u(R(s, u)\ \&\ R(u, t))\Rightarrow R(s, t))$\\
          \hline
     Euclidean & $\forall u \forall v \forall w ( R(u,v) \ \&\ R(u,w) \Rightarrow R(v,w))$\\
        \hline
\end{tabular}
\end{center}
\end{definition}
The following proposition directly derives from well known results in the Sahlqvist theory of classical normal modal logic, and has made it possible to completely characterize  various subclasses  of generalized approximation spaces in terms of  well known modal axioms.
\begin{proposition}[\cite{yao1998generalized}, Section 2] \label{prop:crisp Sahlqvist}
For any generalized approximation space $\mathbb{X}= (S,R)$, 
\begin{enumerate}
    \item $\mathbb{X} \Vdash  \Box p \rightarrow \Diamond p\, $  iff $\, \mathbb{X}$ is serial.
    \item $\mathbb{X} \Vdash \Box p \rightarrow  p\, $  iff $\, \mathbb{X}$ is reflexive.
     \item $\mathbb{X} \Vdash p \rightarrow  \Box\Diamond  p\, $  iff $\, \mathbb{X}$ is symmetric.
      \item $\mathbb{X} \Vdash \Box p \rightarrow  \Box\Box p\, $  iff $\, \mathbb{X}$ is transitive.
       \item $\mathbb{X} \Vdash \Diamond p \rightarrow  \Box \Diamond  p\, $  iff $\, \mathbb{X}$ is Euclidian.
\end{enumerate}
\end{proposition}

\subsection{The algebraic approach to Rough Set Theory} \label{ssec:algebraic review}
Several logical frameworks motivated by Rough Set Theory have been introduced based on classes of generalized Boolean algebras with operators. Their associated logics typically have a propositional base that is strictly weaker than classical propositional logic. 
\paragraph{Rough algebras and their generalizations.}Chakraborty and Banerjee \cite{banerjee1996rough} introduced  topological quasi-Boolean algebras (tqBas) as a generalization of S5 Boolean  algebras with operators (BAOs).  
\begin{definition}[\cite{banerjee1996rough}]
 A {\em topological quasi-Boolean algebra} (tqBa) is an algebra $\mathbb{T} =(\mathbb{L}, I)$ such that $\mathbb{L}=(\mathrm{L}, \vee, \wedge,\neg, \top,\bot)$ is a De Morgan algebra and for all $a,b \in \mathbb{L}$,
 \[
 T1.\, I(a \wedge b)= Ia \wedge Ib \quad T2.\, IIa=Ia \quad T3.\, Ia \leq a \quad T4.\, I \top=\top. 
 \]
\end{definition}
For any $a \in A$, let $Ca : =\neg I \neg a$. Subclasses of tqBas have been defined as in the following table.
\begin{center}
 \begin{tabular}{|c|c|c|}
\hline
     \textbf{Algebras} & \textbf{Acronyms} & \textbf{Axioms}\\
     \hline
     topological quasi Boolean algebra 5& tqBa5& T5:  $CIa=Ia$\\
          \hline
     intermediate algebra of type 1 &  IA1& T5, T6: $Ia \vee \neg Ia =\top$\\
          \hline
     intermediate algebra of type 2 &  IA2& T5, T7: $I(a \vee b) =Ia \vee Ib$\\
          \hline
     intermediate algebra of type 3 &  IA3& T5, T8: $Ia \leq Ib$ and $Ca \leq Cb$ imply $a \leq b$\\
          \hline
     pre-rough algebra & pra& T5, T6, T7, T8.\\
     \hline
\end{tabular}   
\end{center}
A {\em rough algebra} is  a complete and completely distributive pre-rough algebra. Different sub-classes of tqBas have been studied extensively in the literature on rough set theory \cite{banerjee1996rough,banerjee2004algebras}. Notice that the lower and upper approximation operators $I$ and $C$ are finitely meet-preserving and join-preserving operators on $\mathbb{L}$ respectively. Thus, these operators can serve as algebraic interpretations of normal modal operators  $\Box$ and $\Diamond$ on  $\mathbb{L}$. Thus, tqBa5s and their sub-classes can be seen as normal modal expansions of De Morgan algebras.

\paragraph{Rough algebras based on distributive lattices.}In \cite{kumar2015algebras}, Kumar and Banerjee define rough lattices as an algebraic abstraction of  algebras obtained from quasi order-generated covering-based approximation spaces (QOCAS).  These approximation spaces can be seen as the approximation spaces arising from granules seen as open neighborhoods in a topological space. 

\begin{definition}[\cite{kumar2015algebras}, Definition 4.4]\label{def:rough lattice}
 An algebra $\mathbb{L} = (\mathrm{L}, \vee, \wedge, I, C, 0,1)$ is a  {\em rough lattice}  if $(\mathrm{L}, \vee, \wedge,0,1)$ is a completely distributive lattice join-generated by its completely join-irreducible elements and $I, C: \mathrm{L} \to \mathrm{L}$ satisfy the following conditions for all $a, b\in \mathrm{L}$ and any $A\subseteq \mathrm{L}$:
 \begin{enumerate}
     \item $I (\bigwedge A) = \bigwedge\{ I(a)\mid a\in A\}$ \quad and \quad $C (\bigvee A) = \bigvee\{C(a) \mid a\in A\}$.
     \item $C (\bigwedge A) = \bigwedge\{ C(a)\mid a\in A\}$ \quad and \quad $I (\bigvee A) = \bigvee\{I(a) \mid a\in A\}$.
      
      \item $Ia \leq a$ \quad and \quad $a \leq Ca$.
      \item $I1 =1$ \quad and \quad $C0=0$.
     \item  $IIa = I a$ \quad and \quad $CC a =  Ca$.
     \item $ICa =Ca$ \quad and \quad $CI a =Ia$.
     \item $Ia \leq Ib$ and $Ca \leq Cb$ imply $a \leq b$. 
 \end{enumerate}
\end{definition}
Again, note that, as lower and upper approximation operators are completely join-preserving and meet-preserving respectively, these algebras are perfect distributive modal algebras \cite{GeNaVe05}, and hence they are naturally endowed with a (bi)Heyting algebra structure.  
Kumar \cite{kumar2020study} also introduced {\em rough Heyting algebras} to describe the interaction between  the approximation operators $I$ and $C$ and Heyting implication.

\begin{definition} [\cite{kumar2020study}, Definition 40]
  An algebra $\mathbb{L} = (L, \vee, \wedge, \rightarrow, I, C, 0,1)$ is a {\em rough Heyting algebra} iff  its $\{\rightarrow \}$-free reduct is a   rough lattice, its $\{I, C\}$-free reduct is a Heyting algebra, and the following identities are valid:
  \begin{enumerate}
      \item $I(a \rightarrow b) = (Ia \rightarrow Ib) \wedge (Ca \rightarrow Cb)$.
      \item $C(a \rightarrow b)= Ca \rightarrow Cb$. 
  \end{enumerate}
\end{definition}

\paragraph{Rough algebras based on concept lattices.} In \cite{CONRADIE2021371},  several classes of algebras are defined as varieties of modal expansions of  general (i.e.~possibly non-distributive) lattices. Since, following the approach of {\em formal concept analysis} \cite{ganter2012formal}, general lattices can be studied  as hierarchies of {\em formal concepts} arising from relational databases,  the role of modal operators can be seen as  approximating  categories or concepts under incomplete or uncertain information \cite{CONRADIE2021371}.
\begin{definition} [cf.~\cite{CONRADIE2021371}, Section 6]
  An {\em abstract conceptual rough algebra} (resp.~conceptual rough algebra) (acra (resp.~cra))   is an algebra $\mathbb{A}= (\mathbb{L}, \Box, \Diamond)$  such that $\mathbb{L}$ is a bounded lattice (resp.~complete lattice)  and $\Box$ and $\Diamond$ are unary operations on $\mathbb{L}$ such that $\Diamond$ is finitely join-preserving (resp.~completely join-preserving) and $\Box$ is finitely meet-preserving (resp.~completely meet-preserving),  and for any $a \in \mathbb{L}$,
  \[
  \Box a \leq \Diamond a.
  \]
  Subclasses of  abstract conceptual rough algebras (resp.~conceptual rough algebras) have been considered, which are reported in the following table. 
  \begin{center}
    \begin{tabular}{|c|c|}
  \hline
       \textbf{Class} & \textbf{Axioms} \\
   \hline 
    reflexive & $\Box a \leq a $ and $a \leq \Diamond a$\\
    \hline
    transitive & $\Box a \leq   \Box\Box a  $ and $\Diamond  \Diamond a \leq \Diamond a$\\
    \hline
    symmetric & $a \leq   \Box\Diamond a $ and $ \Diamond  \Box a \leq  a$\\
    \hline 
    dense & $ \Box\Box a \leq   \Box a$ and $\Diamond a \leq \Diamond \Diamond a$\\
    \hline
  \end{tabular}  
  \end{center}
  In particular, a  {\em conceptual tqBa} is a reflexive and transitive acra. Subclasses of tqBas have been defined as reported in the following table.
  \begin{center}
      \begin{tabular}{|c|c|}
      \hline
           \textbf{Class}& \textbf{Axioms}  \\
      \hline
          Conceptual tqBa5 & $\Diamond\Box a \leq \Box a $ and $ \Diamond a \leq \Box\Diamond a$\\
    \hline
    conceptual IA1 & $\Box (a \vee b) \leq \Box a \vee \Box b $ and $\Diamond a \wedge \Diamond b \leq \Diamond (a \wedge b)$\\
    \hline
    conceptual IA2 & $(\Box a \leq \Box b\ \& \ \Diamond a \leq \Diamond b)\Rightarrow a \leq b$\\
    \hline
      \end{tabular}
  \end{center}
  A {\em conceptual pre-rough algebra} is a conceptual tqBA5 which is both conceptual IA1 and IA2. 
\end{definition}
\paragraph{Rough algebras based on posets.} Cattaneo \cite{cattaneo1996mathematical} further generalizes rough set theory by defining the notion of abstract approximation spaces on posets.
\begin{definition}[\cite{cattaneo1996mathematical}, Section 1]
  An {abstract approximation space} is a structure $\mathcal{U} = (\Sigma, \mathbb{O}(\Sigma), \mathbb{C}(\Sigma) ) $ such that  $(\Sigma, \leq , 0,1)$ is a bounded poset, $\mathbb{O}(\Sigma)$ and $\mathbb{C}(\Sigma)$ are subposets of $\Sigma$ containing $0$ and $1$ such that inner and outer approximation maps $i: \Sigma \to \mathbb{O}(\Sigma)$ and  $c:\Sigma \to \mathbb{C}(\Sigma) $ exist such that, for any $x\in \Sigma$, any $\alpha\in \mathbb{O}(\Sigma)$, and any $\gamma\in \mathbb{C}(\Sigma)$,
  \begin{enumerate}
      \item $i(x) \leq x$ and $x \leq c(x)$. 
      \item  $\alpha \leq x \Rightarrow \alpha \leq i(x)$.
       \item  $x \leq \gamma \Rightarrow  c(x) \leq \gamma$.
  \end{enumerate}
\end{definition}
Algebraic structures arising from rough operators defined on various classes of lattice-based algebras such as  orthomodular lattices \cite{dai2021rough},  quasi-Brouwer-Zadeh distributive lattices \cite{cattaneo2004algebraic}, MV-algebras \cite{rasouli2010roughness}, cylindric algebras have been studied in literature. 

\paragraph{Fuzzy and probabilistic rough algebras.} Several fuzzy and probabilistic generalizations of rough set theory have been studied extensively in the literature \cite{yao2008probabilistic,liu2008axiomatic,dubois1990rough}. Yao \cite{yao2008probabilistic} introduces a framework for  probabilistic rough set approximations using the notion of rough membership. 

\begin{definition} [\cite{yao2008probabilistic}, Section 3.1]
 A {\em probabilistic approximation space} is a structure $\mathbb{X} = (U,E,P)$ such that $E \subseteq U \times U$ is an equivalence relation and $P: 2^U \to [0,1]$ is a probability function. The  {\em rough membership} in any $A \subseteq U$  is given by the map $\mu_A: U \to [0,1]$ defined by the assignment $x\mapsto P(A \mid [x])$, 
 where $[x]\coloneqq \{ y \in U \mid x E y\} $. 
\end{definition}

\begin{definition}[\cite{yao2008probabilistic}, Section 4.3]
  For all parameters  $\alpha, \beta \in [0,1]$, and any  probabilistic approximation space $\mathbb{X} = (U,E,P)$, the generalized probabilistic approximation operators $\underline{appr_\alpha},\overline{appr_\beta}: \mathcal{P}(U)\to \mathcal{P}(U)$  are defined as follows.  If $0 \leq \beta < \alpha  \leq 1$, then for any $A\subseteq U$,
\[\underline{appr_\alpha}(A) = \{ x \in U \mid P(A\mid [x]) \geq \alpha \}\quad \text{and} \quad \overline{appr_\beta} (A)= \{ x \in U \mid P(A\mid [x]) < \beta\}.\]
If $\alpha=\beta\neq 0$, then 
\[\underline{appr_\alpha}(A) = \{ x \in U \mid P(A\mid [x]) > \alpha\} \quad \text{and} \quad \overline{appr_\alpha}(A) = \{ x \in U \mid P(A\mid [x]) \leq  \alpha\}.\]
\end{definition}

 Given the properties they were shown to satisfy (e.g.~monotonicity, distributivity over union and intersection), the  approximation operators  $\underline{appr_\alpha}$ and $\overline{appr_\beta}$ can be regarded as paramentric   fuzzy regular modal operators $\Box_\alpha$ and $\Diamond_\beta$ on the powerset algebra $\mathcal{P}(U)$, respectively. 

Using both constructive and algebraic methods, Wu and Zhang \cite{wu2004constructive} introduce a generalization of rough set theory to a fuzzy framework. 
\begin{definition}[\cite{wu2004constructive}, Definition 4]
  A generalized {\em fuzzy approximation space}  is a structure $\mathbb{X} =(W,U,R)$ such that $U$ and $W$ are  non-empty finite sets, and $R: U \times W \to [0,1]$ is a fuzzy relation. For any fuzzy subset $A$ of $W$, the lower and upper approximations of $A$, $\underline{R}(A)$ and $\overline{R}(A)$, are defined as follows: For any $x \in U$, 
  \[
  \overline{R}(A)(x) = \bigvee_{y \in W} (R(x,y) \wedge A(y)) \quad \text{and} \quad \underline{R}(A)(x) = \bigwedge_{y \in W} ((1-R(x,y)) \vee A(y)).
  \]
  A generalized fuzzy approximation space $\mathbb{X} =(W,U,R)$  with 
  $U=W$ is: 
  \begin{enumerate}
      \item {\em reflexive} if $R(x,x)=1$ for all $x \in U$.
      \item {\em symmetric} if $R(x,y)= R(y,z)$ for all $x,y \in U$. 
      \item {\em transitive} if $R(x,z) \geq \bigvee_{y \in U} (R(x,y) \wedge R(y,x))$ for all $x,z \in U$. 
      \item {\em serial} if $\exists y (R(x,y)=1)$ for all $x \in U$. 
  \end{enumerate}
\end{definition}
Let $\mathcal{F}(W)$ denote the set of all fuzzy subsets of $W$. 
The fuzzy approximation operators $\underline{R}$ and  $\overline{R}$ defined above satisfy the following properties \cite{wu2004constructive}: for all $A, B \in \mathcal{F}(W)$, and  any $\alpha \in [0,1]$, 
\begin{center}
    \begin{tabular}{lcl}
   1.$\underline{R}(A)= \neg \overline{R} (\neg A) $ & \quad &  2.$\overline{R}(A)= \neg \underline{R} (\neg A) $\\
   3.$ \underline{R}(A \vee \hat{\alpha})=\underline{R}(A) \vee \hat{\alpha}  $ &\quad &4.$ \overline{R}(A \wedge \hat{\alpha})=\overline{R}(A) \wedge \hat{\alpha}  $  \\
   5.$ \underline{R}(A \wedge B)=\underline{R}(A) \wedge \underline{R}(B)$&\quad &6.$ \overline{R}(A \vee B)=\overline{R}(A) \vee \overline{R}(B)$, \\
\end{tabular}
\end{center}
where $\hat{\alpha}$ is the constant fuzzy set $\hat{\alpha}(x)= \alpha $ for all $x \in U \cup W$. As lower and upper approximation operators $\underline{R}$ and $\overline{R}$ are completely meet-preserving and join-preserving, these operators are semantic modal operators $\Box$ and $\Diamond$ on $\mathcal{F}(W)$. 

The following proposition generalizes Proposition \ref{prop:crisp Sahlqvist} to the fuzzy setting.
The inclusion symbol between fuzzy sets in the proposition denotes as usual that the function on the left-hand side of the inclusion is pointwise dominated by the one on the right-hand side. We let $A$  denote an element of $\mathcal{F}(W)$, $x,y$  denote elements of $W$, and $\alpha$  denote a constant in $[0,1]$. For any $Y \subseteq W$, we let $1_Y$  denote the fuzzy set which assigns membership value $1$ to all the elements of $Y$ and $0$ to all the elements not in $Y$.
\begin{proposition}[\cite{wu2004constructive},  Section 3.3] \label{prop:correspondence Kripke}
A generalized fuzzy approximation space $\mathbb{X} =(W,W,R)$ is:
\begin{enumerate}
    \item reflexive iff $A \subseteq \overline{R}(A)$ or equivalently iff $\underline{R}(A) \subseteq A$ for every $A\in \mathcal{F}(W)$. 
    \item symmetric iff for all $x, y\in W$ \[\underline{R}(1_{U-\{x\}})(y)= \underline{R}(1_{U-\{y\}})(x) \text{ or equivalently iff  } \overline{R}(1_x)(y)= \overline{R}(1_y)(x)  .\]
     \item transitive iff $\underline{R}(A) \subseteq \underline{R}(\underline{R}(A))$ or equivalently iff $\overline{R}(\overline{R}(A))\subseteq \overline{R}(A) $ for every $A\in \mathcal{F}(W)$. 
     \item serial  iff $ \underline{R}(\varnothing) = \varnothing $ iff $\overline{R}(W)= W$ iff $\forall \alpha (\overline{R}(\hat{\alpha})= \hat{\alpha}) $ iff  $\forall \alpha (\underline{R}(\hat{\alpha})= \hat{\alpha}) $ iff $\forall  A  (\underline{R}(A) \subseteq \overline{R}(A)) $. 
\end{enumerate}
\end{proposition}
Several types of  many-valued algebras generalizing the interval $[0,1]$ have been studied in the literature as an environment for fuzzy rough set theory \cite{wu2004constructive,CONRADIE2021371,sun2008fuzzy,radzikowska2004fuzzy,chakraborty2011fuzzy}. 

\subsection{Basic normal non-distributive modal logic and its graph-based semantics}\label{ssec:logics}
The present section collects basic definitions and facts from \cite{GeHa01,CoPa:non-dist,Conradie2020NondistributiveLF}.
Let $\Prop$ be a (countable or finite) set of atomic propositions. The language $\mathcal{L}$ of the {\em basic normal non-distributive modal logic} is defined as follows:
\begin{center}
  $\varphi := \bot \mid \top \mid p \mid  \varphi \wedge \varphi \mid \varphi \vee \varphi \mid \Box \varphi \mid  \Diamond\varphi$,  
\end{center}
where $p\in \Prop$. 
The {\em basic}, or {\em minimal normal} $\mathcal{L}$-{\em logic} is a set $\mathbf{L}$ of sequents $\phi\vdash\psi$  with $\phi,\psi\in\mathcal{L}$, containing the following axioms:

{{\centering
\begin{tabular}{ccccccccccccc}
     $p \vdash p$ & \quad\quad & $\bot \vdash p$ & \quad\quad & $p \vdash p \vee q$ & \quad\quad & $p \wedge q \vdash p$ & \quad\quad & $\top \vdash \Box\top$ & \quad\quad & $\Box p \wedge \Box q \vdash \Box(p \wedge q)$
     \\
     & \quad & $p \vdash \top$ & \quad & $q \vdash p \vee q$ & \quad & $p \wedge q \vdash q$ &\quad &  $\Diamond\bot \vdash \bot$ & \quad & $\Diamond(p \vee q) \vdash \Diamond p \vee \Diamond q$
\end{tabular}
\par}}

		and closed under the following inference rules:
		{\small{
		\begin{displaymath}
			\frac{\phi\vdash \chi\quad \chi\vdash \psi}{\phi\vdash \psi}
			\quad
			\frac{\phi\vdash \psi}{\phi\left(\chi/p\right)\vdash\psi\left(\chi/p\right)}
			\quad
			\frac{\chi\vdash\phi\quad \chi\vdash\psi}{\chi\vdash \phi\wedge\psi}
			\quad
			\frac{\phi\vdash\chi\quad \psi\vdash\chi}{\phi\vee\psi\vdash\chi}
			\quad
			\frac{\phi\vdash\psi}{\Box \phi\vdash \Box \psi}
\quad
\frac{\phi\vdash\psi}{\Diamond \phi\vdash \Diamond \psi}
		\end{displaymath}
		}}
By an {\em $\mathcal{L}$-logic} we understand any  extension of $\mathbf{L}$  with $\mathcal{L}$-axioms $\phi\vdash\psi$.

\medskip

Graph-based models for non-distributive logics arise in close connection with the topological structures dual to general lattices in Plo\v{s}\v{c}ica's representation \cite{ploscica1994} (see also \cite{craig-priestley,craig2015tirs}). 

A {\em reflexive graph} is a structure $\mathbb{G} = (Z, E)$ such that $Z$ is a nonempty set, and $E\subseteq Z\times Z$ is a reflexive relation.\footnote{In the context of the generalization of rough set theory pursued in the present paper (cf.~Section \ref{sec:Rough set theory on graph-based frames}), the intuitive reading of $zEz'$ is `$z$ is inherently indiscernible from $z'$' for all $z, z'\in Z$.} We let $D\subseteq Z\times Z$ be defined as $xDa$ iff $aEx$. For every set $S$, we let $\mathbb{G}_S=(S,\Delta)$.   From now on, we assume that all graphs we consider are reflexive even when we drop the adjective.  
Any graph $\mathbb{G} = (Z, E)$  defines the polarity\footnote{ A {\em formal context} \cite{ganter2012formal}, or {\em polarity},  is a structure $\mathbb{P} = (A, X, I)$ such that $A$ and $X$ are sets, and $I\subseteq A\times X$ is a binary relation. Every such $\mathbb{P}$ induces maps $(\cdot)^\uparrow: \mathcal{P}(A)\to \mathcal{P}(X)$ and $(\cdot)^\downarrow: \mathcal{P}(X)\to \mathcal{P}(A)$, respectively defined by the assignments $B^\uparrow: = I^{(1)}[B]$ and $Y^\downarrow: = I^{(0)}[Y]$. A {\em formal concept} of $\mathbb{P}$ is a pair $c = (B, Y)$ such that $B\subseteq A$, $Y\subseteq X$, and $B^{\uparrow} = Y$ and $Y^{\downarrow} = B$. Given a formal concept $c = (B,Y)$ we often write $\val{c}$ for $B$ and $\descr{c}$ for $Y$ and, consequently, $c = (\val{c}, \descr{c})$.  The set $L(\mathbb{P})$  of the formal concepts of $\mathbb{P}$ can be partially ordered as follows: for any $c, d \in L(\mathbb{P})$, \[c\leq d\quad \mbox{ iff }\quad \val{c}\subseteq \val{d} \quad \mbox{ iff }\quad \descr{d}\subseteq \descr{c}.\]
	With this order, $L(\mathbb{P})$ is a complete lattice, the {\em concept lattice} $\mathbb{P}^+$ of $\mathbb{P}$. Any complete lattice $\mathbb{L}$ is isomorphic to the concept lattice $\mathbb{P}^+$ of some polarity $\mathbb{P}$.} $\mathbb{P_G} = (Z_A,Z_X, I_{E^{c}})$ where $Z_A = Z = Z_X$ and  $I_{E^{c}}\subseteq Z_A\times Z_X$ is defined as $aI_{E^{c}} x$ iff $aE^cx$.  More generally, any relation $R\subseteq Z\times Z$ `lifts' to relations $I_{R^c}\subseteq Z_A\times Z_X$ and $J_{R^c}\subseteq Z_X\times Z_A$
defined as $aI_{R^c} x$ iff $a R^c x$ and $xJ_{R^c} a$ iff $x R^c a$. In what follows, for any $W\subseteq Z$, we will write $(W)_A\subseteq Z_A$ and $(W)_X\subseteq Z_X$ as its corresponding `lifted' copies.  
The next lemma follows directly from the definitions above: 
\begin{lemma}
	\label{lemma:round and square brackets}
	For any relation $R\subseteq Z \times Z$ and any $Y, B \subseteq Z$,
	\[I^{(0)}_{R^c}[Y] = (R^{[0]}[Y])_A\quad  I^{(1)}_{R^c}[B] = (R^{[1]}[B])_X\quad J^{(0)}_{R^c}[B] = (R^{[0]}[B])_X\quad  J^{(1)}_{R^c}[Y] = (R^{[1]}[Y])_A.\]
\end{lemma}
The complete lattice $\mathbb{G}^{+}$ associated with a graph $\mathbb{G}$ is the concept lattice of $\mathbb{P_G}$.
By specializing  \cite[Proposition 3.1]{Conradie2020NondistributiveLF} to  $\mathbb{P_G}$ we get the following.
\begin{proposition}
\label{prop:generators}
For every graph $\mathbb{G} =(Z,E)$, the  lattice $\mathbb{G}^{+}$ is completely join-generated  by the set $\{\mathbf{z}_A= ( z^{[10]}, z^{[1]})\mid z \in Z\}$ and completely meet-generated by the set $\{ \mathbf{z}_X= ( z^{[0]}, z^{[01]}) \mid z \in Z\}$.
\end{proposition}
%
%

\begin{definition}\label{def:graph:based:frame:and:model}
A {\em graph-based} $\mathcal{L}$-{\em frame}  is a structure $\mathbb{F} = (\mathbb{G},  R_{\Diamond}, R_{\Box})$ where $\mathbb{G} = (Z,E)$ is a reflexive graph,\footnote{\label{footnote: abbreviations} Applying  notation \eqref{eq:def round square brackets} to a graph-based $\mathcal{L}$-frame $\mathbb{F}$, we sometimes abbreviate $E^{[0]}[Y]$ and $E^{[1]}[B]$ as $Y^{[0]}$ and $B^{[1]}$, respectively, for all $Y, B\subseteq Z$. If $Y= \{y\}$ and  $B = \{b\}$, we write $y^{[0]}$ and $b^{[1]}$ for $\{y\}^{[0]}$ and $\{b\}^{[1]}$, and write $Y^{[01]}$ and $B^{[10]}$ for $(Y^{[0]})^{[1]}$ and $(B^{[1]})^{[0]}$, respectively. Lemma \ref{lemma:round and square brackets} implies that $Y^{[0]} = I_{E^c}^{(0)}[Y] = Y^{\downarrow}$ and $B^{[1]} = I_{E^c}^{(1)}[B] = B^{\uparrow}$, where $(\cdot)^\downarrow$ and $(\cdot)^\uparrow$ are the maps associated with $\mathbb{P_G}$.} and  $R_{\Diamond}$ and $R_{\Box}$  are binary relations on $Z$ satisfying the following  $E$-{\em compatibility} conditions (notation defined in \eqref{eq:def round square brackets}): for all $b,y \in Z$,
\begin{align*}
(R_\Box^{[0]}[y])^{[10]} &\subseteq R_\Box^{[0]}[y] \qquad   &(R_\Box^{[1]}[b])^{[01]} \subseteq R_\Box^{[1]}[b]\\
(R_\Diamond^{[0]}[b])^{[01]} &\subseteq R_\Diamond^{[0]}[b] \qquad   &(R_\Diamond^{[1]}[y])^{[10]} \subseteq R_\Diamond^{[1]}[y].
\end{align*}	
	
	The {\em complex algebra} of a graph-based $\mathcal{L}$-frame $\mathbb{F}= (\mathbb{G},  R_{\Diamond}, R_{\Box})$ is the complete $\mathcal{L}$-algebra $\mathbb{F}^+ = (\mathbb{G}^+, [R_\Box], \langle R_\Diamond\rangle),$
	where $\mathbb{G}^+$ is the concept lattice of 
	$\mathbb{P}_{\mathbb{G}}$,
	and $[R_\Box]$ and $\langle R_\Diamond\rangle$ are unary operations on 
	$\mathbb{P}_{\mathbb{G}}^+$
	defined as follows: for every $c = (\val{c}, \descr{c}) \in 
	\mathbb{P}_{\mathbb{G}}^+$,
		\[[R_\Box]c: = (R_{\Box}^{[0]}[\descr{c}], (R_{\Box}^{[0]}[\descr{c}])^{[1]}) \quad \mbox{ and }\quad \langle R_\Diamond\rangle c: = ((R_{\Diamond}^{[0]}[\val{c}])^{[0]}, R_{\Diamond}^{[0]}[\val{c}]).\] 
		\vspace{-0.5cm}
\end{definition}

\begin{lemma}\label{equivalents of I-compatible} (\cite{CONRADIE2021371}, Lemma 2.6)
For every graph $(Z,E)$ and every relation $R\subseteq Z\times Z$,
	\begin{enumerate}
		\item the following are equivalent:
		\begin{enumerate}
			\item [(i)] $(R^{[0]}[y])^{[10]} \subseteq R^{[0]}[y]$  for every $y\in Z$;
			\item [(ii)]  $(R^{[0]}[Y])^{[10]} \subseteq R^{[0]}[Y]$ for every $Y\subseteq Z$;
			\item [(iii)] $R^{[1]}[B]=R^{[1]}[B^{[10]}]$ for every  $B\subseteq Z$;
		\end{enumerate}
		\item the following are equivalent:
		\begin{enumerate}
			\item [(i)] $(R^{[1]}[b])^{[01]} \subseteq R^{[1]}[b]$  for every $b\in Z$;
			\item [(ii)]  $(R^{[1]}[B])^{[01]} \subseteq R^{[1]}[B]$ for every $B\subseteq Z$;
			\item [(iii)] $R^{[0]}[Y]=R^{[0]}[Y^{[01]}]$ for every  $Y\subseteq Z$.
		\end{enumerate}
	\end{enumerate}
\end{lemma}
For any graph-based $\mathcal{L}$-frame $\mathbb{F}$, we let $R_{\Diamondblack}\subseteq Z\times Z$ be defined  by $x R_{\Diamondblack} a$ iff $aR_{\Box} x$, and $R_{\blacksquare}\subseteq Z\times Z$ by $a R_{\blacksquare} x$ iff $xR_{\Diamond} a$. Hence, for every $B, Y\subseteq Z$,
\begin{equation}
\label{eq:zero is 1 is zero}
R_{\Diamondblack}^{[0]}[B] = R_{\Box}^{[1]}[B] \quad R_{\Diamondblack}^{[1]}[Y] = R_{\Box}^{[0]}[Y] \quad R_{\blacksquare}^{[0]}[Y] = R_{\Diamond}^{[1]}[Y] \quad R_{\blacksquare}^{[1]}[B] = R_{\Diamond}^{[0]}[B].
\end{equation}
By Lemma \ref{equivalents of I-compatible}, the $E$-compatibility of $R_{\Box}$ and $ R_{\Diamond}$ guarantees that  the operations $[R_\Box], \langle R_\Diamond\rangle$ (as well as $[R_\blacksquare], \langle R_\Diamondblack\rangle$) are well defined on 
$\mathbb{G}^+$. 

 
\begin{lemma}[\cite{graph-based-wollic}, Lemma 5]
\label{lemma:existence adjoints}
Let \,  $\mathbb{F} = (\mathbb{G}, R_{\Box}, R_{\Diamond})$ be a graph-based $\mathcal{L}$-frame. Then  
	the algebra\, $\mathbb{F}^+ = (\mathbb{G}^+, [R_{\Box}], \langle R_{\Diamond}\rangle)$ is a complete  lattice expansion such that $ \langle R_\Diamondblack\rangle$ and $[R_\Box]$ (resp.~$\langle R_\Diamond\rangle$ and $[R_\blacksquare]$) form an adjoint pair.\footnote{For any posets $A$ and $B$, the maps $f:A \to B$ and $g: B \to A$ form an {\em adjoint pair} (notation: $f\dashv g$) if for every $a \in A$, and every $b\in B$, $f(a) \leq b \mbox{ iff } a \leq g(b)$. The map $f$ (resp.\ $g$) is the unique {\em left} (resp.\ {\em right}) adjoint of $g$ (resp.\ $f$). It is well known (cf.~\cite{davey2002introduction}) that in a complete lattice a map is completely join preserving (resp.\ meet preserving) iff it is a left (resp.\ right) adjoint of some map. }  Therefore,  $[R_\Box]$ (resp.~$ \langle R_\Diamond\rangle$) is completely meet-preserving (resp.~completely join-preserving). 
\end{lemma}
\begin{definition}
\label{def:fxf_xfx}
For every Kripke frame $\mathbb{X} = (S,R)$, the relational structure $\mathbb{F_\mathbb{X}} \coloneqq (\mathbb{G}_S,R_\Diamond,R_\Box)$, where $R_\Box = R_\Diamond = R$, is a graph-based frame, since any relation $R$ on $S$ is clearly $\Delta$-compatible (cf.~Definition \ref{def:graph:based:frame:and:model}). Conversely, for any graph-based frame $\mathbb{F} = ((Z, E), R_\Diamond, R_\Box)$ such that $E = \Delta$, and $R_\Diamond=R_\Box$, let $\mathbb{X_F} \coloneqq (Z, R_\Box)$ be its associated Kripke frame.
\end{definition}

For any Kripke frame $\mathbb{X}$, let $\mathbb{X^+}$ be its associated complex algebra. 
\begin{proposition} [\cite{CONRADIE2021371}, Proposition 3.7]
\label{prop:from Kripke frames to enriched polarities}
  If $\mathbb{X}$ is a Kripke frame, then  $\mathbb{F}_{\mathbb{X}}^+ \cong \mathbb{X}^+$.
\end{proposition}

\begin{proposition}
\label{prop:fxf_xfx}
For any graph-based frame $\mathbb{F}$, and any Kripke frame $\mathbb{X}$, 
\smallskip

{{\centering
$\mathbb{F_{X_F}} \cong \mathbb{F}$ \quad\quad and \quad\quad  $\mathbb{X_{F_X}} \cong \mathbb{X}$.
\par}}
\end{proposition}

\begin{definition} [\cite{graph-based-wollic}, Definition 4]
	A {\em graph-based} $\mathcal{L}$-{\em model} is a tuple $\mathbb{M} = (\mathbb{F}, V)$ where $\mathbb{F}=(\mathbb{G},R_\Diamond,R_\Box)$ is a graph-based $\mathcal{L}$-frame and $V: \mathsf{Prop}\to \mathbb{F}^+$. Since $V(p)$ is a formal concept (cf.~Section \ref{ssec:logics})  of  $\mathbb{P}_\mathbb{G}$, we write $V(p) = (\val{p}, \descr{p})$.
	For every such $\mathbb{M}$,  the valuation $V$ can be extended compositionally to all $\mathcal{L}$-formulas as follows:
	\vspace{-2mm}
{\small{
		\begin{center}
			\begin{tabular}{r c l c r cl}
				$V(p)$ & $ = $ & $(\val{p}, \descr{p})$\\
				$V(\top)$ & $ = $ & $(Z, \varnothing)$ &\quad& $V(\bot)$ & $ = $ & $(\varnothing, Z)$\\
				$V(\phi\wedge\psi)$ & $ = $ & $(\val{\phi}\cap \val{\psi}, (\val{\phi}\cap \val{\psi})^{[1]})$ &&
				$V(\phi\vee\psi)$ & $ = $ & $((\descr{\phi}\cap \descr{\psi})^{[0]}, \descr{\phi}\cap \descr{\psi})$\\
				$V(\Box\phi)$ & $ = $ & $(R_{\Box}^{[0]}[\descr{\phi}], (R_{\Box}^{[0]}[\descr{\phi}])^{[1]})$ &&
				$V(\Diamond\phi)$ & $ = $ & $((R_{\Diamond}^{[0]}[\val{\phi}])^{[0]}, R_{\Diamond}^{[0]}[\val{\phi}]).$\\
			\end{tabular}
		\end{center}
}}
Moreover, the existence of the adjoints of $[R_{\Box}]$ and $\langle R_\Diamond \rangle$ (cf.~Lemma \ref{lemma:existence adjoints}) supports the interpretation of the  language $\mathcal{L}^*$, which is an expansion of $\mathcal{L}$ with the  connectives $\blacksquare$ and $\Diamondblack$ interpreted as follows:
{\small{
		\begin{center}
			\begin{tabular}{r c l c r c l}
				$V(\blacksquare\phi)$ & $ = $ & $(R_{\blacksquare}^{[0]}[\descr{\phi}], (R_{\blacksquare}^{[0]}[\descr{\phi}])^{[1]})$&$\quad$&
				$V(\Diamondblack\phi)$ & $ = $ & $((R_{\Diamondblack}^{[0]}[\val{\phi}])^{[0]}, R_{\Diamondblack}^{[0]}[\val{\phi}])$\\
				%
			\end{tabular}
\end{center}}}
\end{definition}
Spelling out the definition above as is done in \cite{conradie2015relational}, we can define the satisfaction and refutation relations $\mathbb{M}, \mathbf{z} \Vdash \phi$ and $\mathbb{M}, \mathbf{z} \succ \phi$
for every graph-based $\mathcal{L}$-model $\mathbb{M} = (\mathbb{F}, V)$, $\mathbf{z} \in Z$, and any $\mathcal{L}$-formula $\phi$,  by the following simultaneous recursion:
\smallskip

{{\centering
			\begin{tabular}{lllllll}
				$\mathbb{M}, \mathbf{z} \Vdash \bot$ &&never & \quad\quad\quad &$\mathbb{M}, \mathbf{z} \succ \bot$ && always\\
				$\mathbb{M}, \mathbf{z} \Vdash \top$ &&always & &$\mathbb{M}, \mathbf{z} \succ \top$ &&never\\
				$\mathbb{M}, \mathbf{z} \Vdash p$ & iff & $\mathbf{z}\in \val{p}$ & &$\mathbb{M}, \mathbf{z} \succ p$ & iff & $\forall \mathbf{z}'[\mathbf{z}'E\mathbf{z} \Rightarrow \mathbf{z}' \not\Vdash p]$\\
				%
				%
				%
				$\mathbb{M}, \mathbf{z} \succ \phi \vee \psi$ &iff &$\mathbb{M},\mathbf{z}\succ \phi \text{ and } \mathbb{M},\mathbf{z} \succ \psi$%
				& &$\mathbb{M}, \mathbf{z} \Vdash \phi \vee \psi$ &iff &$\forall \mathbf{z}' [\mathbf{z} E\mathbf{z}'   \Rightarrow \mathbb{M},\mathbf{z}' \not\succ \phi \vee \psi]$\\
				$\mathbb{M}, \mathbf{z} \Vdash \phi \wedge \psi$ &iff &$\mathbb{M},\mathbf{z}\Vdash \phi \text{ and } \mathbb{M},\mathbf{z} \Vdash \psi$%
				& &$\mathbb{M},\mathbf{z} \succ \phi \wedge\psi$ &iff &$\forall \mathbf{z}'[\mathbf{z}' E\mathbf{z} \Rightarrow \mathbb{M},\mathbf{z}'\not\Vdash \phi \wedge \psi]$\\
				$\mathbb{M}, \mathbf{z} \succ \Diamond\phi$ &iff &$\forall \mathbf{z}' [\mathbf{z} R_{\Diamond}\mathbf{z}'  \Rightarrow \mathbb{M},\mathbf{z}' \not \Vdash \phi]$%
				& &$\mathbb{M}, \mathbf{z} \Vdash \Diamond\phi$ &iff &$\forall \mathbf{z}'[\mathbf{z} E\mathbf{z}' \Rightarrow \mathbb{M},\mathbf{z}' \not\succ \Diamond \phi]$\\
				$\mathbb{M}, \mathbf{z} \Vdash \Box \psi$ &iff &$\forall \mathbf{z}' [\mathbf{z} R_{\Box}\mathbf{z}'  \Rightarrow \mathbb{M},\mathbf{z}' \not\succ \psi]$%
				& &$\mathbb{M}, \mathbf{z} \succ \Box\psi$ &iff &$\forall \mathbf{z}' [\mathbf{z}' E\mathbf{z}  \Rightarrow \mathbb{M},\mathbf{z}' \not\Vdash \Box \psi]$\\
				$\mathbb{M}, \mathbf{z} \succ \Diamondblack\phi$ &iff &$\forall \mathbf{z}' [\mathbf{z} R_{\Diamondblack}\mathbf{z}'  \Rightarrow \mathbb{M},\mathbf{z}' \not \Vdash \phi]$%
				& &$\mathbb{M}, \mathbf{z} \Vdash \Diamondblack\phi$ &iff &$\forall \mathbf{z}'[\mathbf{z} E\mathbf{z}' \Rightarrow \mathbb{M},\mathbf{z}' \not\succ \Diamondblack \phi]$\\
				$\mathbb{M}, \mathbf{z} \Vdash \blacksquare \psi$ &iff &$\forall \mathbf{z}' [\mathbf{z} R_{\blacksquare}\mathbf{z}'  \Rightarrow \mathbb{M},\mathbf{z}' \not\succ \psi]$%
				& &$\mathbb{M}, \mathbf{z} \succ \blacksquare\psi$ &iff &$\forall \mathbf{z}' [\mathbf{z}' E\mathbf{z}  \Rightarrow \mathbb{M},\mathbf{z}' \not\Vdash \blacksquare \psi]$\\
				%
				%
			\end{tabular}
			\par
}}
\smallskip

Note that the satisfaction of disjuction formulas at a state  is defined as  the non-refutation of the disjuncts at all its $E$-successors, which is different from the satisfaction clause in distributive logics. Also, unlike  distributive modal logic, the satisfiability clauses for all  modal operators including the diamonds are given by universal quantifiers. 

An $\mathcal{L}$-sequent $\phi \vdash \psi$ is {\em true} in $\mathbb{M}$, denoted $\mathbb{M} \models \phi \vdash \psi$, if for all $\mathbf{z}, \mathbf{z}' \in Z$, if $\mathbb{M}, \mathbf{z} \Vdash \phi$ and $\mathbb{M}, \mathbf{z}' \succ \psi$ then $\mathbf{z} E^c \mathbf{z}'$. An $\mathcal{L}$-sequent $\phi \vdash \psi$ is {\em valid} in $\mathbb{F}$, denoted $\mathbb{F} \models \phi \vdash \psi$, if it is true in every model based on $\mathbb{F}$.

The next lemma follows immediately from the definition of an $\mathcal{L}$-sequent being true in a graph-based $\mathcal{L}$-model. 
\begin{lemma} [\cite{graph-based-wollic}, Lemma 7]
	\label{lemma:for:completeness}
	For any graph-based $\mathcal{L}$-frame $\mathbb{F}$  and any $\mathcal{L}$-sequent $\phi\vdash \psi$,
 \begin{center}$\mathbb{F}\models \phi\vdash \psi$ \, iff \,  $\mathbb{F}^+ \models \phi\vdash \psi$. \end{center}
\end{lemma}

\subsection{Inductive modal reduction principles and their ALBA-outputs}
\label{ssec:mrp}
A {\em modal reduction principle} (mrp) of the basic language of non-distributive modal logic $\mathcal{L}$ is an inequality $s(p)\leq t(p)$ such that both $s$ and $t$ are built up using only $\Box$ and $\Diamond$.
The term $s(p)$ (resp.\ $t(p)$) is {\em good} if it is a (possibly empty) sequence of diamonds (resp.\ boxes) followed by a (possibly empty) sequence of boxes (resp.\ diamonds).  A modal reduction principle is {\em inductive} iff it is Sahlqvist iff either $s(p)$ or $t(p)$ is good, and is {\em analytic inductive} if both are good.
So, inductive modal reduction principles are of one of the following types: 
\[
\text{(a)} \ \ \ \phi [\alpha (p)/!y]\leq \psi [\chi (p)/!x] \quad \quad \quad \text{(b)} \ \ \ \phi [\zeta (p)/!y]\leq \psi [\delta (p)/!x],
\] 
where $\Box$ (resp.\ $\Diamond$) is the only connective allowed in $\alpha(p)$ and $\psi(!x)$ (resp.\ $\varphi(!y)$ and $\delta(p)$), and $\chi$ (resp.~$\zeta$)   has $\Diamond$ (resp.~$\Box$) as principal connective whenever $\chi(p)\neq p$ (resp.~$\zeta(p)\neq p$).
{\em Analytic inductive}  mrps are of the form $\phi [\alpha (p)/!y]\leq \psi [\delta (p)/!x]$, where $\psi (!x)$, $\phi(!y)$, $\delta(p)$ and $\alpha (p)$ are as above, and thus they are of both types (a) and (b).
Prime examples of analytic inductive mrps feature prominently in several proposed modal axiomatizations of generalized approximation spaces (cf.~Sections \ref{ssec:approx spaces} and \ref{ssec:algebraic review}).
The ALBA outputs of inductive mrps  are the following  inequalities  in the expanded language $\mathcal{L}^+$ (cf.\ \cite{CoPa:non-dist}):\footnote{
The language $\mathcal{L}^+$ expands $\mathcal{L}^*$ with two denumerable sets of sorted variables $\mathsf{NOM} \ni \nomi, \nomj$ and $\mathsf{CONOM} \ni \cnomm, \cnomn$. The reader is referred to \cite{conradie2016constructive,CoPa:non-dist} for a general presentation of ALBA.
} 
\begin{center}
\begin{tabular}{rl}
type (a) & $\quad\forall \nomj[\mathsf{LA}(\psi)[\phi[\nomj/!y]/!u] \leq \chi[\mathsf{LA}(\alpha)[\nomj/!u]/p]]$, \\
type (b) & $\quad\forall \cnomm[\zeta[\mathsf{RA}(\delta)[\cnomm/!u]/p] \leq \mathsf{RA}(\phi)[\psi[\cnomm/!y]/!u]$,
\end{tabular}
\end{center}
where, when $\theta$ is a sequence of $\Box$ (resp.\ $\Diamond$) connectives of length $n$, $\mathsf{LA}(\theta)(!u)$ (resp.\ $\mathsf{RA}(\theta)(!u)$) is a sequence of $\Diamondblack$ (resp.\  $\blacksquare$) connectives of length $n$ ending with a (single) occurrence of the variable $u$.

\begin{example}
The mrp $\Diamond\Diamond\Box\Box p \leq \Box\Diamond\Box p$ is inductive of type (a), but not analytic, since it instantiates type (a) above as follows: $\varphi(y) \coloneqq \Diamond\Diamond y$, $\alpha(p) \coloneqq \Box\Box p$, $\psi(x) \coloneqq \Box x$, $\chi(p) \coloneqq \Diamond\Box p$.
Thus, $\mathsf{LA}(\psi)(u) = \Diamondblack u$, $\mathsf{LA}(\alpha)(u) = \Diamondblack\Diamondblack u$; therefore, its ALBA output is
\smallskip

{{\centering
\begin{tabular}{rl}
& $\forall\nomj \; [(\Diamondblack u) [(\Diamond\Diamond y)[\nomj / y] /u]  \leq (\Diamond\Box p)[ (\Diamondblack\Diamondblack u)[\nomj / u] / p]]$, \\
i.e. & $\forall\nomj \; [(\Diamondblack u) [(\Diamond\Diamond \nomj) /u]  \leq (\Diamond\Box p)[ (\Diamondblack\Diamondblack \nomj) / p]]$, \\
i.e. & $\forall\nomj \; [\Diamondblack \Diamond\Diamond \nomj \leq \Diamond\Box \Diamondblack\Diamondblack \nomj]$.
\end{tabular}
\par}}
\smallskip

The mrp $\Diamond\Box\Box p \leq \Box\Box\Box\Diamond\Diamond\Diamond\Diamond p$ is analytic inductive, as it instantiates the analytic shape above as follows:
$\varphi(y) \coloneqq \Diamond y$, $\alpha(p) \coloneqq \Box\Box p$, $\psi(y) \coloneqq \Box\Box\Box y$, $\delta(p) \coloneqq \Diamond\Diamond\Diamond\Diamond p$. Thus, $\mathsf{RA}(\delta)(u) = \blacksquare\blacksquare\blacksquare\blacksquare u$, and $\mathsf{RA}(\varphi)(u) \coloneqq \blacksquare u$.
Since this mrp is analytic, it can be solved both as a type (a) and as a type (b); solving it as type (b) yields
\smallskip

{{\centering
\begin{tabular}{rl}
& $\forall\cnomm \; [(\Box\Box y)[(\blacksquare\blacksquare\blacksquare\blacksquare u)[\cnomm / u] / p] \leq (\blacksquare u) [(\Box\Box\Box y)[\cnomm / y] /u] ]$, \\
i.e. & $\forall\cnomm \; [(\Box\Box y)[(\blacksquare\blacksquare\blacksquare\blacksquare \cnomm) / p] \leq (\blacksquare u) [(\Box\Box\Box \cnomm) /u] ]$, \\
i.e. & $\forall\cnomm \; [\Box\Box\blacksquare\blacksquare\blacksquare\blacksquare \cnomm\leq \blacksquare \Box\Box\Box \cnomm]$. \\

\end{tabular}
\par}}
\end{example}

\section{  Two case studies on informational entropy}
\label{ssec:twocases studies}

In the present section, we discuss two examples in which  informational entropy is relevant to the analysis, and is explicitly accounted for in the environment of graph-based frames.

\subsection{Social media news network}
Let the reflexive graph $(Z, E)$ represent a social media news network, where $Z$ is a set of accounts, and  $z E z'$ iff  account $z$ is a source of news for account $z'$, that is, $z'$ follows the news posted by $z$. Under this interpretation, the reflexivity  of $E$ says that every account is  a source of information for itself. 
For any news (fact) $p$, an account $z$ refutes $p$ iff none of its $E$-predecessors post (i.e.~support)  $p$. That is, $z$ refutes $p$ iff none of its information sources support $p$.
We interpret $z \succ p$ as account `$z$ refutes news $p$' (that is, $z$ makes a post denying $p$) and  $z \Vdash  p$ as account `$z$ confirms (posts) news $p$'. Formally, for any $z \in Z$, and any proposition $p$,

\[
z \succ p \quad \text{iff}\quad \text{for all }  z', \text{ if }  z' E z \, \text{ then } \, z' \not\Vdash p.
\]
 On the other hand,   an account $z$ posts news $p$ iff 
 none of the followers of $z$ refutes $p$.\footnote{Recall that, since $E$ is reflexive, $z$ is among its followers, and hence not refuting $p$ is a (clearly) necessary, but in general {\em not sufficient} condition for $z$ to post (support) $p$.} Formally, for  any $z \in Z$, and any proposition $p$, 
\[
z \Vdash p \quad \text{iff}\quad \text{for all $z'$}, \text{ if } z E z' \, \text{ then } \, z' \not\succ p.
\] 
This clause reflects the common  practice, in the management of   social media accounts, of not posting any news item which one's own audience would find controversial, on pain of loss of followers,  conflicts, backlashes or embarrassments \cite{powers2019shouting,doi:10.1080/1369118X.2013.862563,vraga2015individual}. Seen from  the followers' perspective, if one of the accounts they follow (confirm) post a news item $p$, then they either decide to no longer refute it or they stop following this account  (such behaviour is especially common in more politically active users\cite{10.1007/978-3-319-50901-3_30,rainie2012social}). Thus, the satisfaction clause above reflects  the equilibrium to which these dynamics converge. 

Note that, as $E$ is reflexive, if it is the case that $ z E z' \implies z' \not\succ p$, for all $z'$, then $z \not \succ p$. By the defining condition of refutation, the latter implies  an account $z'' \in Z$ exists such that $z'' E z$ and $z'' \Vdash p$. Thus, the compatibility conditions between relations $\Vdash$ and $\succ$ in graph-based modal logic is naturally justified in this interpretation. Notice that, in general, accounts followed by an account $z$ need not see news posted by $z$. That is, accounts which are sources of news for $z$ need not follow $z$ for the news. Thus, in general, it is possible that there exist accounts $z, z' \in Z $ such that $z E z'$, $z \succ p$ and $z' \Vdash p$. 
Note that relation $E$ encodes inherent indiscernibility of the model in the sense that the accounts which have the same sources of information (accounts they follows) will refute exactly same set of news items,  while accounts with the same set of followers will post the same set of news items. Thus, they are indiscernible from each other with regards to their behaviour on social media. This phenomenon is often observed on social media in the formation on communities which are so-called {\em echo-chambers}. Thus, we can see a concept $(\val{\phi}, \descr{\phi})$ as an echo-chamber (community) with news sources $\val{\phi}$ and followers $\descr{\phi}$. 

Let $R_\Diamond \supseteq E^{-1}$ be a relation such that for any $z,z' \in Z$, $z R_\Diamond z'$ iff $z$ can possibly hear news from $z'$. Recall that the refutation clause for $\Diamond$ (see Section \ref{ssec:logics}) is 
\[
z \succ \Diamond p \quad \text{iff} \quad \text{for all $z'$, } \; z R_\Diamond z' \; \text{implies} \; z' \not\Vdash p.
\]
This can be interpreted as follows. In some circumstances for some  news $p$, a social media account may want to be very sure before refuting it. In this case $z$ may check the posts from all the   news sources they can possibly access, and not only the accounts they follow, and only refute if none of the possible sources confirm (post about) $p$. That is, $z \succ \Diamond p $, iff none of the possible news sources for $z$ confirm (post) $p$.  

Let $R_\Box \supseteq E$ be a relation such that for any $z,z' \in Z$, $z R_\Box z'$ iff $z'$ can possibly get the news posted by  $z$. By the definition of refutation clause for $\Box$ (See Section \ref{ssec:logics}), is given by 
\[
z \Vdash \Box p \quad \text{iff} \quad \text{for all $z'$} \, z R_\Box z' \, \text{implies} \, z' \not\succ p.
\]
This can be interpreted as follows. In some circumstances for some  news $p$, a social media account may want to be sure that none of the accounts who are possibly  reached by them  refute it. In this case $z$ may confirm (post) news $p$ if none of the accounts who can possibly be reached by a post from $z$ refute $p$.  That is, $z \Vdash \Box p $, iff none of the possible recipients of news  from  $z$ refute $p$.  The relations $R_\Diamond$ and $R_\Box$ can be seen as the upper approximations of relations $E^{-1}$ and $E$. Thus, the accounts which have the same sets of accounts related to them by $R_\Diamond$, that is accounts who see posts from same set of accounts (even for important news) can be seen as a core set of followers of these sources (core of an echo-chamber), on the other hand,  the set of news sources with the same follower accounts related by $R_\Box$ are the accounts whose posts are always seen by the same set of accounts. These can be seen as the core news sources of an echo-chamber. Thus, for an echo-chamber (community) described by concept $(\val{\phi}, \descr{\phi})$, the sets $\val{\Box{\phi}}$  and  $\descr{\Diamond{\phi}}$ can be interpreted as  the core followers and core news sources of this echo-chamber respectively. The conditions $E \subseteq R_\Box$ and $E^{-1}  \subseteq R_\Diamond $, imply that $\val{\Box{\phi}} \subseteq \val{\phi}$  and  $\descr{\Diamond{\phi}} \subseteq \descr{\phi}$ (and hence that $\descr{\phi} \subseteq \descr{\Box{\phi}}$  and  $\val{\phi} \subseteq \val{\Diamond{\phi}}$) which can be interpreted as saying the set of core followers and core news sources of an echo-chamber are subsets of the sets of all followers and news sources in this echo-chamber respectively. 

\subsection{Rashomon}

Akira Kurosawa's acclaimed 1950 film {\em Rashomon} is set in medieval Japan and depicts the death by cutlass of a samurai through the widely differing and mutually conflicting accounts of the characters mainly involved: the bandit ($b$), the samurai's wife ($\ell$), and the samurai himself ($s$) speaking through a medium, as well as the account of a witness, the woodcutter ($w$). Both in popular culture and in academic literature, {\em Rashomon}'s story has been regarded as the paradigmatic illustration of the subjectivity of truth and the impossibility of reaching  factual accuracy, and the expression {\em Rashomon effect} has been coined to refer to the effect of observers' subjective perception on their recollection of an event causing them to produce widely divergent but equally plausible accounts of it.

Each main character testifies to be the one who killed $s$, and rather than  self-exculpation,
the main intent of the characters testimony (including $w$'s testimony) is to portray an image of themselves which is aligned with their own social role, and which embodies its virtues. Hence, in each testimony, the witnesses arrange their facts (highlighting some of them, ignoring or excluding or misrepresenting others) according to their own `moral agenda', and it is precisely the compatibility between their various `moral agendas' that we represent in the reflexive graph below (where all reflexive arrows are omitted), rather than the {\em factual} compatibility among the various testimonies: 
\begin{center}
\begin{tikzpicture}

\draw[very thick, ->] (0, 0) -- (-1.2, 0) ;
\draw[very thick, ->] (0, 0) -- (0, 1.2) ;
	\draw[very thick, ->] (0, 0) -- ( 1.2, 0);

\filldraw[black] (0,0) circle (2 pt);
	\filldraw[black] (-1.3,0) circle (2 pt);
	\filldraw[black] (1.3, 0) circle (2 pt);
    \filldraw[black] (0, 1.3) circle (2 pt);
	
		\draw (0, 1.6) node {$z_b$};

	\draw (-1.6, 0) node {{\small{$z_\ell$}}};

    \draw (0.20, -0.2) node {{\small{$z_w$}}};

    \draw (1.6, 0) node {{\small{$z_s$}}};

\end{tikzpicture}
\end{center}
In the graph $(Z, E)$ represented in the picture above, each state in the domain $Z: = \{z_b, z_\ell, z_s, z_w\}$ represents one testimony,  and the intended interpretation of edges in $E \subseteq Z\times Z$ is the  following:
\smallskip

{{\centering
$z_xEz_y$ \quad iff \quad  $y$'s portrayal of  $x$'s moral agenda does not clash with  $x$'s own portrayal.
\par}}
\smallskip
As represented in the graph above, the moral agendas of  $\ell, s$ and $b$  are undermined by all the other testimonies: indeed, $b$'s portrayal of himself as a competent, powerful, resourceful, and sexually irresistible bandit with his own code of honour is undermined by  $\ell$'s and $s$'s portrayal of him as a rapist, and by $w$'s portrayal of him as clumsy, incompetent, and gullible; 
$\ell$'s portrayal of herself as a passive victim (not only  of rape by the bandit, but also of her husband's cynicism and cold lack of interest in her honour and her fate) is undermined both by $b$'s portrayal of her as a fierce woman, and by $s$'s and $w$'s portrayals of her as being cold, manipulative, and capable of turning situations to her own advantage;   
$s$'s portrayal of himself as a honourable man who has suffered  the basest betrayal by his lawful wife is undermined by $b$'s description of  $\ell$'s legitimate interest in protecting her own honour, and by $\ell$'s and $w$'s portrayal of $s$ as a cowardly and dishonorable man. Finally, since the presence of  $w$ at the scene is not registered in any other testimony, $w$'s portrayal of himself as a pure witness and a law-abiding citizen is compatible with all testimonies (including his own).  
The concept lattice arising from this graph is represented in the following picture,
\footnote{Notice that the resulting concept lattice is distributive; this is unsurprising, since the graph originating it is not only reflexive but also transitive (cf.~\cite[Proposition 4.1]{Conradie2020NondistributiveLF})} where e.g.~$(\ell s, bw)$ is the shorthand for the formal concept of $(Z,E)$ having $\{z_\ell, z_s\}$ as its extension and $\{z_b, z_w\}$ as its intension. 

\begin{center}
\begin{tikzpicture}
\draw[very thick] (-1, 0) -- (-1, 1) --
	(0, 0) -- (1, 1) -- (1, 0) -- (0, 1) -- (-1, 0);
	\draw[very thick] (0, 2) -- (-1, 1);
\draw[very thick] (0, 2) -- (0, 1);
\draw[very thick] (0, 2) -- (1, 1);
	\draw[very thick] (0, -1) -- (-1, 0);
\draw[very thick] (0, -1) -- (0, 0);
\draw[very thick] (0, -1) -- (1, 0);
\draw[very thick] (0, 3) -- (0, 2);
\filldraw[black] (0,3) circle (2 pt);
	\filldraw[black] (0,-1) circle (2 pt);
	\filldraw[black] (0, 2) circle (2 pt);
    \filldraw[black] (-1, 1) circle (2 pt);
	\filldraw[black] (1, 1) circle (2 pt);
	\filldraw[black] (0, 1) circle (2 pt);
	\filldraw[black] (-1, 0) circle (2 pt);
	\filldraw[black] (1, 0) circle (2 pt);
	\filldraw[black] (0, 0) circle (2 pt);
		\draw (0.6, 2.2) node {$(b\ell s , w)$};
	\draw (0, -1.3) node {$(\varnothing, b\ell sw)$ };
 \draw (0, 3.3) node {$(b\ell sw, \varnothing)$ };
	\draw (-1.6, 0) node {{\small{$(b, \ell sw)$}}};
\draw (-1.6, 1) node {{\small{$(b\ell, sw)$}}};
    \draw (0.20, -0.2) node {{\small{$(\ell, bs w)$}}};
    \draw (0.07, 1.2) node {{\small{$(bs, \ell w)$}}};
    \draw (1.6, 0) node {{\small{$(s, b\ell w)$}}};
    \draw (1.6, 1) node {{\small{$(\ell s, b w)$}}};
     
\end{tikzpicture}
\end{center}
Additional ``epistemic'' relations, besides the basic relation $E$, can be considered, each of which encoding e.g.~how each of the main characters might argue (for instance, for the purpose of persuading a jury of their moral stance) that some of the other testimonies do not actually invalidate their own portrayal of themselves. For instance, $s$ might argue that the overall portrayal  $b$ makes of him is not incompatible with his own portrayal of himself as a dignified man, who adheres to his social role; likewise, $b$ might argue that, through $s$'s testimony,  his outlaw code of honour emerges with sufficient clarity; finally, $\ell$ might argue that, while in the beginning  $b$ might have mistaken her character as fierce, later on he recognized this as a misunderstanding, and hence, overall, $b$'s portrayal of $\ell$ is compatible with $\ell$'s  own. Thus, the additional relations $R_b$, $R_\ell$ and $R_s$ are represented as follows (in the following pictures, $E$ is understood to be properly included in each of the three additional relations):
\begin{center}
\begin{tikzpicture}
\draw[very thick, ->, red] (-5, 1.3) -- (-3.7, 0.1) ;
\draw[very thick, ->] (-5, 0) -- (-6.2, 0) ;
\draw[very thick, ->] (-5, 0) -- (-5, 1.2) ;
	\draw[very thick, ->] (-5, 0) -- ( -3.8, 0);

\filldraw[black] (-5,0) circle (2 pt);
	\filldraw[black] (-6.3,0) circle (2 pt);
	\filldraw[black] (-3.6, 0) circle (2 pt);
    \filldraw[black] (-5, 1.3) circle (2 pt);
	
		\draw (-5, 1.6) node {$z_b$};

	\draw (-6.6, 0) node {{\small{$z_\ell$}}};

    \draw (-5, -0.2) node {{\small{$z_w$}}};

    \draw (-3.3, 0) node {{\small{$z_s$}}};

\draw (-5, -0.7) node {{\small{$R_b$}}};

\draw[very thick, ->, red] (-1.2,0.2) -- (-0.1, 1.2) ;
\draw[very thick, ->] (0, 0) -- (-1.2, 0) ;
\draw[very thick, ->] (0, 0) -- (0, 1.2) ;
	\draw[very thick, ->] (0, 0) -- ( 1.2, 0);

\filldraw[black] (0,0) circle (2 pt);
	\filldraw[black] (-1.3,0) circle (2 pt);
	\filldraw[black] (1.3, 0) circle (2 pt);
    \filldraw[black] (0, 1.3) circle (2 pt);
	
		\draw (0, 1.6) node {$z_b$};

	\draw (-1.6, 0) node {{\small{$z_\ell$}}};

    \draw (0.20, -0.2) node {{\small{$z_w$}}};

    \draw (1.6, 0) node {{\small{$z_s$}}};

    \draw (0, -0.7) node {{\small{$R_\ell$}}};


    \draw[very thick, <-, red] (5.1, 1.2) -- (6.2, 0.1) ;
\draw[very thick, ->] (5, 0) -- (6.2, 0) ;
\draw[very thick, ->] (5, 0) -- (5, 1.2) ;
	\draw[very thick, ->] (5, 0) -- ( 3.8, 0);

\filldraw[black] (5,0) circle (2 pt);
	\filldraw[black] (6.3,0) circle (2 pt);
	\filldraw[black] (3.6, 0) circle (2 pt);
    \filldraw[black] (5, 1.3) circle (2 pt);
	
		\draw (5, 1.6) node {$z_b$};

	\draw (6.6, 0) node {{\small{$z_s$}}};

    \draw (5, -0.2) node {{\small{$z_w$}}};

    \draw (3.3, 0) node {{\small{$z_\ell$}}};

\draw (5, -0.7) node {{\small{$R_s$}}};

\end{tikzpicture}
\end{center}
That is,  the intended interpretation of edges in $R_u \subseteq Z\times Z$, for any $u\in \{b, \ell, s\}$, is the  following: for  all $x, y\in \{b, \ell, s, w\}$,
\smallskip

{{\centering
$z_xR_{u}z_y$ \quad iff \quad  $y$'s portrayal of  $x$'s moral agenda does not clash with  $x$'s own portrayal, according to $u$.
\par}}

These additional relations are all $E$-compatible, and each of them induces a pair of adjoint modal operators $\Diamond_u\dashv \Box_u$ for any $u\in \{b, \ell, s\}$, represented as follows:\footnote{In the pictures, the red descending arrows represent the non-identity assignments of the $\Box$-type operators and the blue ascending ones those of their left adjoints; for the sake of readability, the identity assignments of all operations are omitted.}
\begin{center}
\begin{tabular}{ccc}
\begin{tikzpicture}
\draw[very thick] (-1, 0) -- (-1, 1) --
	(0, 0) -- (1, 1) -- (1, 0) -- (0, 1) -- (-1, 0);
	\draw[very thick] (0, 2) -- (-1, 1);
\draw[very thick] (0, 2) -- (0, 1);
\draw[very thick] (0, 2) -- (1, 1);
	\draw[very thick] (0, -1) -- (-1, 0);
\draw[very thick] (0, -1) -- (0, 0);
\draw[very thick] (0, -1) -- (1, 0);
\draw[very thick] (0, 3) -- (0, 2);
\filldraw[black] (0,3) circle (2 pt);
	\filldraw[black] (0,-1) circle (2 pt);
	\filldraw[black] (0, 2) circle (2 pt);
    \filldraw[black] (-1, 1) circle (2 pt);
	\filldraw[black] (1, 1) circle (2 pt);
	\filldraw[black] (0, 1) circle (2 pt);
	\filldraw[black] (-1, 0) circle (2 pt);
	\filldraw[black] (1, 0) circle (2 pt);
	\filldraw[black] (0, 0) circle (2 pt);
		\draw (0.6, 2.2) node {$(b\ell s , w)$};
	\draw (0, -1.3) node {$(\varnothing, b\ell sw)$ };
 \draw (0, 3.3) node {$(b\ell sw, \varnothing)$ };
	\draw (-1.6, 0) node {{\small{$(b, \ell sw)$}}};
\draw (-1.6, 1) node {{\small{$(b\ell, sw)$}}};
    \draw (0.20, -0.2) node {{\small{$(\ell, bs w)$}}};
    \draw (0.07, 1.2) node {{\small{$(bs, \ell w)$}}};
    \draw (1.6, 0) node {{\small{$(s, b\ell w)$}}};
    \draw (1.6, 1) node {{\small{$(\ell s, b w)$}}};
        
    \draw (0, -2) node {$\Diamond_b \dashv \Box_b$ };  



\draw[very thick, red, <-] (0, 0.1)  .. controls (-0.1, 1.2) and  (-0.6, 1.3)  .. (-1, 1);

\draw[very thick, red, <-] (0, -0.9)  .. controls (-0.1, 0.2) and  (-0.6, 0.3)  .. (-1, 0); 

\draw[very thick, blue, ->] (1, 0.1)  .. controls (0.9, 1.2) and  (0.4, 1.3)  .. (0, 1);

\draw[very thick, blue, ->] (1, 1.1)  .. controls (0.9, 2.2) and  (0.4, 2.3)  .. (0, 2);


\end{tikzpicture}
& 
\begin{tikzpicture}
\draw[very thick] (-1, 0) -- (-1, 1) --
	(0, 0) -- (1, 1) -- (1, 0) -- (0, 1) -- (-1, 0);
	\draw[very thick] (0, 2) -- (-1, 1);
\draw[very thick] (0, 2) -- (0, 1);
\draw[very thick] (0, 2) -- (1, 1);
	\draw[very thick] (0, -1) -- (-1, 0);
\draw[very thick] (0, -1) -- (0, 0);
\draw[very thick] (0, -1) -- (1, 0);
\draw[very thick] (0, 3) -- (0, 2);
\filldraw[black] (0,3) circle (2 pt);
	\filldraw[black] (0,-1) circle (2 pt);
	\filldraw[black] (0, 2) circle (2 pt);
    \filldraw[black] (-1, 1) circle (2 pt);
	\filldraw[black] (1, 1) circle (2 pt);
	\filldraw[black] (0, 1) circle (2 pt);
	\filldraw[black] (-1, 0) circle (2 pt);
	\filldraw[black] (1, 0) circle (2 pt);
	\filldraw[black] (0, 0) circle (2 pt);
		\draw (0.6, 2.2) node {$(b\ell s , w)$};
	\draw (0, -1.3) node {$(\varnothing, b\ell sw)$ };
 \draw (0, 3.3) node {$(b\ell sw, \varnothing)$ };
	\draw (-1.6, 0) node {{\small{$(b, \ell sw)$}}};
\draw (-1.6, 1) node {{\small{$(b\ell, sw)$}}};
    \draw (0.20, -0.2) node {{\small{$(\ell, bs w)$}}};
    \draw (0.07, 1.2) node {{\small{$(bs, \ell w)$}}};
    \draw (1.6, 0) node {{\small{$(s, b\ell w)$}}};
    \draw (1.6, 1) node {{\small{$(\ell s, b w)$}}};

    \draw (0, -2) node {$\Diamond_\ell \dashv \Box_\ell$ };   

\draw[very thick, red, <-] (1.1, 0)  .. controls (1.4, 0.2) and  (1.4, 0.8)  .. (1.1, 1); 

\draw[very thick, red, <-] (0.1, -1)  .. controls (0.4, -0.8) and  (0.4, -0.2)  .. (0.1, 0);

    \draw[very thick, blue, ->] (-1.1, 0)  .. controls (-1.4, 0.2) and  (-1.4, 0.8)  .. (-1.1, 1); 
  \draw[very thick, blue, ->] (-0.1, 1)  .. controls (-0.4, 1.2) and  (-0.4, 1.8)  .. (-0.1, 2);

\end{tikzpicture}
&
\begin{tikzpicture}
\draw[very thick] (-1, 0) -- (-1, 1) --
	(0, 0) -- (1, 1) -- (1, 0) -- (0, 1) -- (-1, 0);
	\draw[very thick] (0, 2) -- (-1, 1);
\draw[very thick] (0, 2) -- (0, 1);
\draw[very thick] (0, 2) -- (1, 1);
	\draw[very thick] (0, -1) -- (-1, 0);
\draw[very thick] (0, -1) -- (0, 0);
\draw[very thick] (0, -1) -- (1, 0);
\draw[very thick] (0, 3) -- (0, 2);
\filldraw[black] (0,3) circle (2 pt);
	\filldraw[black] (0,-1) circle (2 pt);
	\filldraw[black] (0, 2) circle (2 pt);
    \filldraw[black] (-1, 1) circle (2 pt);
	\filldraw[black] (1, 1) circle (2 pt);
	\filldraw[black] (0, 1) circle (2 pt);
	\filldraw[black] (-1, 0) circle (2 pt);
	\filldraw[black] (1, 0) circle (2 pt);
	\filldraw[black] (0, 0) circle (2 pt);
		\draw (0.6, 2.2) node {$(b\ell s , w)$};
	\draw (0, -1.3) node {$(\varnothing, b\ell sw)$ };
 \draw (0, 3.3) node {$(b\ell sw, \varnothing)$ };
	\draw (-1.6, 0) node {{\small{$(b, \ell sw)$}}};
\draw (-1.6, 1) node {{\small{$(b\ell, sw)$}}};
    \draw (0.20, -0.2) node {{\small{$(\ell, bs w)$}}};
    \draw (0.07, 1.2) node {{\small{$(bs, \ell w)$}}};
    \draw (1.6, 0) node {{\small{$(s, b\ell w)$}}};
    \draw (1.6, 1) node {{\small{$(\ell s, b w)$}}};
     
\draw[very thick, red, <-] (0, 0.1)  .. controls (0.1, 1.2) and  (0.6, 1.3)  .. (1, 1);
  \draw (0, -2) node {$\Diamond_s \dashv \Box_s$ };  

  \draw[very thick, red, <-] (0, -0.9)  .. controls (0.1, 0.2) and  (0.6, 0.3)  .. (1, 0); 

\draw[very thick, blue, ->] (-1, 0.1)  .. controls (-0.9, 1.2) and  (-0.4, 1.3)  .. (0, 1);

\draw[very thick, blue, ->] (-1, 1.1)  .. controls (-0.9, 2.2) and  (-0.4, 2.3)  .. (0, 2);
  \draw (0, -2) node {$\Diamond_s \dashv \Box_s$ };  
\end{tikzpicture}
\\
\end{tabular}
\end{center}
As it can be easily verified by inspection, the operations represented above validate the following axioms for any $u\in \{b, \ell, s\}$ and any $c\in \mathbb{L}$:
\[\Box_u c\vdash c\quad c\vdash \Diamond_u c\quad c\vdash \Box_u\Diamond_u c\quad \Diamond_u\Box_uc\vdash c\quad \Box_uc\vdash \Box_u\Box_u c\quad \Diamond_u\Diamond_u c\vdash \Diamond_u c,\]
which implies (cf.~Table \ref{fig:MPS:Famous}) that the relations $R_b$, $R_\ell$ and $R_s$ are $E$-reflexive and $E$-transitive (see Remark \ref{remark:wtf_are_E_properties} below).
Hence, each of the pairs of adjoint operations induced by the additional relation can be understood to represent the lower and upper  approximations of `knowable  truth' from the subjective viewpoint of the main characters of the story. Finally, notice that the `knowable truth' can be exactly captured via the collective upper and lower  approximations respectively defined by  the  assignments $c\mapsto \Box_b c\vee \Box_\ell c\vee \Box_s c$ 
and $c\mapsto \Diamond_b c\wedge \Diamond_\ell c\wedge \Diamond_s c$.

\begin{remark}
\label{remark:wtf_are_E_properties}
In Section \ref{ssec:shifting}, we show how the first-order correspondents of modal reduction principles in Kripke frames can be {\em shifted} (see Definition \ref{def:shifting} and Theorem \ref{thm: sahlqvist shifting}) to properties on graph-based frames which are parametric on the relation $E$; hence we use the convention that the shifting (parametric on $E$) of some property {\em prop} is called $E$-{\em prop}. For instance, the shifting of denseness is called $E$-denseness, and the shifting of reflexivity is $E$-reflexivity.
\end{remark}

%
In  both examples considered above, the reflexive  relations $E$ is used to model the {\em inherent indiscernibility} or {\em limits of knowability} in a given situation, while additional relations ($R_\Diamond$, $R_\Box$, $R_s$, $R_\ell$ etc.) are used to model {\em agent-specific indiscernibility}. The inherent indiscernibility relation $E$ gives rise to a complete lattice, representing the space of conceptual possibilities and their ordering (in the first example above this becomes the space of so-called echo chambers while, in the second, the moral compatibility space of testimonies). This is analogous to the equivalence relation in an approximation space giving rise to a Boolean algebra of rough sets. Unlike in the modal logic approach to rough set theory (see, e.g.,\ \cite{orlowska1994rough}), however, $E$ is not used to interpret modalities. Rather, $E$ is used to generate a lattice of stable sets on which the propositional connectives are algebraically interpreted, thereby   `internalizing $E$' into the propositional fragment of the logic, which causes it to be non-classical and, in fact, non-distributive. This is similar to the situation in intuitionistic (modal) logic where the partial orders in intuitionistic Kripke frame generate Heyting algebras of upwards closed sets. 

The relations associated with an agent captures the peculiarities of their specific epistemic access or  lens through which they view the space of conceptual possibilities  (in the first example above, this takes the form of epistemic access between agents beyond the ``source/follower'' relation and, in the second, the agents' notions of compatibility as viewed from their personal moral standpoints). As one would expect, such views remain inherently conditioned and mediated by the space of conceptual possibilities (i.e.~by the inherent bounds on knowability), and this fact is mathematically captured by the requirement that the agents' relations be compatible with $E$. Accordingly, a modal expression $\Diamond p$ represents the proposition $p$ reflected through a particular agent-specific lens.  Thus, the modeling of the scenarios above using graph-based frames maintains a neat separation between the inherent and the epistemic indiscernibility.  Such a clear division of labour would not be possible had we used $E$ to interpret an additional set of modal operators.

\section{Hyper-constructivist approximation spaces} \label{sec:Rough set theory on graph-based frames}

In the previous section, we showed how graph-based frames and their associated  modal languages can formalize reasoning under epistemic uncertainty (modelled by additional relations $R_\Box$, $R_\Diamond$ etc.) in a framework with inherent indiscernibility (modeled by relation $E$). In a similar way, such models have served to represent   other scenarios in which inherent limits to knowability play a relevant role  \cite{graph-based-wollic,CONRADIE2021115,Conradie2019/08}. These examples motivate  the introduction of {\em hyper-constructivist} (or {\em evidentialist}) {\em approximation spaces}, based  on graph-based frames, as a generalization of approximation spaces to a framework with inherent informational entropy. The strategy we  follow for introducing this generalization  is  similar  to the one adopted in \cite{CONRADIE2021371} for introducing the notion of conceptual approximation space. Namely, we identify the defining properties of hyper-constructivist approximation spaces (cf.\ Definition \ref{def:graph based approximation spaces}) in terms of the first-order conditions corresponding, on graph-based frames, to the Sahlqvist mrps corresponding, on  Kripke frames, to the first-order conditions expressing the seriality, reflexivity, transitivity, and symmetry of the accessibility relations associated with the modal operators. 
 
\begin{definition}\label{def:graph based approximation spaces}
{\em A hyper-constructivist} (or  {\em evidentialist}) {\em approximation space} is a graph-based $\mathcal{L}$-frame $\mathbb{F} = (Z,E,R_\Box, R_\Diamond)$ such that   $E \subseteq R_{\Box}\, \compbox \, R_{\blacksquare}$, where $\compbox$ is defined  as in Definition \ref{def:relational composition}. 
A hyperconstructivist approximation space  $\mathbb{F} = (Z,E,R_\Box, R_\Diamond)$ is:
\begin{enumerate}
    \item $E$-{\em reflexive} if   $E \subseteq R_\Box$ and $D \subseteq R_\Diamond $. 
    \item $E$-{\em symmetric} if   $R_\blacksquare \subseteq R_\Box$ and $R_\Diamondblack \subseteq R_\Diamond $.
    \item $E$-{\em transitive} if   $R_{\Box}\compbox R_{\Box}\subseteq R_{\Box}$ and  $R_{\Diamond}\compdia R_{\Diamond}\subseteq R_{\Diamond}$.
    \item{\em Pawlak} if it is $E$-reflexive, $E$-symmetric and $E$-transitive. 
\end{enumerate}
\end{definition}
In what follows, we motivate the definition above. As discussed in Section \ref{ssec:approx spaces}, one of the most basic requirements in the notion of generalized approximation spaces is that the indiscernibility relation be {\em serial}, since this condition is the first-order correspondent of the mrp $\Box a \leq \Diamond a $, which guarantees  that 
the lower approximation is smaller than or equal to the upper approximation. The ALBA outputs of   $\Box a \leq \Diamond a $ (which is an analytic inductive mrp) are $\forall\nomj(\nomj\leq \Diamond\Diamondblack \nomj)$ and $\forall \cnomm(\Box\blacksquare \cnomm\leq \cnomm)$ (cf.~Section \ref{ssec:mrp}), which,  when translated in the first-order language of Kripke frames, 
 yield the first-order correspondents $\forall w(\{w\}\subseteq R_\Diamond^{-1}[R^{-1}_{\Diamondblack}[w]])$ and $\forall w((R_\Box^{-1}[(R_{\blacksquare}^{-1}[\{w\}^{cc}])^{cc}])^c\subseteq \{w\}^c)$, which can be respectively unravelled as $\forall w \exists v (w R_\Diamond v \, \& \, v R_{\Diamondblack} w)$ and $\forall w \exists v (w R_\Box v \, \& \, v R_{\blacksquare} w)$. 
 These first-order sentences can   equivalently be rewritten  as  the  relational inclusions $\Delta \subseteq R_\Diamond \circ R_{\Diamondblack}$ and   $\Delta \subseteq R_\Box \circ R_{\blacksquare}$, respectively. Translating the ALBA outputs above in the first-order language of graph-based frames  yields 
 \begin{center}
 $\forall z(R_\Diamond^{[0]}[E^{[0]}[R_{\Diamondblack}^{[0]}[z^{[10]}]]]\subseteq z^{[1]} )\quad $ and $\quad \forall z(R_\Box^{[0]}[E^{[1]}[R_{\blacksquare}^{[0]}[z^{[01]}]]] \subseteq z^{[0]} )$.
 \end{center} 
By the $E$-compatibility of $R_\Box$ and $R_\Diamond$ (cf.~Lemma \ref{equivalents of I-compatible}) these inclusions can be simplified  as follows: 

 \begin{center}
 $\forall z(R_\Diamond^{[0]}[E^{[0]}[R_{\Diamondblack}^{[0]}[z]]]\subseteq z^{[1]} )\quad $ and $\quad \forall z(R_\Box^{[0]}[E^{[1]}[R_{\blacksquare}^{[0]}[z]]] \subseteq z^{[0]} )$,
 \end{center} 
 and then be partially unravelled as the following implications (recall that $D\coloneqq E^{-1})$: 
 \begin{center}
  $\forall x \forall z ( x \in R_\Diamond^{[0]}[E^{[0]}[R_\Diamondblack^{[0]}[z]]] \Rightarrow x D^c z)\quad $ and 
  $\forall a \forall z ( a \in R_\Box^{[0]}[E^{[1]}[R_\blacksquare^{[0]}[z]]] \Rightarrow a E^c z)$, 
 \end{center} 
 which, on application of contraposition yields 
 \begin{center}
  $\forall x \forall z (x D z \Rightarrow x \notin R_\Diamond^{[0]}[E^{[0]}[R_\Diamondblack^{[0]}[z]]])\quad $ and 
 $\quad \forall a\forall z (a E z \Rightarrow a \notin R_\Box^{[0]}[E^{[1]}[R_\blacksquare^{[0]}[z]]])$, 
 \end{center} 
 and then as follows: 
  \begin{center}
  $\forall x \forall z (x D z \Rightarrow x (R_\Diamond \compdia R_\Diamondblack) z) \quad $ and 
 $\quad \forall a\forall z (a E z \Rightarrow a(R_\Box \compbox R_\blacksquare)z) $, 
 \end{center} 
 by letting $x(R_\Diamond \compdia R_\Diamondblack) z$  iff  $x \notin R_\Diamond^{[0]}[E^{[0]}[R_\Diamondblack^{[0]}[z]]]$, and $a(R_\Box \compbox R_\blacksquare)z$ iff $a \notin R_\Box^{[0]}[E^{[1]}[R_\blacksquare^{[0]}[z]]]$.
These `abbreviations', which are in fact  instances of the notion of $E$-mediated composition (cf.~Definition \ref{def:relational composition} in Section \ref{ssec:lifting properties}, below), allow us to equivalently rewrite  these first-order sentences as relational inclusions $D \subseteq R_\Diamond \compdia R_\Diamondblack$ and  $E \subseteq R_\Box\compbox R_\blacksquare$, respectively.

What is more, when rewritten in this way, 
 it is possible to recognize the similarity between the first-order correspondents on  graph-based frames and those 
 on Kripke frames. Specifically, it is not difficult to verify that, instantiating $E\coloneqq \Delta$ in $D \subseteq R_\Diamond \compdia R_\Diamondblack$, one obtains a condition which is equivalent to $\Delta \subseteq R_\Diamond \circ R_\Diamondblack$, and similarly with the second pair of inclusions.  
This hints at the existence of a {\em transfer} of relational constructions and  properties along the natural embedding  of sets into reflexive graphs. Such transfer  hinges on a `parametric shift' from $\Delta$ to $E$. 


In Section \ref{ssec:shifting},  
we show that this transfer of properties  and constructions, formalized in terms of the notion of {\em shifting}  (cf.~Definition \ref{def:shifting}), extends to all first-order correspondents of inductive mrps.
In Section \ref{sec:Applications to rough set theory}, we apply this result to the principled definition of various hyper-constructivist counterparts of generalized approximation spaces, 
%
%
 and discuss how this transfer works not only technically but also conceptually, given that  the intuitive meaning of modal axioms transfers  from Kripke frames to the graph-based frames (cf.~Examples \ref{ex:reflexivity}, \ref{ex:symmetry}, \ref{ex:transitivity}).  
\section{Relational compositions, and the algebras of relations on graph-based  and Kripke frames}
\label{sec: composing crisp}

In this section we partly recall, partly introduce several notions of relational compositions  on  graph-based frames \cite{graph-based-wollic,conradie2022modal}. We also introduce the relational languages $\mathsf{KRel}_{\mathcal{L}}$ and $\mathsf{GRel}_{\mathcal{L}}$ associated with Kripke frames and graph-based frames respectively and used to express first-order conditions in these contexts. We define a translation from $\mathsf{GRel}$ to $\mathsf{KRel}$ which is compatible with the natural relationship between these classes. The definitions and notation we introduce here will be key to  a meaningful comparison of the correspondents of inductive mrps in different semantic settings.

\subsection{Composing relations on graph-based frames }
\label{ssec:lifting properties}

The following definition expands those given in \cite[Section 3.4]{CONRADIE2021371}. As will be clear from Lemma \ref{lemma:compdia compbox delta}, the $E$-compositions can be understood as the counterparts of the usual composition of relations in the setting of graph-based frames. 
\begin{definition}
\label{def:relational composition}
	For any graph  $\mathbb{G}= (Z, E)$, and any $R, T\subseteq Z\times Z$, the $E$-mediated compositions $R\,\compdia T$ and $ R\,\compbox T$ and the composition $ R\,\ast T \subseteq Z\times Z$ are defined as follows: for all $a,x\in Z$,
	\begin{center}
    \begin{tabular}{rcccl}
    $x (R\, \compdia T) a$ & iff & $x \in (R^{[0]}[E^{[0]}[T^{[0]} [a]]])^c$ & i.e. & $(R\,\compdia T)^{[0]} [a] = {R}^{[0]}[E^{[0]}[T^{[0]} [a]]]$\\
    $a (R\, \compbox T) x$ & iff & $a \in (R^{[0]}[E^{[1]}[T^{[0]} [x]]])^c$ & i.e. & $(R\, \compbox T)^{[0]} [x] = R^{[0]}[E^{[1]}[T^{[0]} [x]]]$\\ 
    $a (R\, \ast T) x$ & iff & $a \in (R^{[0]}[T^{[0]} [x]])^c$ & i.e. & $(R\, \ast T)^{[0]} [x] = R^{[0]}[T^{[0]} [x]]$.
    \end{tabular}
    \end{center}
\end{definition}

The following lemma, which readily follows from the definition, shows that the two notions of $E$-mediated composition collapse to the usual notion when $E \coloneqq \Delta$.
\begin{lemma}
\label{lemma:compdia compbox delta}
For all sets $S$, all relations $R,T \subseteq S \times S$ on the graph $\mathbb{G}_S = (S, \Delta)$, 
\begin{center}
    $R \, \compdiadelta T = R \, \compboxdelta T= R \, \circ T.$
\end{center}
\end{lemma}
The following lemma shows that the $E$-mediated compositions $\compbox$ and $\compdia$ are well defined, associative, and have unit $E$ and $D$, respectively; hence they endow the set of $E$-compatible relations over a reflexive graph with a monoidal structure.
\begin{lemma}
\label{lemma:properties of homogeneous composition}
For any graph $\mathbb{G}= (Z,E)$,  all $R, T, U \subseteq Z\times Z$, and any $Y \subseteq Z$,
\begin{enumerate}
    \item[1.] if $R$ is $E$-compatible, then $R\,\compdia  D = R = D\, \compdia R  \quad \text{and} \quad 
R\,\compbox  E = R = E\, \compbox R$;
    \item[2.] if $R$ is $E$-compatible, then so are $R\,\compdia T$, $R\,\compbox T$, and $R\,\ast T$;
    \item[3.] if $R$ and $T$ are $E$-compatible, then $(R \, \compdia T)^{[0]}[Y]=R^{[0]}[E^{[0]}[T^{[0]}[Y]]]$ \ \ and \ \ $(R\, \compbox T)^{[0]}[Y]=R^{[0]}[E^{[1]}[T^{[0]}[Y]]]$;
    \item[4.] if $R$, $T$, and $U$ are $E$-compatible, then $R \,\compdia (T \,\compdia U) = (R \,\compdia T) \,\compdia U$ \ \ and \ \ $R \,\compbox (T \,\compbox U) = (R \,\compbox T) \,\compbox U$.
\end{enumerate}
\end{lemma}
\begin{proof}
1. For any $a\in Z$, $(R\,\compdia D)^{[0]}[a] = R^{[0]}[E^{[0]}[D^{[0]}[a]]] = R^{[0]}[E^{[0]}[E^{[1]}[a]]]= R^{[0]}[a]$, the last identity holding by the $E$-compatibility of $R$. The remaining identities are omitted.\\
2. For any $a \in Z$, 
\begin{center}
\begin{tabular}{r c l r }
   $((R\,\compdia T)^{[0]} [a])^{[10]}$& =& $({R}^{[0]}[E^{[0]}[T^{[0]} [a]]])^{[10]} $ & Definition \ref{def:relational composition} (i) \\
   &=& ${R}^{[0]}[E^{[0]}[T^{[0]} [a]]] $ & Lemma \ref{equivalents of I-compatible} \\
   &=& $(R\,\compdia T)^{[0]} [a]$.& Definition \ref{def:relational composition} (i) \\
\end{tabular}
\end{center}
The remaining identities are proven similarly.\\
3. We only prove the first statement.
\begin{center}
\begin{tabular}{r c l r}
    $R^{[0]}[E^{[0]}[T^{[0]}[Y]]]$ & =& $R^{[0]}[E^{[0]}[T^{[0]}[\bigcup_{y \in Y}y]]]$& $Y = \bigcup_{y\in Y}y$\\
    &=& $R^{[0]}[E^{[0]}[ \bigcap_{y \in Y} T^{[0]}[y]   ]]$ & Lemma \ref{lemma: basic} (v)\\
    &=& $R^{[0]}[E^{[0]}[ \bigcap_{y \in Y} E^{[1]}[E^{[0]}[T^{[0]}[y]   ]]]]$ & $E$-compatibilty of $T$\\
    &=& $R^{[0]}[E^{[0]} [E^{[1]}[ \bigcup_{y \in Y} E^{[0]}[T^{[0]}[y]]]]]$ & Lemma \ref{lemma: basic} (v)\\
    &=& $R^{[0]} [\bigcup_{y \in Y} E^{[0]}[T^{[0]}[y]]] $ & $E$-compatibilty of $R$\\
    &=& $ \bigcap_{y \in Y} R^{[0]}[ E^{[0]}[T^{[0]}[y]]] $ & Lemma \ref{lemma: basic} (v)\\
    &=& $ \bigcap_{y \in Y} (R \, \compdia T)^{[0]}[y] $ & Definition \ref{def:relational composition} (i) \\
    &=& $(R \, \compdia T)^{[0]}[Y] $. & Lemma \ref{lemma: basic}(v)
\end{tabular}
\end{center}
4. Immediate consequence of item 3. 
\end{proof}


\begin{remark}
\label{remark:astcomp not associative}
In general, $\ast$-composition is not associative;
let us build a counterexample to $R\ast(T\ast U) = (R\ast T)\ast U$. Let $\mathbb{G} = (Z, \Delta)$, such that $Z = \{z_1, z_2\}$. let $R,U,T\subseteq Z\times Z$ such that $R =  Z \times Z$, $U = \varnothing$, and $T = \Delta$, respectively. 
Since $E=\Delta$, every set is Galois-stable, hence all three relations are $E$-compatible. For any $x\in Z$,
\begin{center}
\begin{tabular}{rclr}
 $ R^{[0]}[(T \, \ast U)^{[0]}[x]]$ & = 
&  $R^{[0]}[T^{[0]}[U^{[0]}[x]]]$ & Definition \\
& = & $R^{[0]}[T^{[0]}[Z]]$ & $U=\varnothing$ \\
& = & $R^{[0]}[\{ b \in Z \mid \forall w(w \in Z \Rightarrow b\Delta^c w)\}]$ & $T=\Delta$ \\
& = & $R^{[0]}[\varnothing]$ &  \\
& = & $\{ b \in Z \mid \forall w(w \in \varnothing \Rightarrow bR^c w)\}$\\
& = & $Z$ &  \\
\end{tabular}
\end{center}

\begin{center}
\begin{tabular}{rclr}
 $(R\,\ast T)^{[0]}[ U^{[0]}[x]]$ 
&  = & $ (R\,\ast T)^{[0]}[Z]$ & $U=\varnothing$ \\
& = & $ \{ b \in Z \mid \forall w(w \in Z \Rightarrow b (R\,\ast T)^c w ) \}$ & Definition of $(-)^{[0]}$ \\
& = & $ \{ b \in Z \mid \forall w(w \in Z \Rightarrow b\in  R^{[0]}[ T^{[0]} [w]] ) \}$ & Definition of $\ast$\\
& = & $\bigcap_{w\in Z}R^{[0]}[ T^{[0]} [w]]$ \\
& = & $\bigcap_{w\in Z}R^{[0]}[Z\setminus\{w\}]$ & $T = \Delta$\\
& = & $R^{[0]}[\bigcup_{w\in Z}Z\setminus\{w\}]$ & Lemma \ref{lemma: basic}\\
& = & $R^{[0]}[Z]$ & \\
& = & $\{ b \in Z \mid \forall w(w \in Z \Rightarrow bR^c w)\}$ & \\
& = & $\varnothing$ &   $R = Z\times Z$\\
\end{tabular}
\end{center}
\end{remark}

\begin{remark}
In general, $\ast$-composition is not associative. It is easy to show a counterexample analogous to the one in Remark \ref{remark:astcomp not associative} which its natural counterpart in Kripke frames. To simplify notation, we will write e.g.~$R\ast T\ast U$ which should be read as $(R\ast (T\ast U))$. 
\end{remark}

\subsection{(Heterogeneous) relation algebras and their propositional languages}\label{ssec:krel and grel}
In this section, we introduce algebraic structures in the context of which  the notions of composition introduced above can be studied systematically from a universal-algebraic or category-theoretic perspective. Besides being of potential independent interest, these structures can be naturally associated with {\em  propositional languages} which are key tools to build a common ground where the first-order correspondents of modal axioms in various semantic contexts can be compared and systematically related. 
\begin{definition} 
The relation algebras associated with any graph-based frame   $\mathbb{F} = (Z, E, R_\Diamond, R_\Box)$, and any Kripke frame   $\mathbb{X} = (W, R_\Diamond, R_\Box)$ are respectively defined as follows: 
\begin{center}
$\mathbb{F}^\ast: = (\mathcal{P}(Z\times Z), E, D, R_\Diamond, R_\Diamondblack, R_{\Box}, R_\blacksquare, \compdia,  \compbox, \ast)$, 
$\quad \mathbb{X}^\ast: = (\mathcal{P}(W\times W), \Delta, R_\Diamond, R_\Diamondblack, R_{\Box}, R_\blacksquare, \circ, \star),$
\end{center}
where  $R_\Diamondblack$ and $R_\blacksquare$ denote the converse relations of $R_\Box$ and $R_\Diamond$ respectively, the binary operations $\compdia, \compbox$ and $\ast$ are defined as in Definition \ref{def:relational composition}, the binary operation $\circ$  is   the standard relational composition, and  $\star$ is   defined by the same assignment  as $\ast$.
\end{definition}

Consider the set $\mathsf{Rel}_{\mathcal{L}}: = \{\Delta, E, D, R_\Diamond, R_\Diamondblack, R_{\Box}, R_\blacksquare\}$ of constant symbols. The {\em  propositional languages}  $\mathsf{KRel}_{\mathcal{L}}$ and $\mathsf{GRel}_{\mathcal{L}}$ associated with Kripke frames (resp.\ graph-based frames) are defined by the following  recursions:
\begin{align*}
\mathsf{KRel}_{\mathcal{L}} \ni \xi :: = &\, \Delta \mid R_\Diamond  \mid R_\Diamondblack\mid R_{\Box} \mid R_\blacksquare \mid  \xi\circ \xi  \mid \xi \star\xi, \\
\mathsf{GRel}_{\mathcal{L}} \ni \alpha :: = &\, D \mid E \mid  R_\Diamond  \mid R_\Diamondblack\mid R_{\Box} \mid R_\blacksquare \mid  \alpha \, \compdia \alpha  \mid \alpha \, \compbox \alpha  \mid \alpha \,\ast\alpha,
\end{align*}
where the intended interpretation of the constant symbols on graph-based frames and on Kripke frames is the obvious one. 
Term inequalities of $\mathsf{KRel}_{\mathcal{L}}$ (resp.~$\mathsf{GRel}_{\mathcal{L}}$ ) are written as $\xi_1\subseteq \xi_2$ (resp.\ $\alpha_1 \subseteq \alpha_2)$ and are interpreted on Kripke frames (resp.\ graph-based frames) in the obvious way. In particular, 
\begin{center}
$\mathbb{X}\models \xi_1\subseteq \xi_2\quad \text{ iff } \quad\mathbb{X}^\ast \models \xi_1\subseteq \xi_2
\quad \text{and} \quad \mathbb{F}\models \alpha_1\subseteq \alpha_2\quad \text{ iff } \quad\mathbb{F}^\ast \models \alpha_1\subseteq \alpha_2.$    
\end{center}

A natural translation exists between $\mathsf{GRel}_{\mathcal{L}}$-terms and $\mathsf{KRel}_{\mathcal{L}}$-terms which is defined as follows.

\begin{definition}
\label{def:translationgrelkrel}
Let $\tau: \mathsf{GRel}_{\mathcal{L}} \to \mathsf{KRel}_{\mathcal{L}}$ be  defined as follows:
\smallskip

{{\centering
\begin{tabular}{rclcrcl}
$\tau(E)$ & $=$ & $\Delta$ &
\quad\quad\quad &
$\tau(D)$ & $=$ & $\Delta$ \\
$\tau(R_\Diamond)$ & $=$ & $R_\Diamond$ &
\quad\quad\quad &
$\tau(R_\Diamondblack)$ & $=$ & $R_\Diamondblack$ \\
$\tau(R_\Box)$ & $=$ & $R_\Box$ &
\quad\quad\quad &
$\tau(R_\blacksquare)$ & $=$ & $R_\blacksquare$ \\
$\tau(\alpha_1 \compdia \alpha_2)$ & $=$ & $\tau(\alpha_1) \circ \tau(\alpha_2)$ &
\quad\quad\quad &
$\tau(\alpha_1 \compbox \alpha_2)$ & $=$ & $\tau(\alpha_1) \circ \tau(\alpha_2)$ \\
$\tau(\alpha_1 \ast \alpha_2)$ & $=$ & $\tau(\alpha_1) \star \tau(\alpha_2)$ &
\quad\quad\quad &
$\tau(\alpha_1 \ast \alpha_2)$ & $=$ & $\tau(\alpha_1) \star \tau(\alpha_2)$.
\end{tabular}
\par}}
\end{definition}

\begin{lemma}
\label{lemma:translationgrel}
For any Kripke frame $\mathbb{X}$, any $\mathsf{GRel}_{\mathcal{L}}$-inequality $\alpha_1 \subseteq \alpha_2$, and any $\alpha \in \mathsf{GRel}_{\mathcal{L}}$
\begin{enumerate}
\item $\alpha^\mathbb{F_X} = \tau(\alpha)^\mathbb{X}$,
\item $\mathbb{F_X} \models \alpha_1 \subseteq \alpha_2 
\quad\quad
\mbox{iff}
\quad\quad
\mathbb{X} \models \tau(\alpha_1) \subseteq \tau(\alpha_2).$
\end{enumerate}
\end{lemma}
\begin{proof}
Item 1 is proved by straightforward induction on the complexity of $\alpha$. Item 2 directly follows from 1.
\end{proof}

In the next section we represent the first-order correspondents of inductive mrps as term inequalities of $\mathsf{GRel}_{\mathcal{L}}$ and $\mathsf{KRel}_{\mathcal{L}}$.
Intuitively,  the various compositions considered above are used as ``abbreviations'' to achieve a more compact representation of the correspondents which is also more amenable to computation. 

\section{Correspondents of inductive modal reduction principles}
\label{sec:crisp}
With the definitions of the previous section in place, we are now in a position to give a concrete representation of the first-order correspondents of inductive mrps  on graph-based $\mathcal{L}$-frames and on  Kripke $\mathcal{L}$-frames as term-inequalities in the languages $\mathsf{GRel}_{\mathcal{L}}$ and $\mathsf{KRel}_{\mathcal{L}}$. This representation also puts us in a position to show that these two representations are naturally connected in the same way in which Kripke frames are connected to graph-based frames.  Specifically, in the following two subsections, we show that 

\begin{proposition}
\label{prop: pure inclusions crisp}
The first-order correspondent of any inductive mrp on graph-based (resp.~Kripke) frames can be represented as a term inequality in $\mathsf{GRel}_{\mathcal{L}}$ (resp.~$\mathsf{KRel}_{\mathcal{L}}$). 
\end{proposition}
In Section \ref{ssec:shifting}, we also show that  the first-order correspondent on graph-based $\mathcal{L}$-frames of any inductive mrp is the ``shifted'' version of its first-order correspondent on Kripke $\mathcal{L}$-frames. 
\subsection{Correspondents on graph-based frames}
\label{ssec: correspondents on graph-based}
\begin{definition}\label{def:graph-based rels associated with terms}
Let us associate terms in $\mathsf{GRel}_{\mathcal{L}}$ with ${\mathcal{L}^*}$-terms of certain syntactic shapes as follows: 
\begin{enumerate}
    \item if $\varphi(!w)$ (resp.~$\psi(!w)$) is a finite (possibly empty) concatenation of diamond (resp. box) connectives, let us define $R_{\phi}\in \mathsf{GRel}_{\mathcal{L}}$ (resp.~$R_{\psi}\in \mathsf{GRel}_{\mathcal{L}}$)  by induction on $\phi$ (resp.~ $\psi$) as follows:\\
if $\phi: = y$, then $R_{\phi}: = D$; \quad 
if $\phi: = f \phi'$ with $f\in\{\Diamond, \Diamondblack\}$, then $R_{\phi}: = R_f \compdia R_{\phi'}$;\\
if $\psi: = y$, then $R_{\psi}: = E$; \quad 
if $\psi: = g \psi'$ with $g \in \{\Box,\blacksquare\}$, then $R_{\psi}: = R_g \compbox R_{\psi'}$;
\item if $\chi(!w) = \phi_1\psi_1\ldots\phi_{n_\chi}\psi_{n_\chi}(!w)$ such that each $\phi_i$ (resp. $\psi_i$) is a finite concatenation of diamond (resp.\ box) connectives  of which only $\psi_{n_\chi}$ is possibly empty, then  $R_{\chi}\in \mathsf{GRel}_{\mathcal{L}}$ is defined 
as follows: \\
if $\psi_{n_\chi}$ is empty,
$
    R_{\chi} \coloneqq R_{\varphi_1} \ast R_{\psi_1} \ast \cdots \ast  R_{\varphi_{n_\chi}};\\
$ 
if $\psi_{n_\chi}$ is nonempty,
$
    R_{\chi} \coloneqq R_{\varphi_1} \ast R_{\psi_1} \ast \cdots \ast R_{\varphi_{n_\chi}} \ast R_{\psi_{n_\chi}} \ast D;
$
\item if $\zeta(!w) = \psi_1\phi_1\ldots\psi_{n_\zeta}\phi_{n_\zeta}(!w)$ such that each $\phi_i$ (resp. $\psi_i$) is a finite concatenation of diamond (resp.\ box) connectives, of which only $\phi_{n_\zeta}$ is possibly empty, then $R_{\zeta}\in \mathsf{GRel}_{\mathcal{L}}$ is defined 
as follows: \\
if $\phi_{n_\zeta}$ is empty,
$
    R_{\zeta} \coloneqq R_{\psi_1} \ast R_{\phi_1} \ast \cdots \ast R_{\psi_{n_\zeta}}\hspace{-0.4mm};
$ \\
if $\phi_{n_\zeta}$ is nonempty,
$
    R_{\zeta} \coloneqq R_{\psi_1} \ast R_{\phi_1} \ast \cdots \ast R_{\psi_{n_\zeta}} \ast R_{\phi_{n_\zeta}} \ast E.
$
\end{enumerate}
\end{definition}
The first item of the following lemma shows that the semantic operations induced by the relations introduced  in Definition \ref{def:graph-based rels associated with terms} 
coincide with the term functions induced by their associated formulas. As is to be expected, this strong property does not fully extend to formulas of a more general shape such as $\chi$ and $\zeta$; however, a more limited version applies for singleton subsets. 
\marginnote{MP:  expand proof}
\begin{lemma}
\label{lemma:molecular polarity}
For any graph-based frame $\mathbb{F}$ based on $(Z,E)$, all terms $\varphi(!t)$, $\psi(!t)$, $\chi(!t)$, $\zeta(!t)$ as in Definition \ref{def:graph-based rels associated with terms}, all $B,Y\subseteq Z$, and all $a,x \in Z$,
\begin{enumerate}
    \item $\descr{\varphi(!t)}(t \coloneqq (B^{[10]}, B^{[1]}) =  R^{[0]}_{\varphi}[B]$ \ \ and \ \ $\val{\psi(!t)}(t \coloneqq Y = (Y^{[0]}, Y^{[01]})) = R^{[0]}_\psi[Y]$;
    \item $
    \descr{\chi[\nomj/!t]}(\nomj \coloneqq (a^{[10]}, a^{[1]})) = R^{[0]}_{\chi}[a^{[10]}]=  R^{[0]}_{\chi}[a];
    $
    \item $
    \val{\zeta[\cnomm/!t]}(\cnomm \coloneqq (x^{[0]},x^{[01]}) = R^{[0]}_{\zeta}[x^{[01]}] = R^{[0]}_{\zeta}[x].
    $
\end{enumerate}
\end{lemma}
\begin{proof}
Item (1) can be shown straightforwardly by induction on $\varphi$ and $\psi$, using Lemma \ref{lemma:properties of homogeneous composition} (2 and 3). Items (2) and (3) can be shown  by induction on the number of alternations of box and diamond connectives, using  item (1), and Definition \ref{def:relational composition}.
\end{proof}

As discussed in Section \ref{ssec:mrp}, the ALBA output of mrps of type (a) is
\begin{center}
    $\forall \nomj\left(\mathsf{LA}(\psi)[\phi[\nomj/!y]/!u] \leq \chi[\mathsf{LA}(\alpha)[\nomj/!u]/p]\right).$
\end{center}
When interpreted on graph-based frames, the condition above translates as follows:
\begin{equation}
\label{eq: alba output a translated graphs}
\forall a \left(\descr{\chi[\mathsf{LA}(\alpha)[\nomj/!u]/p] }(\nomj: =(a^{[10]},a^{[1]}))\subseteq \descr{\mathsf{LA}(\psi)[\phi[\nomj/!y]/!u]}(\nomj: =(a^{[10]}, a^{[1]})\right).
\end{equation}
Since $\mathsf{LA}(\psi)[\phi(!y)/!u]$ and $\chi[\mathsf{LA}(\alpha)/p]$ are of the shape required by
by items (1) and (2) of Lemma \ref{lemma:molecular polarity} respectively, the following identities hold:
\begin{center}
\begin{tabular}{rcccl}
$\descr{\mathsf{LA}(\psi)[\phi[\nomj/!y]/!u]}(\nomj: = (a^{[10]},a^{[1]} )$ & $=$ & $R_{\mathsf{LA}(\psi)[\phi/!u]}^{[0]}[a^{[10]}]$ & $=$ & $(R_{\mathsf{LA}(\psi)} \,\compdia R_{\phi})^{[0]}[a]$. \\
$\descr{\chi[\mathsf{LA}(\alpha)[\nomj/!u]/p] }(\nomj:= (a^{[10]},a^{[1]}))$ & $=$ & $R_{\chi[\mathsf{LA}(\alpha)/p]}^{[0]}[a^{[10]}]$ & $=$ & $R_{\chi[\mathsf{LA}(\alpha)/p]}^{[0]}[a]$.
\end{tabular}
\end{center}
Hence, \eqref{eq: alba output a translated graphs} can be rewritten as follows:
\begin{equation}
\label{eq: alba output a final graphs}
\forall a \left(R_{\chi[\mathsf{LA}(\alpha)/p]}^{[0]}[a]\subseteq (R_{\mathsf{LA}(\psi)} \,\compdia R_{\phi})^{[0]}[a]\right) \quad \text{i.e.} \quad R_{\mathsf{LA}(\psi)} \,\compdia R_{\phi} \subseteq R_{\chi[\mathsf{LA}(\alpha)/p]}
\end{equation}

Similarly, the ALBA output of mrps of type (b) can be represented as follows:
\begin{center}
$\forall \cnomm\left(\zeta[\mathsf{RA}(\delta)[\cnomm/!u]/p] \leq \mathsf{RA}(\phi)[\psi[\cnomm/!y]/!u]\right),$
\end{center}
and, when interpreted on graph-based frames, the condition translates as follows:
\begin{equation}
\label{eq: alba output b translated graphs}
\forall x\left (\val{\zeta[\mathsf{RA}(\delta)[\cnomm/!u]/p]}(\cnomm: = (x^{[0]},x^{[01]})) \subseteq \val{\mathsf{RA}(\phi)[\psi[\cnomm/!y]/!u]}(\cnomm: = (x^{[0]},x^{[01]}))\right ).
\end{equation}
By items (3) and (1) of Lemma \ref{lemma:molecular polarity}, 
\vspace{-4mm}
\begin{center}
\begin{tabular}{rcccl}
     $\val{\mathsf{RA}(\phi)[\psi[\cnomm/!y]/!u]}(\cnomm:= (x^{[0]},x^{[01]}))$ & $=$ &  $R_{\mathsf{RA}(\phi)[\psi/!u]}^{[0]}[x^{[01]}]$ & $=$ & $(R_{\mathsf{RA}(\phi)} \,\compbox R_{\psi})^{[0]}[x]$. \\
$\val{\zeta[\mathsf{RA}(\delta)[\cnomm/!u]/p] }(\cnomm:= (x^{[0]},x^{[01]}))$ &  $=$ &  $R_{\zeta[\mathsf{RA}(\delta)/p]}^{[0]}[x^{[01]}]$ & $=$ & $R_{\zeta[\mathsf{RA}(\delta)/p]}^{[0]}[x]$.
\end{tabular}
\end{center}
Therefore, we can rewrite \eqref{eq: alba output b translated graphs} as follows:
\begin{equation}
\label{eq: alba output b final graphs}
\forall x \left( R_{\zeta[\mathsf{RA}(\delta)/p]}^{[0]}[x]\subseteq (R_{\mathsf{RA}(\phi)} \,\compbox R_{\psi})^{[0]}[x]\right) \quad \text{i.e.} \quad R_{\mathsf{RA}(\phi)} \,\compbox R_{\psi} \subseteq R_{\zeta[\mathsf{RA}(\delta)/p]}.
\end{equation}

Finally, notice that, if the mrp is also analytic, the shape of $\chi(p)$ (resp.\ $\zeta(p)$) simplifies to $\chi(p) = \phi_{n_\chi} \psi_{n_\chi}(p)$ with $n_\chi = 1$ and $\psi_{n_\chi}$ empty 
(resp.\ $\zeta(p) = \psi_{n_\zeta} \phi_{n_\zeta}(p)$ with $n_\zeta = 1$ and $\phi_{n_\zeta}$ empty). 
Hence, \eqref{eq: alba output a final graphs}  and \eqref{eq: alba output b final graphs}  simplify to the following inclusions:
\begin{equation}
\label{eq final analytic}
R_{\mathsf{LA}(\psi)} \,\compdia R_{\phi} \subseteq R_{\phi_{n_\chi}}\, \compdia R_{\mathsf{LA}(\alpha)} \quad\quad R_{\mathsf{RA}(\phi)} \,\compbox R_{\psi} \subseteq R_{\psi_{n_\zeta}}\, \compbox R_{\mathsf{RA}(\delta)}.
\end{equation}

\begin{example}
\label{ex: p leq diamond box p}
The mrp $p\leq \Diamond \Box p$  is inductive of type (a), where $\phi (y): =  y$, and $\alpha (p): = p$, hence $\mathsf{LA}(\alpha)(u): = u$, and $\psi(x): =  x$, hence $\mathsf{LA}(\psi)(v): =  v$, and $\chi(p) := \Diamond \Box p$. Then the ALBA output of this inequality is 
\begin{center}
\begin{tabular}{rll}
&$\forall \nomj\left([[\nomj/y]/v] \leq \Diamond\Box[[\nomj/u]/p]\right),
$& \\
iff & $\forall \nomj\left(\nomj \leq \Diamond \Box\nomj\right)$ & \\
iff & $\forall \nomj \left(\descr{\Diamond \Box\nomj } \subseteq \descr{\nomj}\right)$ & (definition of order on polarities)\\
iff& $\forall a \left(R_{\Diamond}^{[0]}[R_{\Box}^{[0]}[E^{[1]}[a]]]\subseteq E^{[1]}[a]\right)$& (standard translation)\\
iff& $\forall a \left((R_{\Diamond}\, \ast R_{\Box}\, \ast  D)^{[0]}[a]\subseteq ({D})^{[0]}[a]\right)$& (Definition \ref{def:relational composition})\\
iff& $D \subseteq R_{\Diamond}\, \ast R_{\Box}\, \ast D.$ & (Lemma \ref{lemma: basic})
\end{tabular}
\end{center}
\end{example}

\begin{example}
The mrp $\Box\Diamond p \leq \Box \Diamond\Diamond p $  is inductive of shape (b) 
with $\phi(y): = y$, hence $\mathsf{RA}(\phi)(v): = v$ and $\psi (x): = \Box x$, and $\delta(p): = \Diamond\Diamond p$, hence $\mathsf{RA}(\delta)(u): = \blacksquare\blacksquare u$, and $\chi(p): = \Box\Diamond p$.
\begin{center}
			\begin{tabular}{r l l l}
				&$\forall p  [\Box\Diamond p \leq \Box \Diamond\Diamond p ]$\\
				iff& $\forall \cnomm  [\Box\Diamond \blacksquare\blacksquare \cnomm \le \Box \cnomm]$
				& ALBA output \\		
				i.e. &$\forall x\left (R^{[0]}_{\Box}[R^{[0]}_{\Diamond}[R_{\blacksquare}^{[0]}[ E^{[1]}[R_{\blacksquare}^{[0]}[x^{\downarrow \uparrow}] ]]]\subseteq R_{\Box}^{[0]}[x^{\downarrow \uparrow}]\right )$\\
	iff&$\forall x\left ( R^{[0]}_{\Box}[R^{[0]}_{\Diamond}[R_{\blacksquare}^{[0]}[E^{[1]}[ R_{\blacksquare}^{[0]}[x] ]]]]\subseteq R_{\Box}^{[0]}[x]\right )$ &  ($R_{\Box}$ and $R_{\Diamond}$ $E$-compatible)\\
	iff&$\forall x\forall a\left ( a\in  R^{[0]}_{\Box}[R^{[0]}_{\Diamond}[R_{\blacksquare}^{[0]}[E^{[1]}[ R_{\blacksquare}^{[0]}[x] ]]]]\Rightarrow a\in R_{\Box}^{[0]}[x]\right )$ &  \\
	
	iff&$\forall x\forall a\left ( a (R_{\Box}\, \ast R_{\Diamond}\, \ast (R_{\blacksquare}\, \compbox R_{\blacksquare}) )^c x\Rightarrow a R_{\Box}^c x\right )$ &  \\
	iff&$ R_{\Box} \subseteq R_{\Box}\, \ast R_{\Diamond}\, \ast (R_{\blacksquare}\, \compbox R_{\blacksquare}) $. &  \\
				\end{tabular}
		\end{center}
\end{example}

\subsection{Correspondents on Kripke frames}
\label{ssec: correspondents on kripke}
In the present subsection, which is an adaptation of \cite[Section 4.2]{conradie2022modal}, we specialize the discussion of the previous subsection to the environment of Kripke frames, so as to obtain a specific representation of the first-order correspondents of inductive mrps on Kripke frames as pure inclusions of binary relations on Kripke frames.  

\begin{definition}\label{def:kripke-based rels associated with terms}
Let us associate terms in $\mathsf{KRel}_{\mathcal{L}}$ with ${\mathcal{L}}^*$-terms of certain syntactic shapes as follows:
\begin{enumerate}
    \item if $\varphi(!w)$ (resp.~$\psi(!w)$) is a finite (possibly empty) concatenation of diamond (resp. box) connectives, let us define $R_{\phi}\in \mathsf{KRel}_{\mathcal{L}}$ (resp.~$R_{\psi}\in \mathsf{KRel}_{\mathcal{L}}$)  by induction on $\phi$ (resp.~$\psi$) as follows:\\
if $\phi: = y$, then $R_{\phi}: = \Delta$; \quad 
if $\phi: = f\phi'$ with $f \in \{\Diamond, \Diamondblack \}$, then $R_{\phi}: = R_f\circ R_{\phi'}$.\\
if $\psi: = y$, then $R_{\psi}: = \Delta$; \quad 
if $\psi: = g\psi'$ with $g \in \{\Box,\blacksquare\}$, then $R_{\psi}: = R_g\circ R_{\psi'}$.
\item if $\chi(!w) = \phi_1\psi_1\ldots\phi_{n_\chi}\psi_{n_\chi}(!w)$ such that each $\phi_i$ (resp. $\psi_i$) is a finite concatenation of diamond (resp.\ box) connectives, of which only $\psi_{n_\chi}$ is possibly empty. Let us define $R_{\chi}\in \mathsf{KRel}_{\mathcal{L}}$ by induction on $n_\chi$ as follows: \\
if $\psi_{n_\chi}$ is empty,
$
    R_{\chi} \coloneqq R_{\varphi_1} \star R_{\psi_1} \star \cdots \star R_{\varphi_{n_\chi}};
$\\
if $\psi_{n_\chi}$ is nonempty,
$
    R_{\chi} \coloneqq R_{\varphi_1} \star R_{\psi_1} \star \cdots \star R_{\varphi_{n_\chi}} \star R_{\psi_{n_\chi}} \star \Delta
$ .
\item if $\zeta(!w) = \psi_1\phi_1\ldots\psi_{n_\zeta}\phi_{n_\zeta}(!w)$ such that each $\phi_i$ (resp. $\psi_i$) is a finite concatenation of diamond (resp.\ box) connectives, of which only $\phi_{n_\zeta}$ is possibly empty. Let us define $R_{\zeta}\in \mathsf{KRel}_{\mathcal{L}}$ by induction on $n_\zeta$ as follows: \\
if $\phi_{n_\zeta}$ is empty,
$
    R_{\zeta} \coloneqq R_{\psi_1} \star R_{\phi_1} \star \cdots \star R_{\psi_{n_\zeta}};
$\\
if $\phi_{n_\zeta}$ is nonempty,
$
    R_{\zeta} \coloneqq R_{\psi_1} \star R_{\phi_1} \star \cdots \star R_{\psi_{n_\zeta}} \star R_{\phi_{n_\zeta}} \star \Delta
$ .
\end{enumerate}
\end{definition}
The  next lemma is a Kripke-frame analogue of Lemma \ref{lemma:molecular polarity}. 

\begin{lemma}
\label{lemma:molecular kripke}
For any Kripke frame $\mathbb{X}$ based on $W$, all terms $\varphi(!w)$, $\psi(!w)$, $\chi(!w)$, and $\zeta(!w)$ as in Definition \ref{def:kripke-based rels associated with terms}, any $S\subseteq W$, and any $a,x \in W$,
\begin{enumerate}
    \item $\val{\varphi(!w)}(w \coloneqq S) = \langle R_{\varphi} \rangle S \quad \text{and} \quad \val{\psi(!w)}(w \coloneqq S) = [R_\psi]S.$
    \item $\val{\chi[\nomj/!w]}(\nomj \coloneqq 
    \{a\}) = \langle R_{\chi} \rangle \{a\}.$
    \item $\val{\zeta[\cnomm/!w]}(\cnomm \coloneqq \{x\}^c) = [ R_{\zeta} ] \{x\}^c.$
\end{enumerate}
\end{lemma}
\begin{proof}
Item (1) is proved by induction on the complexity of $\varphi$ and $\psi$ respectively.
As to item (2), we proceed by induction on the number of alternations of boxes and diamonds in $\chi$. If $\chi = w$, $R_\chi = \Delta$, so the statement follows trivially. For the inductive case, $\chi\neq x$, and then $\chi = \varphi_1 \psi_1 \cdots \varphi_n \psi_n(!x)$ such that $n\geq 1$, and the $\varphi_i$ (resp.\ $\psi_i$) are finite, nonempty, concatenations of diamond (resp.\ box) operators, except for $\psi_n$ which is possibly empty. If $n=1$ and $\psi_1$ empty, then $R_\chi = R_{\varphi_1}$, hence the statement follows from item (1). If $n=1$ and $\psi_1(!z)$ is not empty, then, $R_\chi = R_{\varphi_1} \star R_{\psi_1} \star \Delta$, so:
\begin{center}
\begin{tabular}{rlr}
& $\val{\chi[\nomj/x]}(\nomj \coloneqq \{a\})$ & \\
$=$ & $\val{\varphi_1(!y)}(y\coloneqq \val{\psi_1[\nomj/z]}(\nomj \coloneqq \{a\}))$ & \\
$=$ & $\langle R_{\varphi_1} \rangle \val{\psi_1[\nomj/z]} (\nomj \coloneqq \{a\})$ & item (1) \\
$=$ & $\langle R_{\varphi_1} \rangle [R_{\psi_1}]\{a\} $ & item (1)\\
$=$ & $\langle R_{\varphi_1} \rangle R^{[0]}_{\psi_1}[\{a\}^c] $ & Lemma \ref{lemma:properties of square bracket superscript} \\
$=$ & $ (R^{[0]}_{\varphi_1}[ R^{[0]}_{\psi_1}[\{a\}^c]])^c $ & Lemma \ref{lemma:properties of square bracket superscript} \\
$=$ & $ (R^{[0]}_{\varphi_1}[ R^{[0]}_{\psi_1}[\Delta^{[0]}[a]]])^c $ & $\Delta^{[0]}[a]=\{a\}^c$ \\
$=$ & $((R_{\varphi_1} \star R_{\psi_1} \star \Delta)^{[0]}[a] )^c$ & Definition \ref{def:relational composition} \\
$=$ & $(R_{\chi}^{[0]}[a] )^c$ & $\chi=\varphi_1\psi_1$ \\
$=$ & $\langle R_{\chi} \rangle \{a\}$ 
& Lemma \ref{lemma:properties of square bracket superscript} \\
\end{tabular}
\end{center}
The proof of the inductive step and of item (3) is similar, and is omitted.
\end{proof}

As discussed in Section \ref{ssec:mrp}, the ALBA output of mrps of type (a) is
\begin{center}
    $\forall \nomj\left(\mathsf{LA}(\psi)[\phi[\nomj/!y]/!u] \leq \chi[\mathsf{LA}(\alpha)[\nomj/!u]/p]\right).$
\end{center}
When interpreted on Kripke frames, the condition above translates as follows:
\begin{equation}
\label{eq: alba output a translated kripke}
\forall a \left( \val{\mathsf{LA}(\psi)[\phi[\nomj/!y]/!u]}(\nomj: =\{a\}) \subseteq \val {\chi[\mathsf{LA}(\alpha)[\nomj/!u]/p] }(\nomj: = \{a\})\right).
\end{equation}
Since $\mathsf{LA}(\psi)[\phi(!y)/!u]$ and $\chi[\mathsf{LA}(\alpha)/p]$ are of the shape required by
by items (i) and (ii) of Lemma \ref{lemma:molecular kripke} respectively, \eqref{eq: alba output a translated kripke} can be rewritten as follows
\begin{equation}
\label{eq: alba output a final kripke}
\forall a \left((R_{\mathsf{LA}(\psi)} \,\circ R_{\phi})^{-1}[a] \subseteq R_{\chi[\mathsf{LA}(\alpha)/p]}^{-1)}[a]\right) \quad \text{i.e.} \quad R_{\mathsf{LA}(\psi)} \,\circ R_{\phi} \subseteq R_{\chi[\mathsf{LA}(\alpha)/p]}.
\end{equation}

Similarly, the ALBA output of mrps of type (b) is
\begin{center}
$\forall \cnomm\left(\zeta[\mathsf{RA}(\delta)[\cnomm/!u]/p] \leq \mathsf{RA}(\phi)[\psi[\cnomm/!y]/!u]\right),$
\end{center}
and, when interpreted on Kripke frames, the condition translates as follows:
\begin{equation}
\label{eq: alba output b translated kripke}
\forall x\left (\val{\zeta[\mathsf{RA}(\delta)[\cnomm/!u]/p]}(\cnomm: = \{x\}^c) \subseteq \val{\mathsf{RA}(\phi)[\psi[\cnomm/!y]/!u]}(\cnomm: = \{x\}^c) \right ).
\end{equation}
By items (3) and (1) of Lemma \ref{lemma:molecular kripke}, Lemma \ref{lemma:properties of square bracket superscript}, and contraposition, we can rewrite \eqref{eq: alba output b translated kripke} as follows:
\begin{equation}
\label{eq: alba output b final kripke}
\forall x \left((R_{\mathsf{RA}(\phi)} \,\circ R_{\psi})^{-1}[x] \subseteq  R_{\zeta[\mathsf{RA}(\delta)/p]}^{-1}[x] \right) \quad \text{i.e.} \quad R_{\mathsf{RA}(\phi)} \,\circ R_{\psi} \subseteq R_{\zeta[\mathsf{RA}(\delta)/p]}.
\end{equation}

Finally, notice that, if the mrp is also analytic, the shape of $\chi(p)$ (resp.\ $\zeta(p)$) simplifies to $\chi(p) = \phi_{n_\chi} \psi_{n_\chi}(p)$ with $n_\chi = 1$ and $\psi_{n_\chi}$ empty 
(resp.\ $\zeta(p) = \psi_{n_\zeta} \phi_{n_\zeta}(p)$ with $n_\zeta = 1$ and $\phi_{n_\zeta}$ empty). 
Hence, \eqref{eq: alba output a final kripke}  and \eqref{eq: alba output b final kripke}  simplify to the following inclusions:
\begin{equation}
\label{eq final analytic 2}
 R_{\mathsf{LA}(\psi)} \,\circ R_{\phi} \subseteq R_{\phi_{n_\chi}}\, \circ R_{\mathsf{LA}(\alpha)} \quad\quad R_{\mathsf{RA}(\phi)} \,\circ R_{\psi} \subseteq R_{\psi_{n_\zeta}}\, \circ R_{\mathsf{RA}(\delta)}.
\end{equation}

\begin{example}
\label{ex: p leq diamond box p Kripke}
The mrp $p\leq \Diamond \Box p$  is inductive of type (a), where $\phi (y): =  y$, and $\alpha (p): = p$, hence $\mathsf{LA}(\alpha)(u): = u$, and $\psi(x): =  x$, hence $\mathsf{LA}(\psi)(v): =  v$, and $\chi(p) := \Diamond \Box p$. Then the ALBA output of this inequality is
\begin{center}
\begin{tabular}{rll}
& $\forall \nomj\left([[\nomj/y]/v] \leq \Diamond\Box[[\nomj/u]/p]\right)$&\\
    iff & $\forall \nomj\left(\nomj \leq \Diamond \Box\nomj\right)$ & \\
    iff & $\forall \nomj \left(\descr{\Diamond \Box\nomj } \subseteq \descr{\nomj}\right) $ & (Definition of order on polarities)\\
    iff & $\forall a \left(R_{\Diamond}^{[0]}[R_\Box^{[0]}[a^c]] \subseteq a^c \right)$ & (Standard translation)\\
    iff & $\forall a \left(R_{\Diamond}^{[0]}[R_\Box^{[0]}[\Delta^{[0]}[a]]] \subseteq \Delta^{[0]}[a] \right)$ & (Definition of $\Delta$) \\
    iff & $ \forall a \left((R_{\Diamond}\, \star R_{\Box}\, \star  \Delta )^{[0]}[a]\subseteq \Delta ^{[0]}[a]\right)$& (Definition \ref{def:relational composition})\\
    iff & $ \forall x\forall a \left(x (R_{\Diamond}\, \star R_{\Box}\, \star  \Delta )^{c}a\subseteq x \Delta ^{c}[a]\right)$& (Definition of $(\cdot)^{[0]}$)\\
    iff& $\Delta \subseteq R_{\Diamond}\, \star R_{\Box}\, \star \Delta$& (Lemma \ref{lemma: basic})
\end{tabular}
\end{center}
\end{example}

\begin{example}
The mrp $\Box\Diamond p \leq \Box \Diamond\Diamond p $  is inductive of shape (b) 
with $\phi(y): = y$, hence $\mathsf{RA}(\phi)(v): = v$ and $\psi (x): = \Box x$, and $\delta(p): = \Diamond\Diamond p$, hence $\mathsf{RA}(\delta)(u): = \blacksquare\blacksquare u$, and $\chi(p): = \Box\Diamond p$.

\begin{center}
			\begin{tabular}{r l l l}
				&$\forall p  [\Box\Diamond p \leq \Box \Diamond\Diamond p ]$\\
				iff& $\forall \cnomm  [\Box\Diamond \blacksquare\blacksquare \cnomm \le \Box \cnomm]$
				
				& ALBA output\\		
				i.e. &$\forall x\left (R^{[0]}_{\Box}[R^{[0]}_{\Diamond}[R_{\blacksquare}^{[0]}[ \Delta^{[1]}[R_{\blacksquare}^{[0]}[x^{[01]]}] ]]]\subseteq R_{\Box}^{[0]}[x^{[01]}]\right )$\\
	iff&$\forall x\left ( R^{[0]}_{\Box}[R^{[0]}_{\Diamond}[R_{\blacksquare}^{[0]}[\Delta^{[1]}[ R_{\blacksquare}^{[0]}[x] ]]]]\subseteq R_{\Box}^{[0]}[x]\right )$ & ($R_{\Box}$ and $R_{\Diamond}$ $\Delta$-compatible)\\
	iff&$\forall x\forall a\left ( a\in  R^{[0]}_{\Box}[R^{[0]}_{\Diamond}[R_{\blacksquare}^{[0]}[\Delta^{[1]}[ R_{\blacksquare}^{[0]}[x] ]]]]\Rightarrow a\in R_{\Box}^{[0]}[x]\right )$ &  \\
	
	iff&$\forall x\forall a\left ( a (R_{\Box}\, \star R_{\Diamond}\, \star (R_{\blacksquare}\, \circ R_{\blacksquare}) )^c x\Rightarrow a R_{\Box}^c x\right )$ &  \\
	iff&$ R_{\Box} \subseteq R_{\Box}\, \star R_{\Diamond}\, \star (R_{\blacksquare}\, \circ R_{\blacksquare}) $. &  \\
				\end{tabular}
		\end{center}

\end{example}
\subsection{Shifting Kripke frames to graph-based frames}
\label{ssec:shifting}
In the two previous subsections, we have shown that, both in the setting of graph-based frames and in that of Kripke frames, the first-order correspondents of inductive mrps can be expressed as pure inclusions of (compositions,  pseudo compositions,  $E$-mediated compositions of)  binary relations associated with certain terms. Besides providing a more compact way to represent first-order sentences, 
this relational representation has also specific {\em correspondence-theoretic} consequences:   
it allows us to formulate and establish a systematic  connection between the first-order correspondents, in the two semantic settings, of any inductive mrp. Establishing this connection is the focus of the present subsection. Specifically, we show that, for any inductive mrp $s(p)\leq t(p)$,  its first-order correspondent on graph-based frames is  the ``shifted version" (in a sense that will be made precise below) of its first-order correspondent on Kripke frames. That is, thanks to the relational representation, the interesting phenomenon observed in    \cite[Proposition 5]{CONRADIE2021371} in the context of polarity-based frames and in \cite[Proposition 4]{graph-based-wollic} in the context of graph-based frames  can be generalized  (and made much more explicit in the process) to all inductive mrps. 

For any $\xi \in \mathsf{KRel}_{\mathcal{L}}$ (resp.~$\gamma \in \mathsf{GRel}_{\mathcal{L}}$), let  $\xi^{\mathbb{X}}$ (resp.~$\gamma^{\mathbb{F}}$) denote the interpretation of $\xi$ (resp.~$\gamma$) in $\mathbb{X}$ (resp.~$\mathbb{F}$).
\begin{definition}\label{def:shifting}
  The $\mathsf{GRel}_{\mathcal{L}}$-inequality $\gamma_1\subseteq \gamma_2$ is the {\em shifting} of the $\mathsf{KRel}_{\mathcal{L}}$-inequality $\xi_1\subseteq \xi_2$ if, for any Kripke frame $\mathbb{X}$,
\[\xi_1^\mathbb{X} = \gamma_1^{\mathbb{F_X}} \quad \text{ and }\quad
\xi_2^\mathbb{X} = \gamma_2^{\mathbb{F_X}}.\] 
\end{definition}

\begin{example}
\label{ex:erefl etrans}
The $\mathsf{GRel}_{\mathcal{L}}$-inequality $ E \subseteq R_{\Box} $  is the shifting of the $\mathsf{KRel}_{\mathcal{L}}$-inequality $\Delta\subseteq R_{\Box}$ (encoding the reflexivity of $R_\Box$). Indeed, $\Delta\subseteq R_{\Box}$ instantiates 
$ E \subseteq R_{\Box} $ on all graph-based frames with $E\coloneqq\Delta$, which are exactly the shifted Kripke frames.

The $\mathsf{GRel}_{\mathcal{L}}$-inequality $ R_{\Box}\, \compbox R_{\Box} \subseteq R_{\Box}$ is the shifting of the $\mathsf{KRel}_{\mathcal{L}}$-inequality $R_{\Box}\circ R_{\Box}\subseteq R_{\Box}$ (encoding the transitivity of $R_\Box$) on Kripke frames. Indeed, by Lemma \ref{lemma:compdia compbox delta}, $R_{\Box}\circ R_{\Box}\subseteq R_{\Box}$ is the instantiation of $ R_{\Box}\, \compbox R_{\Box} \subseteq R_{\Box}$  when $E \coloneqq \Delta$.
\end{example}

\begin{proposition}
\label{prop:shifting_via_translation}
Every $\mathsf{GRel}_\mathcal{L}$-inequality $\gamma_1 \subseteq \gamma_2$ is the shifting of the $\mathsf{KRel}_\mathcal{L}$-inequality $\tau(\gamma_1) \subseteq \tau(\gamma_2)$ (see Definition \ref{def:translationgrelkrel}).
\end{proposition}
\begin{proof}
Immediate by item 1 of Lemma \ref{lemma:translationgrel}.
\end{proof}

\begin{theorem}
\label{thm: sahlqvist shifting}
For any inductive mrp $s(p)\leq t(p)$, the $\mathsf{GRel}_{\mathcal{L}}$-inequality encoding its first-order correspondent  on graph-based frames  is the shifting of the $\mathsf{KRel}_{\mathcal{L}}$-inequality encoding its first-order correspondent  on Kripke frames.
\end{theorem}
\begin{proof}
By Proposition \ref{prop:shifting_via_translation}, it is enough to show that the $\mathsf{KRel}_{\mathcal{L}}$-inequality equivalent to the first order correspondent of $s(p) \leq t(p)$ on Kripke frames is the translation of the $\mathsf{GRel}_{\mathcal{L}}$-inequality equivalent to the first order correspondent of $s(p) \leq t(p)$ on graph-based frames. Such correspondents are (depending on the type of $s(p)\leq t(p)$):
\smallskip

{{\centering
\begin{tabular}{|c|c|c|}
\hline
Type & $\mathsf{KRel}_{\mathcal{L}}$-inequalities & $\mathsf{GRel}_{\mathcal{L}}$-inequalities \\
\hline
(a) & $R_{\mathsf{LA}(\psi)} \circ R_{\phi} \subseteq R_{\chi[\mathsf{LA}(\alpha)/p]}$ & $R_{\mathsf{LA}(\psi)} \,\compdia R_{\phi} \subseteq R_{\chi[\mathsf{LA}(\alpha)/p]}$ \\
(b) & $R_{\mathsf{RA}(\phi)} \,\circ R_{\psi} \subseteq R_{\zeta[\mathsf{RA}(\delta)/p]}$ & $R_{\mathsf{RA}(\phi)} \,\compbox R_{\psi} \subseteq R_{\zeta[\mathsf{RA}(\delta)/p]}$ \\
\hline
\end{tabular}
\par }}
\smallskip

\noindent where the notation $R_\varphi, R_\psi, R_\chi, R_\zeta$ refers to Definition \ref{def:graph-based rels associated with terms} in the graph-based context, and Definition \ref{def:kripke-based rels associated with terms} in the Kripke one. Therefore, it is sufficient to prove that the $\mathsf{KRel}_{\mathcal{L}}$-terms associated to formulas in Definition \ref{def:kripke-based rels associated with terms} coincide with the translations of the terms associated with the same formulas in Definition \ref{def:graph-based rels associated with terms}. This follows straightforwardly by induction and by the definition of the translation (cf.\ Definition \ref{def:translationgrelkrel}).
\end{proof}

In what follows, we will sometimes refer to the   shiftings of $\mathsf{KRel}_\mathcal{L}$-inequalities corresponding to Sahlqvist mrps of type (a) (resp.~(b)) as {\em $D$-shiftings} (resp.~{\em $E$-shiftings}). 

\begin{example}
The mrp $p\leq \Diamond \Box p$  is inductive of type (a), where $\phi (y): =  y$, and $\alpha (p): = p$, hence $\mathsf{LA}(\alpha)(u): = u$, and $\psi(x): = x$, hence $\mathsf{LA}(\psi)(v): =  v$, and $\chi(p) := \Diamond \Box p$, and, as discussed in Example \ref{ex: p leq diamond box p Kripke} , its first-order correspondent on classical Kripke frames can be represented as the following $\mathsf{KRel}$-inequality: $\Delta \subseteq R_{\Diamond}\, \star R_{\Box}\, \star \Delta$, while as discussed in Example \ref{ex: p leq diamond box p}, its first-order correspondent on graph-based frames can be represented as the following $\mathsf{GRel}$-inequality: $D \subseteq R_{\Diamond}\, \ast R_{\Box}\, \ast  D$. Since $\tau(D) = \Delta$ and $\tau(R_{\Diamond}\, \ast R_{\Box}\, \ast  D) = R_{\Diamond}\, \star R_{\Box}\, \star \Delta$, By Proposition \ref{prop:shifting_via_translation}, $D \subseteq R_{\Diamond}\, \ast R_{\Box}\, \ast  D$  is a {\em shifting} (particularly $D$-{\em shifting})  of $\Delta \subseteq R_{\Diamond}\, \star R_{\Box}\, \star \Delta$.
\end{example}

We conclude the present subsection with the following example in which all the steps from the ALBA output and the generation of the inequalities are presented.
\begin{example} \label{example:BoxDiamondBox}
The mrp $\Diamond p\leq \Box\Diamond \Box p$  is inductive of type (a), where $\phi (y): = \Diamond y$, and $\alpha (p): = p$, hence $\mathsf{LA}(\alpha)(u): = u$, and $\psi(x): = \Box x$, hence $\mathsf{LA}(\psi)(v): = \Diamondblack v$, and $\chi(p) := \Diamond \Box p$. 
Let us compute the $\mathsf{GRel}_{\mathcal{L}}$-inequality representing the first order correspondent of  this mrp on graph-based frames: 
\begin{center}
    \begin{tabular}{rll}
    &$\forall \nomj\left(\Diamondblack[\Diamond[\nomj/y]/v] \leq \Diamond\Box[[\nomj/u]/p]\right)$&\\
    iff & $\forall \nomj\left(\Diamondblack\Diamond\nomj \leq \Diamond \Box\nomj\right)$& \\
    iff & $\forall \nomj \left(\descr{\Diamond \Box\nomj } \subseteq \descr{\Diamondblack\Diamond\nomj}\right) $ & (Definition of order on polarities)\\
    iff & $ \forall a\left(R_{\Diamond}^{[0]}[R_{\Box}^{[0]}[E^{[1]}[a]]]\subseteq R_{\Diamondblack}^{[0]}[E^{[1]}[R_{\Diamond}^{[0]}[a]]]\right)$ & (Standard translation)\\
    iff & $\forall a \left((R_{\Diamond}\, \ast R_{\Box}\, \ast  D)^{[0]}[a]\subseteq (R_{\Diamondblack}\,\compdia R_{\Diamond})^{[0]}[a]\right)$& (Definition \ref{def:relational composition}) \\
    iff& $ R_{\Diamondblack}\, \compdia R_{\Diamond}\subseteq R_{\Diamond}\, \ast R_{\Box}\, \ast D$. & 
    (Lemma \ref{lemma: basic})
    \end{tabular}
\end{center}
Let us now compute the $\mathsf{KRel}_{\mathcal{L}}$-inequality equivalent to the first order correspondent of the mrp on Kripke frames:
\begin{center}
    \begin{tabular}{rll}
   & $\forall \nomj\left(\Diamondblack[\Diamond[\nomj/y]/v] \leq \Diamond\Box[[\nomj/u]/p]\right)$&\\
        iff & $\forall \nomj\left(\Diamondblack\Diamond\nomj \leq \Diamond \Box\nomj\right)$& \\
        iff& $\forall \nomj \left(\descr{\Diamond \Box\nomj } \subseteq \descr{\Diamondblack\Diamond\nomj}\right)$ & (Definition of order on polarities)\\
        iff & $\forall a \left(R_{\Diamond}^{[0]}[R_{\Box}^{[0]}[\Delta^{[0]}[a]]]\subseteq R_{\Diamondblack}^{[0]}[\Delta^{[1]}[R_{\Diamond}^{[0]}[a]]]\right)$ &(Standard translation)\\
        iff& $ \forall a \left((R_{\Diamond}\, \star R_{\Box}\, \star \Delta )^{[0]}[a]\subseteq (R_{\Diamondblack}\,\circ R_{\Diamond})^{[0]}[a]\right)$ & (Definition \ref{def:relational composition})\\
        iff& $ R_{\Diamondblack}\, \circ R_{\Diamond} \subseteq R_{\Diamond}\, \star R_{\Box}\, \star  \Delta$. & (Lemma \ref{lemma: basic})
    \end{tabular}
\end{center}
The $\mathsf{GRel}$-inequality
$R_\Diamondblack \compdia R_{\Diamond}\,  \subseteq R_\Diamond \ast R_{\Box}\, \ast  D$ is the translation (and $D$-shifting) of $R_\Diamondblack \circ R_{\Diamond}\, \subseteq  R_\Diamond \star R_{\Box}\, \star \Delta$.
\end{example}


\subsection{Lifting graph-based frames to polarity-based frames}
\label{ssec:lifting}
Having established a systematic connection between the first-order correspondents of inductive mrps on Kripke frames and on graph-based frames in the previous subsection, in the present subsection we establish analogous results along the  following {\em lifting} construction of graph-based frames to polarity-based frames. 

 \begin{definition}[Lifting of graph-based frames]
  \label{def:lifting graph frames}
 For any reflexive graph $\mathbb{G} = (Z, E)$ and any graph-based frame $\mathbb{F} = (\mathbb{G}, R_\Diamond, R_\Box)$,  let the {\em lifting} of $\mathbb{F}$ be the polarity-based frame $\mathbb{P}_{\mathbb{F}}: = (\mathbb{P_G}, J_{R_\Diamond^c}, {I}_{R_{\Box}^c})$, where $\mathbb{P_G}\coloneqq (Z_A, Z_X, I_{E^c})$, $J_{R_\Diamond^c}$,  and ${I}_{R_{\Box}^c}$ are defined as in Section \ref{ssec:logics}.
 \end{definition}
 \begin{lemma}
     For any reflexive graph $\mathbb{G} = (Z, E)$ and any $E$-compatible relation $R$,
     the lifted relations $I_{R^c}$ and $J_{R^c}$ are $I_{E^c}$-compatible.
 \end{lemma}
 \begin{proof}
 We show just one of the two properties of $I_{E^c}$-compatibility, as the other is proved similarly.
 \smallskip
 
 {{\centering
 \begin{tabular}{rlr}
 &$ I_{E^c}^{(0)}[I_{E^c}^{(1)}[(I_{R^c}^{(0)}[x])]]$ & \\[1mm]
 $=$ & $ I_{E}^{[0]}[I_{E}^{[1]}[(I_{R}^{[0]}[x])]]$ & Definition of $(\cdot)^{[0]}$ and $(\cdot)^{[1]}$\\ 
 $=$ & $ (E^{[0]}[E^{[1]}[(R^{[0]}[x])]])_A$ & Definition of $I_E$ and $I_R$\\ 
 $=$ & $((R^{[0]}[x])^{[10]})_A$ & notational convention \\
 $\subseteq$ &  $(R^{[0]}[x])_A$ & $E$-compatibility of $R$ \\
 $=$ &  ${I_R}^{[0]}[x]$ & Definition of $I_R$ \\
 $=$ &  ${I_{R^c}}^{(0)}[x]$ & Definition of $(\cdot)^{[0]}$ \\
 \end{tabular}
 \par}}
 \end{proof}
By the lemma above,  $\mathbb{P_F}$ is a well defined polarity-based frame, and, by definition, the complex algebras $\mathbb{F_P}^+$ and $\mathbb{F}^+$ coincide. 
Notice that the lifting of Kripke frames to polarity-based frames defined in \cite[Definition 4.17]{conradie2022modal} is the composition of the shifting defined in Section \ref{ssec:shifting} and the lifting defined above.
In \cite[Section 4.1]{conradie2022modal}, the relational language $\mathsf{PRel}_\mathcal{L}$ (see \cite[Definition 3.20]{conradie2022modal}) has been introduced, which contains terms of types $\mathsf{T}_{A\times X}$, $\mathsf{T}_{X\times A}$, $\mathsf{T}_{A\times A}$ and $\mathsf{T}_{X\times X}$, and is defined by the following simultaneous recursions:
\smallskip

{{\centering
\begin{tabular}{rcl c lcl}
$\mathsf{T}_{A\times X} \ni \beta$ & $:: =$ &\, $I \mid R_{\Box}\mid \beta\, ;_I \beta \mid \rho\, ; \beta \mid \beta ; \lambda$ &\quad\quad& 
$\mathsf{T}_{A\times A}\ni \rho$ & $:: =$ &\,  $\delta \, ;\beta$ \\
$\mathsf{T}_{X\times A}\ni \delta$ & $:: = $ &\, $J \mid R_{\Diamond} \mid \delta\, ;_I \delta\mid  \lambda\, ; \delta \mid \delta ; \rho$ &&
$\mathsf{T}_{X\times X}\ni \lambda$ & $:: =$ &\, $\beta\, ; \delta $,
\end{tabular}
\par}}
\smallskip

\noindent where $I, J, R_\Diamond,$ and $ R_\Box$ are naturally interpreted as the $I, J, R_\Diamond,$ and $ R_\Box$ of a polarity-based frame, while $;$ and $;_I$ are interpreted as the composition operations in \cite[Definition 3.2]{conradie2022modal}.
Moreover, relational $\mathsf{PRel}_{\mathcal{L}}$-terms $R_\varphi, R_\psi, R_\chi, R_\zeta$ are associated with formulas $\varphi, \psi, \chi, \zeta$ as in Definition \ref{def:graph-based rels associated with terms}. 

In \cite[Section 4.1]{conradie2022modal},  the first-order correspondents of inductive mrps on polarity-based frames are shown to be  representable as term inequalities in $\mathsf{PRel}_{\mathcal{L}}$ as indicated below:
\smallskip

{{\centering 
\begin{tabular}{cl}
type (a) \quad & $R_{\chi[\mathsf{LA}(\alpha)/p]}\subseteq R_{\mathsf{LA}(\psi)} \,;_I R_{\phi}$ of type $\mathsf{T_{X\times A}}$, \\
type (b) \quad & $R_{\zeta[\mathsf{RA}(\delta)/p]}\subseteq R_{\mathsf{RA}(\phi)} \,;_I R_{\psi}$  of type $\mathsf{T_{A\times X}}$.
\end{tabular}
\par}}
\smallskip

\noindent The following definition extends the definition of lifting in \cite[Definition 4.17]{conradie2022modal} to graph-based frames.

\begin{definition}
\label{def:ij lifting}
Let $\xi_1\subseteq \xi_2$   be a $\mathsf{GRel}_{\mathcal{L}}$-inequality, and $\beta_1\subseteq \beta_2$ (resp.~$\delta_1\subseteq \delta_2$)  be a $\mathsf{PRel}_{\mathcal{L}}$-inequality of type $\mathsf{T}_{A\times X}$ (resp.~$\mathsf{T}_{X\times A}$).  
\begin{enumerate}
\setlength{\itemsep}{0.2pt}
\setlength{\parskip}{0pt}
\setlength{\parsep}{0pt}
\item The inequality $\beta_1\subseteq \beta_2$ is the $I$-{\em lifting} of $\xi_1\subseteq \xi_2$ if for any graph-based frame $\mathbb{F}$,
$\beta_1^\mathbb{P_F} = I_{(\xi_2^\mathbb{F})^c}$ and
$\beta_2^\mathbb{P_F} = I_{(\xi_1^\mathbb{F})^c}$.
\item The inequality $\delta_1\subseteq \delta_2$ is the $J$-{\em lifting} of $\xi_1\subseteq \xi_2$ if for any graph-based frame $\mathbb{F}$,
$\delta_1^\mathbb{P_F} = J_{(\xi_2^\mathbb{F})^c}$
and $\delta_2^\mathbb{P_F} = J_{(\xi_1^\mathbb{F})^c}$.
\end{enumerate}
\end{definition}

The following theorem extends \cite[Theorem 4.20]{conradie2022modal} to graph-based frames.
\begin{theorem}
\label{prop: sahlqvist lifting to polarities}
For any inductive modal reduction principle  $s(p)\leq t(p)$ of $\mathcal{L}$,
\begin{enumerate}
\item if $s(p)\leq t(p)$ is of type (a), then  the $\mathsf{PRel}_{\mathcal{L}}$-inequality of type $\mathsf{T}_{X\times A}$ encoding its first-order correspondent  on polarity-based frames  is the $J$-lifting of the $\mathsf{GRel}_{\mathcal{L}}$-inequality encoding its first-order correspondent  on graph-based frames;
\item if $s(p)\leq t(p)$ is of type (b), then the $\mathsf{PRel}_{\mathcal{L}}$-inequality of type $\mathsf{T}_{A\times X}$ encoding its first-order correspondent  on polarity-based frames  is the $I$-lifting of the $\mathsf{GRel}_{\mathcal{L}}$-inequality encoding its first-order correspondent  on graph-based frames.
\end{enumerate}
\end{theorem}
\begin{proof}
Analogous to that of \cite[Theorem 4.20]{conradie2022modal}.
\end{proof}

\begin{example}
The $\mathsf{PRel}_{\mathcal{L}}$-inequality $R_\Box \subseteq I$ is the $I$-lifting of the $\mathsf{GRel}_\mathcal{L}$-inequality 
$E \subseteq R_{\Box}$.
Indeed,  $\mathbb{F}\models E \subseteq R_\Box$ iff $I_{R_\Box^c}\subseteq I_{E^c}$ iff $\mathbb{P_F}\models R_\Box \subseteq I$ for every graph-based frame $\mathbb{F}$. These inequalities encode first-order conditions corresponding to the inductive mrp $\Box p \leq p$ on polarity-based frames and on graph-based frames, respectively.

Similarly, the $\mathsf{PRel}_\mathcal{L}$-inequality $R_\Box \subseteq R_\Box \, ;_I R_\Box$ is the $I$-lifting of $R_\Box \compbox R_\Box \subseteq R_\Box$, 
Indeed,   $\mathbb{F}\models R_\Box \compbox R_\Box \subseteq R_\Box$ iff $I_{R_\Box^c}\subseteq I_{(R_\Box \compbox R_\Box)^c}$ iff $I_{R_\Box^c}\subseteq I_{R_\Box^c} \, ;_I \, I_{R_\Box^c}$ iff $\mathbb{P_F}\models R_\Box \subseteq R_\Box \, ;_I R_\Box$ for every graph-based frame $\mathbb{F}$. The second equivalence hinges on the identity $I_{(S \compbox T)^c} = I_{S^c} \, ;_I \, I_{T^c}$ for all $S, T \subseteq Z \times Z$, which is proved similarly to \cite[Lemma 3.12]{conradie2022modal}. These inequalities encode first-order conditions corresponding to the inductive mrp $\Box p \leq \Box\Box p$.
\end{example}

\section{Applications to rough set theory}\label{sec:Applications to rough set theory} 
In the present section, we take stock of the technical  results presented in the previous sections by  further elucidating the relationship between hyperconstructive approximation spaces (cf.~Definition \ref{def:graph based approximation spaces}) and  generalized approximation spaces  
(seen as Kripke frames with $R_\Diamond =R_\Box$, cf.~Definition \ref{def: gen approx space}). 
Firstly, as is expected, the definition of hyperconstructivist approximation spaces collapses to that of generalized approximation spaces when instantiating $E\coloneqq\Delta$ (cf.\ Proposition \ref{prop:fxf_xfx}). Thus,
generalized  approximation spaces  form the proper subclass  of hyperconstructivist approximation spaces   capturing the limit case in which there is no inherent entropy and $R_\Box=R_\blacksquare=R$. On this subclass, {\em all}  $\mathsf{GRel}_{\mathcal{L}}$-inequalities collapse to  $\mathsf{KRel}_{\mathcal{L}}$-inequalities, since the relational compositions on graph-based frames (cf.~Definition \ref{def:relational composition}) reduce to the relational compositions on Kripke frames (cf.~Lemma \ref{lemma:translationgrel}). Hence, in particular, and as is expected, 
the correspondents of  Sahlqvist  mrps 
on graph-based frames,  represented as discussed in Section \ref{ssec: correspondents on graph-based},  collapse, in the sense specified in Lemma \ref{lemma:translationgrel}, to the correspondents of the same mrps on Kripke frames,  represented as discussed in Section \ref{ssec: correspondents on kripke}. 

More interesting is the relationship in the {\em converse direction}, which hinges on the notion of shifting  (cf.~Definition \ref{def:shifting}): every Sahlqvist mrp $s(p)\leq t(p)$ (cf.~Section \ref{ssec:mrp}) defines an elementary (i.e.~first-order definable) class $\mathbf{K}(s(p)\leq t(p))$ of Kripke frames and an elementary class $\mathbf{G}(s(p)\leq t(p))$ of graph-based frames. By Theorem \ref{thm: sahlqvist shifting},  when represented as a $\mathsf{GRel}_{\mathcal{L}}$-inequality,  the first-order condition  defining $\mathbf{G}(s(p)\leq t(p))$ is the {\em shifting}  of the first-order condition defining $\mathbf{K}(s(p)\leq t(p))$, represented as a $\mathsf{KRel}_{\mathcal{L}}$-inequality. 
%
The notion of shifting allows for  a systematic relationship to be established between $\mathbf{K}(s(p)\leq t(p))$ and $\mathbf{G}(s(p)\leq t(p))$ 
which is formulated in terms of their own first-order languages, and hence  {\em independently from their correspondence-theoretic link}, in the sense that this relationship can be established purely on the basis of the syntax of $\mathbf{K}(s(p)\leq t(p))$ and $\mathbf{G}(s(p)\leq t(p))$.
A similarly systematic relationship between the two correspondents of any Sahlqvist mrps on graph-based frames and on polarity-based frames, respectively, is established in Theorem \ref{prop: sahlqvist lifting to polarities} by hinging on the notion of lifting (cf.~Definition \ref{def:ij lifting}). This relationship is again independent from the correspondence-theoretic link of the two first-order conditions. 

It is easy to check that the well known modal axioms 
listed in Table \ref{table 3} are all  Sahlqvist mrps, and hence Theorem \ref{thm: sahlqvist shifting} and Theorem \ref{prop: sahlqvist lifting to polarities} apply to these axioms. Table \ref{table 3} also reports the relational correspondents of these mrps on Kripke frames, graph-based frames, and  polarity-based frames. It is easy to verify that, for each mrp in the table (they  are all analytic inductive, and hence of both types (a) and (b)), its relational correspondent on graph-based frames when considering it as type (a) (resp.~(b)) is the $D$-shifting (resp.~$E$-shifting) of its relational correspondent on Kripke frames, and its relational correspondent on polarity-based frames when considering it as type (a) (resp.~(b)) is the $J$-lifting (resp.~$I$-lifting) of its relational correspondent on graph-based frames.  
\begin{table}[!!h] \label{table:properties}
\begin{center}
\begin{tabular}{|c|c|c|c|c|}
\hline
\textbf{Property} & \textbf{Axiom}&  \textbf{$\mathsf{KRel}_{\mathcal{L}}$} correspondent & \textbf{$\mathsf{GRel}_{\mathcal{L}}$} correspondent & \textbf{$\mathsf{PRel}_{\mathcal{L}}$} correspondent \\
\hline
\multirow{4}{*}{\textbf{Reflexivity}}  & \multirow{2}{*}{$p\leq \Diamond p$} &\cellcolor[HTML]{EFEFEF} $\Delta \subseteq
R_{\Diamond}$  &\cellcolor[HTML]{EFEFEF}  $D \subseteq R_{\Diamond}  $ &\cellcolor[HTML]{EFEFEF}  $ R_\Diamond \subseteq J$  \\ \cline{3-5}
& & $\Delta\subseteq R_{\blacksquare}$ & $E \subseteq R_{\blacksquare} $ & $R_{\blacksquare} \subseteq I$ \\ \cline{2-5}
& \multirow{2}{*}{$\Box p\leq p$ } &\cellcolor[HTML]{EFEFEF} $\Delta\subseteq R_{\Diamondblack}$ &\cellcolor[HTML]{EFEFEF} $D \subseteq R_{\Diamondblack}$ &\cellcolor[HTML]{EFEFEF} $R_{\Diamondblack} \subseteq J$\\ \cline{3-5}
& &$\Delta \subseteq R_{\Box}$  & $E \subseteq R_{\Box}  $ & $R_\Box \subseteq I$ \\ \hline
\multirow{4}{*}{\textbf{Transitivity}}  & \multirow{2}{*}{$\Diamond\Diamond p\leq \Diamond p$} &\cellcolor[HTML]{EFEFEF} $R_{\Diamond}\circ R_{\Diamond}\subseteq R_{\Diamond}$  &\cellcolor[HTML]{EFEFEF} $  R_{\Diamond}\, \compdia \, R_{\Diamond} \subseteq R_{\Diamond}$ &\cellcolor[HTML]{EFEFEF} $   R_{\Diamond} \subseteq R_\Diamond \, ;_I R_\Diamond$ \\ \cline{3-5} 
&&$R_{\blacksquare}\circ R_{\blacksquare}\subseteq R_{\blacksquare}$& $R_{\blacksquare}\, \compbox \, R_{\blacksquare} \subseteq  R_{\blacksquare}$ & $   R_{\blacksquare} \subseteq R_\blacksquare \, ;_I R_\blacksquare$ \\
\cline{2-5}
& \multirow{2}{*}{$ \Box p\leq \Box\Box p$} &\cellcolor[HTML]{EFEFEF} $R_{\Diamondblack}\circ R_{\Diamondblack}\subseteq R_{\Diamondblack}$&\cellcolor[HTML]{EFEFEF} $ R_{\Diamondblack}\, \compdia \, R_{\Diamondblack}  \subseteq R_{\Diamondblack}$&\cellcolor[HTML]{EFEFEF} $ R_{\Diamondblack} \subseteq R_{\Diamondblack} \, ;_I R_{\Diamondblack}$\\ \cline{3-5} 
&& $R_{\Box}\circ R_{\Box}\subseteq R_{\Box}$  & $ R_{\Box}\, \compbox \, R_{\Box} \subseteq R_{\Box}$ & $R_\Box \subseteq R_\Box \, ;_I  R_\Box$\\
\hline
\multirow{4}{*}{\textbf{Symmetry}} & \multirow{2}{*}{$p \leq \Box\Diamond p$} &\cellcolor[HTML]{EFEFEF} $R_\Diamondblack \subseteq R_\Diamond$&\cellcolor[HTML]{EFEFEF} $R_\Diamondblack \subseteq R_\Diamond  $&\cellcolor[HTML]{EFEFEF} $R_\Diamond \subseteq R_\Diamondblack $\\ \cline{3-5}  
&&$R_{\Box}\subseteq R_{\blacksquare}$ & $R_{\Box}\subseteq  R_{\blacksquare}$ & $R_{\blacksquare} \subseteq R_{\Box} $\\
\cline{2-5}
&\multirow{2}{*}{$\Diamond \Box p \leq p$ }&\cellcolor[HTML]{EFEFEF} $R_{\Diamond}\subseteq R_{\Diamondblack}$ &\cellcolor[HTML]{EFEFEF} $R_{\Diamond} \subseteq R_{\Diamondblack} $ &\cellcolor[HTML]{EFEFEF} $R_{\Diamondblack} \subseteq R_{\Diamond} $\\
\cline{3-5}  
&&$R_{\blacksquare}\subseteq R_{\Box}$ & $R_{\blacksquare} \subseteq R_{\Box} $ & $R_\Box \subseteq R_\blacksquare$\\
\hline
\multirow{2}{*}{\textbf{Seriality }} & \multirow{2}{*}{$\Box p \leq \Diamond p$} &\cellcolor[HTML]{EFEFEF} $\Delta \subseteq R_\Diamond \circ R_\Diamondblack$ &\cellcolor[HTML]{EFEFEF} $D \subseteq R_\Diamond \compdia R_\Diamondblack  $ &\cellcolor[HTML]{EFEFEF} $R_\Diamond \, ;_I R_\Diamondblack \subseteq J $\\ \cline{3-5}
& & $\Delta\subseteq R_{\Box}\circ R_{\blacksquare}$ &  $E \subseteq R_{\Box}\, \compbox \, R_{\blacksquare}$  & $R_\Box \, ;_I R_\blacksquare \subseteq I $
 \\ 
\hline
\multirow{2}{*}{\textbf{Partial functionality}} &\multirow{2}{*}{$\Diamond p \leq \Box p$} &\cellcolor[HTML]{EFEFEF} $R_{\Diamondblack}\circ R_{\Diamond}\subseteq \Delta$ &\cellcolor[HTML]{EFEFEF} $ R_{\Diamondblack}\, \compdia \, R_{\Diamond}\subseteq D$ &\cellcolor[HTML]{EFEFEF} $ J \subseteq  R_{\Diamondblack} \, ;_I R_{\Diamond}$
\\ \cline{3-5}
& & $R_{\blacksquare}\circ R_{\Box}\subseteq \Delta$ &  $R_{\blacksquare}\, \compbox \, R_{\Box}\subseteq E$  & $ I \subseteq  R_{\blacksquare} \, ;_I R_{\Box}$\\ \hline
\multirow{4}{*}{\textbf{Euclideanness}} & \multirow{2}{*}{$\Diamond p \leq \Box\Diamond p$}
&\cellcolor[HTML]{EFEFEF}$R_\Diamondblack \circ R_\Diamond \subseteq R_\Diamond$ &\cellcolor[HTML]{EFEFEF}  $R_\Diamondblack\, \compdia \,    R_\Diamond\subseteq R_\Diamond $&\cellcolor[HTML]{EFEFEF}  $R_\Diamond \subseteq R_\Diamondblack \, ;_I    R_\Diamond $\\ 
\cline{3-5}  
&&$R_{\blacksquare}\circ R_{\Box}\subseteq R_{\blacksquare}$ & $R_{\blacksquare}\, \compbox \, R_{\Box} \subseteq R_{\blacksquare}  $&$R_\blacksquare \subseteq R_\blacksquare \, ;_I    R_\Box $\\
\cline{2-5}
&\multirow{2}{*}{$\Diamond\Box p\leq \Box p$} &\cellcolor[HTML]{EFEFEF}$R_\Diamondblack \circ R_\Diamond \subseteq R_\Diamondblack$ &\cellcolor[HTML]{EFEFEF}  $R_\Diamondblack\, \compdia \,    R_\Diamond  \subseteq R_\Diamondblack$&\cellcolor[HTML]{EFEFEF}  $R_\Diamondblack \subseteq R_\Diamondblack \, ;_I R_\Diamond$\\ 
\cline{3-5}  
&&$R_\blacksquare \circ R_\Box \subseteq R_\Box$ & $ R_\blacksquare\, \compbox \,  R_\Box\subseteq R_\Box $ &$R_\Box \subseteq R_\blacksquare \, ;_I R_\Box$\\
 \hline
\multirow{2}{*}{\textbf{Confluence}} & \multirow{2}{*}{$\Diamond\Box p\leq \Box\Diamond p$} &\cellcolor[HTML]{EFEFEF} $R_\Diamondblack \circ R_\Diamond \subseteq R_\Diamond \circ R_\Diamondblack$ &\cellcolor[HTML]{EFEFEF} $R_\Diamondblack \, \compdia \, R_\Diamond \subseteq R_\Diamond \, \compdia\, R_\Diamondblack   $   &\cellcolor[HTML]{EFEFEF}  $ R_\Diamond \, ;_I R_\Diamondblack \subseteq R_\Diamondblack \, ;_I R_\Diamond $
\\ \cline{3-5}
& & $R_\blacksquare \circ R_\Box \subseteq R_\Box \circ R_\blacksquare$ & $R_\blacksquare \, \compbox \, R_\Box \subseteq R_\Box \, \compbox\, R_\blacksquare $ & $R_\Box \, ;_I R_\blacksquare \subseteq R_\blacksquare \, ;_I R_\Box $
\\ \hline
\multirow{4}{*}{\textbf{Denseness}}  & \multirow{2}{*}{$\Diamond p\leq \Diamond\Diamond p$} &\cellcolor[HTML]{EFEFEF} $R_{\Diamond}\subseteq R_{\Diamond}\circ R_{\Diamond}$  &\cellcolor[HTML]{EFEFEF} $R_{\Diamond} \subseteq R_{\Diamond}\, \compdia \, R_{\Diamond} $ &\cellcolor[HTML]{EFEFEF} $ R_{\Diamond} \, ;_I R_{\Diamond} \subseteq R_{\Diamond} $ \\ \cline{3-5} 
&&$R_{\blacksquare}\subseteq R_{\blacksquare}\circ R_{\blacksquare}$& $R_{\blacksquare} \subseteq R_{\blacksquare}\, \compbox \, R_{\blacksquare} $ & $ R_{\blacksquare} \, ;_I R_{\blacksquare}\subseteq R_{\blacksquare}$\\
\cline{2-5}
& \multirow{2}{*}{$ \Box \Box p\leq \Box p$} &\cellcolor[HTML]{EFEFEF} $R_{\Diamondblack}\subseteq R_{\Diamondblack}\circ R_{\Diamondblack}$  &\cellcolor[HTML]{EFEFEF} $  R_{\Diamondblack}  \subseteq R_{\Diamondblack}\, \compdia \, R_{\Diamondblack}$ &\cellcolor[HTML]{EFEFEF} $ R_{\Diamondblack} \, ;_I R_{\Diamondblack}  \subseteq R_{\Diamondblack} $ \\ \cline{3-5} 
&&$R_{\Box}\subseteq R_{\Box}\circ R_{\Box}$& $R_{\Box} \subseteq R_{\Box}\, \compbox \, R_{\Box}$ &$R_{\Box} \, ;_I R_{\Box} \subseteq R_\Box$ \\
\hline
\end{tabular}
\end{center}
\caption{\label{table 3}Well-known modal reduction principles and their correspondents as relational inequalities}
    \label{fig:MPS:Famous}
\end{table}

The mathematically grounded relationships just discussed reflect also on the `preservation' of the {\em intuitive meaning} of Sahlqvist mrps across these different semantics, as we argue in the  following examples.

\begin{example} \label{ex:reflexivity}
The Sahlqvist mrps $p \leq \Diamond p$ and $\Box p \leq p$ correspond to the $\mathsf{KRel}_{\mathcal{L}}$-inequality $\Delta \subseteq R$
on generalized approximation spaces. When interpreting the accessibility relation $R$ as epistemic indiscernibility (e.g.~that of an agent), the informal interpretation of   $\Delta \subseteq R$ is  that if the agent can tell two states apart, then these states are distinct, and since in generalized approximation spaces two states are inherently indistinguishable iff they are identical, then this condition is equivalent to the condition that, if the agent can tell two states apart, then these states must not be inherently indistinguishable, which is by and large what the $\mathsf{GRel}_{\mathcal{L}}$-inequalities $E \subseteq R_\Box$, and $D \subseteq R_\Diamond$ (i.e.~the correspondents of the same mrps on graph-based frames) express, except for a subtle refinement. Namely, the information encoded in $E$ about any state $z$ of a graph-based frame presents itself from two perspectives:  
the set $E^{[1]}[z]$ of the states which are discernible from $z$, and  the set $E^{[0]}[z]$ of the states from which $z$ is  discernible. Since $E$ is not symmetric, these two perspectives are not  equivalent in general, and become equivalent when $E$ is symmetric.   
%
In general, inequality $E \subseteq R_\Box$ encodes the condition that, for any two states $x$ and $y$, if $x$ is inherently indiscernible from $y$, then  $x$  is  indistinguishable from $y$ according to the agent. Inequality   $D \subseteq R_\Diamond$ says that if $x$ is inherently indiscernible from $y$, then $y$ is  indistinguishable from $x$ also according to the agent. 
\end{example}
\begin{example} \label{ex:symmetry}
The Sahlqvist mrps $p \leq \Box \Diamond p$ and 
$\Diamond \Box p \leq  p$ correspond to the $\mathsf{KRel}_{\mathcal{L}}$-inequalities  $R^{-1} \subseteq R$ and $R \subseteq R^{-1}$
on Kripke frames. Under the same interpretation of  the accessibility relation as in the example above, this condition encodes the symmetry of the indiscernibility relation. 
The same mrps correspond to the $\mathsf{GRel}_{\mathcal{L}}$-inequalities  $R_{\Diamondblack}\subseteq R_{\Diamond} $ and $R_\Box\subseteq R_\blacksquare$ 
 on graph-based frames, which express the symmetry of the indiscernibility relations. Hence, the meaning of these conditions across the two semantics is both formally and intuitively verbatim the same. 
\end{example}
\begin{example}\label{ex:transitivity}
The Sahlqvist mrps $\Diamond\Diamond p \leq \Diamond p$ and 
$\Box p \leq \Box\Box p$ correspond to the $\mathsf{KRel}_{\mathcal{L}}$-inequality  $R \circ R \subseteq R$
on generalized approximation spaces. When interpreting the accessibility relation $R$ as epistemic indiscernibility (e.g.~that of an agent), the informal interpretation of  
 $R \circ R \subseteq R$ is that for any two states $x$ and $y$, if the agent can distinguish $y$ from $x$, then the agent can distinguish $y$  from every element in $R^{-1}[x]$, which is the set of all states  which the agent cannot tell apart from  $x$. 

The same mrps correspond to the $\mathsf{GRel}_{\mathcal{L}}$-inequalities $R_\Diamond \compdia R_\Diamond \subseteq R_\Diamond$, and  $R_\Box \compbox R_\Box \subseteq R_\Box$ on graph-based frames. Unravelling $R_\Diamond\, \compdia R_\Diamond \subseteq R_\Diamond$, we get that  $R_\Diamond^{[0]}[x] \subseteq R_\Diamond^{[0]} [E^{[0]}[R_\Diamond ^{[0]}[x]]]$ for every $x$. That is, if the agent can distinguish $y$ from $x$, then the agent can distinguish $y$  from any element of $E^{[0]}[R_\Diamond^{[0]} [x]]$, which is the set of all states which can be {\em (inherently) distinguished} from  every state which the agent can distinguish from $x$, that is, the set of all the states that the informational entropy allows to distinguish from all the states which the agent can distinguish from $x$. Hence, $E^{[0]}[R_\Diamond^{[0]} [x]]$ is `as close as it gets', in a scenario characterized by informational entropy, to the set of states that the agent cannot distinguish from $x$ in a scenario with no informational entropy. Similar considerations, from the other perspective, apply to the meaning of $R_\Box \compbox R_\Box \subseteq R_\Box$.
\end{example}

The above examples show that transferring the modal axioms defining approximation spaces to the hyperconstructive (graph-based) setting preserves their intuitive meaning. In some cases,  axioms which have the same first-order correspondent  on approximation spaces correspond to non-equivalent first-order conditions on graph-based frames, which hints at the comparatively greater richness of the hyperconstructivist setting.

We can now refer to  classes of hyperconstructivist approximation spaces defined by Sahlqvist mrps  using the same name we use for  the classes of generalized approximation spaces of which they are the shiftings. 
For example, {\em reflexive} graph-based frames are defined as graph-based frames satisfying the reflexivity axioms $D \subseteq R_\Diamond$ and $E \subseteq R_\Box$, while the class of symmetric graph-based frames are defined as  graph-based frames satisfying the symmetry axioms $R_\Box \subseteq R_\blacksquare$ and $R_\Diamondblack \subseteq R_\Diamond$. As discussed in Section \ref{sec:Rough set theory on graph-based frames}, this motivates us to define different sub-classes of hyperconstructivist approximation spaces like reflexive, symmetric, transitive and Pawlak approximation spaces (cf.~ Definition \ref{def:graph based approximation spaces}) as analogues of these classes of frames in classical approximation spaces to the graph-based frames. The following propositions follows immediately from the general results discussed in the previous section. 
\begin{proposition}\label{Pawlak axiom fraph-based}
For any Pawlak's hyperconstructivist approximation space  $\mathbb{F} =(Z,E,R_\Diamond, R_\Box)$,  and any $a,b \in \mathbb{F}^+$, 

\begin{center}
\begin{tabular}{lll}
 1.$\Diamond( a \vee b) = \Diamond a \vee  \Diamond b $    \quad\quad  & 2.$\Box (a \wedge b)  = \Box a \wedge  \Box b $ \quad\quad &   3.$ a \leq \Box b \Rightarrow \Diamond a \leq b$ \\
   4.$ \Diamond a \leq b \Rightarrow a \leq \Box b$  \quad\quad  & 5.$\Box a \leq a$ \quad\quad  & 6.$a \leq \Diamond a$\\
    7.$a \leq \Box\Diamond a$ \quad\quad & 8.$ \Diamond \Box a \leq a$ \quad\quad  & 9.$\Diamond\Diamond a  \leq \Diamond a$\\
     10.$ \Box a \leq \Box\Box a $. 
     \end{tabular}   
\end{center}
\end{proposition}
All the  conditions of the proposition above   are satisfied by the lower and upper approximation operators of Pawlak's approximation spaces.

  Different axiomatic classes of approximation spaces are used in rough set theory as models of approximate reasoning in different situations. 
  The systematic relationship established in the present paper among generalized, hyperconstructivist, and conceptual approximation spaces makes it possible  to capture a wider range of situations and reasoning modes while still taking advantage of the basic intuitions developed about the original setting. 

On the algebraic side, the present results pave the way to systematically extending rough set theory to algebras based on general (i.e.~not necessarily distributive) lattices.  Different classes of general lattices with approximation operators can be defined which as lattice-based counterparts of  classes of Boolean algebras with operators. Algebraic models of rough set theory based on various classes of lattices such as Boolean algebras, completely distributive lattices, ortholattices, De Morgan lattices (cf.~Section \ref{ssec:algebraic review}) can then be encompassed and studied uniformly  as subclasses of  lattice-based `rough algebras', which will then serve as  a suitable common ground for comparing and transferring insights and results concerning formal frameworks of rough set theory  defined on (or dual to) different classes of lattices.

\section{Conclusions}
\label{sec:conclusions}
The present paper provides a systematic theory to explain and generalize some observations about similarities in the shape of first-order correspondents of the common rough set theory axioms like reflexivity, symmetry and transitivity in Kripke and graph-based frames. The precise nature of these similarities is made explicit via the notion of shifting. This notion makes it possible to define the shifted counterparts of first-order conditions which are well known and widely used in modal logic and rough set theory; thus, we are now in a position to provide counterparts of rough theory axioms and different classes of approximation spaces in graph-based frames. 

\paragraph{Parametric correspondence.}  Together with \cite{conradie2022modal}, Theorems \ref{thm: sahlqvist shifting} and \ref{prop: sahlqvist lifting to polarities} suggest that the various correspondence theories for different logics and semantic contexts can be not only methodologically {\em unified} by the same algebraic and algorithmic mechanisms; they can also be {\em parametrically} related to each other in terms of their outputs. Furthermore, this parametric relation hinges on the natural connections between different semantic settings. In the present case, Kripke frames can be naturally embedded into graph-based frames, and graph-based frames into polarity-based frames. The present contribution can be understood as a preservation result of the first-order correspondents of inductive mrps under these natural embeddings, as shown in Figure \ref{fig:CampingTent}.

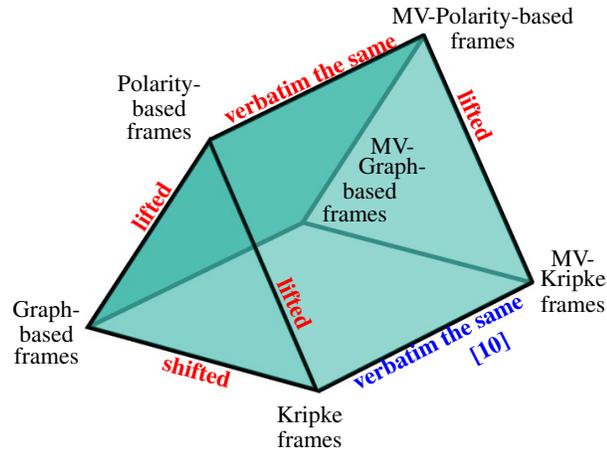
\begin{figure}[h]
\begin{center}
\begin{tikzpicture}
\node at (0,0) {\includegraphics[width=6cm]{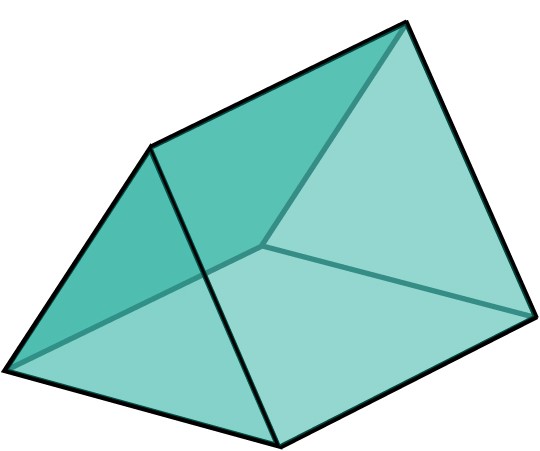}};
\draw (-3.5, 3.5) node {\textcolor{blue}{blue} = all Sahlqvist formulas};
\draw (-3, 3) node {\textcolor{red}{red} = all modal reduction principles};
\draw (3.5, -0.7) node {{MV-}};
\draw (3.5, -1) node {{Kripke}};
\draw (3.5, -1.3) node {{frames}};
\draw (1.7, -1.8) node  [color = blue, rotate = 27]  {{\bf verbatim the same}};
\draw (2.4, -1.9) node  [color = blue, rotate = 27]  {{\bf \cite{BritzMScThesis}}};

\draw (0, -2.8) node {{Kripke }};
\draw (0, -3.1) node {{frames}};
\draw (-0.2, -1.3) node [color = red, rotate = -65] {\bf lifted};
\draw (2.2, 1.2) node [color = red, rotate = -65] {\bf lifted};
\draw (2.3, 2.5) node {{MV-Polarity-based}};
\draw (2.3, 2.2) node {{frames}};
\draw (0, 1.7) node [color = red, rotate = 25] {\bf verbatim the same};
\draw (-2, 1.6) node {{{Polarity-}}};
\draw (-2, 1.3) node {{{based}}};
\draw (-2, 1) node {{{frames}}};
\draw (-2.1, 0) node [color = red, rotate = 57] {\bf lifted};
\draw (-1.5, -2.2) node [color = red, rotate = -15] {\bf shifted};
\draw (-3.5, -1.4) node  {{{Graph-}}};
\draw (-3.5, -1.7) node {{{based}}};
\draw (-3.5, -2) node {{{frames}}};

\draw (1.1, 0.8) node {{{MV-}}};
\draw (1.1, 0.5) node {{{Graph-}}};
\draw (0.8, 0.2) node {{{based}}};
\draw (0.6, -0.1) node {{{frames}}};
\end{tikzpicture}
\end{center}
\caption{A commutative diagram of semantic contexts.}
    \label{fig:CampingTent}
\end{figure}

\paragraph{Notation.} A crucial step in achieving these results is the ability to express the first-order correspondents of modal axioms in various semantic contexts in a more algebraic language, specifically in the languages of (heterogeneous) relation algebras. As we illustrated with the example of the modal transitivity axiom, the mathematical and meaning-based connections are near impossible to detect when the correspondents are expressed in the standard first-order frame correspondence languages of Kripke, polarity and graph based frames, respectively, yet become transparently clear when written in a suitable relational algebraic language as inclusions of relational compositions. The use of these languages to express correspondents is essential in the formulations and proofs of the general results given in Theorems \ref{thm: sahlqvist shifting} and \ref{prop: sahlqvist lifting to polarities}.


\paragraph{Future work.} The present work can be further pursued in several directions. Firstly, the perfect modularity of the graph-based setting allows to systematically study natural subclasses of hyperconstructivist approximation spaces. One example of such a class would be the hyperconstructivist counterparts of tolerance spaces \cite{skowron1996tolerance}, obtained by assuming the symmetry of the indiscernibility relation.
Secondly, similar shifting/lifting results may be sought for other logics, including polyadic modal logics and many-valued modal logics (i.e., a similar correspondence may be shown between many-valued Kripke frames and many-valued graph-based frames). For instance, a natural extension of the present work concerns investigating the remaining connections between the semantics presently investigated and their many-valued versions, as represented by the three un-annotated edges in Figure \ref{fig:CampingTent}.  Thirdly, the class of formulas covered by the shifting/lifting could possibly be expanded to more than inductive modal reduction principles. However, we do not expect this to be possible for all formulas with first-order frame correspondents, not even for all inductive formulas. Indeed, there exist inductive formulas whose corespondents over polarity-based frames are  {\em not} liftings of their correspondents over Kripke frames (see the concluding section of \cite{conradie2022modal}), and we expect similar examples to exist in the case of shiftings to graph-based frames. Fourthly, the literature contains instances of correspondence theory for modal logics on non-classical base obtained via G\"{o}del-McKinsey-Tarski (GMT) translation \cite{Godel:Interp,mckinsey1948some} to the multi-modal classical context (see e.g.~\cite{GeNaVe05} for an instance in the context of distributive modal logic and \cite{CoPaZh16a} for one in the context of bi-intuitionistic modal logic). In the context of graph-based frames, a similar approach suggests itself, whereby one would translate formulas over graph-based frames into formulas over Kripke frames, prefixing propositional variables with modalities interpreted with the relation $E$ to simulate the closure of valuations.   However, this approach involves several challenges, especially on the proof-theoretic front, not least of which is the fact that conditions such as the $E$-compatibility of the additional relations can {\em not} be represented by (analytic) inductive inequalities, and thus, the general methods for developing analytic display calculi for modal logics developed in \cite{GrMaPaTzZh15} would not be available, while they are straightforwardly applicable to the present setting. Future work will investigate such a GMT-translation based approach but, at present, it remains to be seen whether the parametric connection between correspondents on Kripke and graph-based frames can also be obtained in this way. Finally, parametric connections between the correspondents of modal formulas remain to be investigated between other semantics environments.

\subsubsection{Acknowledgements.} 
The research of Krishna Manoorkar is supported by the NWO grant KIVI.2019.001 awarded to Alessandra Palmigiano. The authors   confirm that there are no relevant financial or non-financial competing interests to report.


%
%
%
%

\bibliographystyle{splncs04}
\bibliography{mybib}

\begin{thebibliography}{10}
\providecommand{\url}[1]{\texttt{#1}}
\providecommand{\urlprefix}{URL }
\providecommand{\doi}[1]{https://doi.org/#1}

\bibitem{balbiani2001first}
Balbiani, P., Vakarelov, D.: First-order characterization and modal analysis of indiscernibility and complementarity in information systems. In: Proc.~ECSQARU. pp. 772--781. Springer (2001)

\bibitem{banerjee1994rough}
Banerjee, M., Chakraborty, M.: Rough consequence and rough algebra. In: Rough Sets, Fuzzy Sets and Knowledge Discovery, pp. 196--207. Springer (1994)

\bibitem{banerjee2004algebras}
Banerjee, M., Chakraborty, M.: Algebras from rough sets. In: Rough-Neural Computing, pp. 157--184. Springer (2004)

\bibitem{banerjee1996rough}
Banerjee, M., Chakraborty, M.K.: Rough sets through algebraic logic. Fundamenta Informaticae  \textbf{28}(3-4),  211--221 (1996)

\bibitem{vanBenthem:Reduction:Principles}
van Benthem, J.: Modal reduction principles. Journal of Symbolic Logic  \textbf{41}(2),  301--312 (06 1976), \url{http://projecteuclid.org/euclid.jsl/1183739770}

\bibitem{icla_algebraic}
van~der Berg, I., De~Domenico, A., Greco, G., Manoorkar, K.B., Palmigiano, A., Panettiere, M.: Labelled calculi for lattice-based modal logics. In: Banerjee, M., Sreejith, A.V. (eds.) Logic and Its Applications. pp. 23--47. Springer Nature Switzerland, Cham (2023)

\bibitem{icla_relational}
van~der Berg, I., De~Domenico, A., Greco, G., Manoorkar, K.B., Palmigiano, A., Panettiere, M.: Labelled calculi for the logics of rough concepts. In: Banerjee, M., Sreejith, A.V. (eds.) Logic and Its Applications. pp. 172--188. Springer Nature Switzerland, Cham (2023)

\bibitem{blackburn2002modal}
Blackburn, P., De~Rijke, M., Venema, Y.: Modal logic, vol.~53. Cambridge University Press (2002)

\bibitem{bou2011minimum}
Bou, F., Esteva, F., Godo, L., Rodr{\'\i}guez, R.O.: On the minimum many-valued modal logic over a finite residuated lattice. Journal of Logic and computation  \textbf{21}(5),  739--790 (2011)

\bibitem{BritzMScThesis}
Britz, C.: {Correspondence theory in many-valued modal logics}. Master's thesis, University of Johannesburg, South Africa (2016)

\bibitem{cattaneo2016connection}
Cattaneo, G., Chiaselotti, G., Ciucci, D., Gentile, T.: On the connection of hypergraph theory with formal concept analysis and rough set theory. Information Sciences  \textbf{330},  342--357 (2016)

\bibitem{cattaneo1996mathematical}
Cattaneo, G.: Mathematical foundations of roughness and fuzziness (1996)

\bibitem{cattaneo2004algebraic}
Cattaneo, G., Ciucci, D.: Algebraic structures for rough sets. In: Transactions on Rough sets II, pp. 208--252. Springer (2004)

\bibitem{chakraborty2011fuzzy}
Chakraborty, M.K.: On fuzzy sets and rough sets from the perspective of indiscernibility. In: Indian Conference on Logic and Its Applications. pp. 22--37. Springer (2011)

\bibitem{conradie2015relational}
Conradie, W., Craig, A.: Relational semantics via {TiRS} graphs. TACL 2015  (2015)

\bibitem{TarkPaper2017}
Conradie, W., Frittella, S., Palmigiano, A., Piazzai, M., Tzimoulis, A., Wijnberg, N.: Toward an epistemic-logical theory of categorization. In: Proc.~TARK 2017. EPTCS, vol.~251, pp. 170--189 (2017)

\bibitem{CoPa:non-dist}
Conradie, W., Palmigiano, A.: {Algorithmic correspondence and canonicity for non-distributive logics}. Annals of Pure and Applied Logic p. DOI: 10.1016/j.apal.2019.04.003 (2019)

\bibitem{Conradie2019/08}
Conradie, W., Craig, A., Palmigiano, A., Wijnberg, N.: Modelling competing theories. In: Proceedings of the 11th Conference of the European Society for Fuzzy Logic and Technology (EUSFLAT 2019). pp. 721--739. Atlantis Press (2019/08). \doi{https://doi.org/10.2991/eusflat-19.2019.100}, \url{https://doi.org/10.2991/eusflat-19.2019.100}

\bibitem{graph-based-wollic}
Conradie, W., Craig, A., Palmigiano, A., Wijnberg, N.M.: Modelling informational entropy. In: Iemhoff, R., Moortgat, M., de~Queiroz, R. (eds.) Logic, Language, Information, and Computation. pp. 140--160. Springer Berlin Heidelberg, Berlin, Heidelberg (2019)

\bibitem{conradie2022modal}
Conradie, W., De~Domenico, A., Manoorkar, K., Palmigiano, A., Panettiere, M., Prieto, D.P., Tzimoulis, A.: Modal reduction principles across relational semantics. arXiv preprint arXiv:2202.00899  (2022)

\bibitem{CONRADIE2021371}
Conradie, W., Frittella, S., Manoorkar, K., Nazari, S., Palmigiano, A., Tzimoulis, A., Wijnberg, N.M.: Rough concepts. Information Sciences  \textbf{561},  371--413 (2021). \doi{https://doi.org/10.1016/j.ins.2020.05.074}, \url{https://www.sciencedirect.com/science/article/pii/S0020025520304989}

\bibitem{conradie2016categories}
Conradie, W., Frittella, S., Palmigiano, A., Piazzai, M., Tzimoulis, A., Wijnberg, N.M.: Categories: how i learned to stop worrying and love two sorts. In: International Workshop on Logic, Language, Information, and Computation. pp. 145--164. Springer (2016)

\bibitem{CoGhPa14}
Conradie, W., Ghilardi, S., Palmigiano, A.: Unified correspondence. In: Baltag, A., Smets, S. (eds.) Johan van Benthem on Logic and Information Dynamics, Outstanding Contributions to Logic, vol.~5, pp. 933--975. Springer International Publishing (2014). \doi{10.1007/978-3-319-06025-5_36}, \url{http://dx.doi.org/10.1007/978-3-319-06025-5_36}

\bibitem{ConPal12}
Conradie, W., Palmigiano, A.: Algorithmic correspondence and canonicity for distributive modal logic. Annals of Pure and Applied Logic  \textbf{163}(3),  338 -- 376 (2012)

\bibitem{conradie2016constructive}
Conradie, W., Palmigiano, A.: Constructive canonicity of inductive inequalities. Logical Methods in Computer Science  \textbf{16} (2020)

\bibitem{Tark1}
Conradie, W., Palmigiano, A., Robinson, C., Tzimoulis, A., Wijnberg, N.: Modelling socio-political competition. Submitted  (2019)

\bibitem{CONRADIE2021115}
Conradie, W., Palmigiano, A., Robinson, C., Tzimoulis, A., Wijnberg, N.: Modelling socio-political competition. Fuzzy Sets and Systems  \textbf{407},  115--141 (2021). \doi{https://doi.org/10.1016/j.fss.2020.02.005}, \url{https://www.sciencedirect.com/science/article/pii/S016501142030049X}, knowledge Representation and Logics

\bibitem{Conradie2020NondistributiveLF}
Conradie, W., Palmigiano, A., Robinson, C., Wijnberg, N.: Nondistributive logics --- from semantics to meaning logics. In: Rezu\c{s}, A. (ed.) Contemporary Logic and Computing, Landscapes in Logic, vol.~1. College Publications (2020)

\bibitem{ConPalSou}
Conradie, W., Palmigiano, A., Sourabh, S.: Algebraic modal correspondence: {S}ahlqvist and beyond. Journal of Logical and Algebraic Methods in Programming  \textbf{91},  60--84 (2017)

\bibitem{CoPaZh16a}
Conradie, W., Palmigiano, A., Zhao, Z.: Sahlqvist via translation. Logical Methods in Computer Science  \textbf{15} (2019)

\bibitem{craig2015tirs}
Craig, A., Gouveia, M., Haviar, M.: {TiRS} graphs and {TiRS} frames: a new setting for duals of canonical extensions. Algebra universalis  \textbf{74}(1-2),  123--138 (2015)

\bibitem{craig-priestley}
Craig, A., Haviar, M., Priestley, H.: A fresh perspective on canonical extensions for bounded lattices. Applied Categorical Structures  \textbf{21}(6),  725--749 (2013)

\bibitem{dai2021rough}
Dai, S.: Rough approximation operators on a complete orthomodular lattice. Axioms  \textbf{10}(3), ~164 (2021)

\bibitem{davey2002introduction}
Davey, B., Priestley, H.: Introduction to lattices and order. Cambridge univ. press (2002)

\bibitem{10.1007/978-3-319-50901-3_30}
Du, S., Gregory, S.: The echo chamber effect in twitter: does community polarization increase? In: Cherifi, H., Gaito, S., Quattrociocchi, W., Sala, A. (eds.) Complex Networks {\&} Their Applications V. pp. 373--378. Springer International Publishing, Cham (2017)

\bibitem{dubois1990rough}
Dubois, D., Prade, H.: Rough fuzzy sets and fuzzy rough sets. International Journal of General System  \textbf{17}(2-3),  191--209 (1990)

\bibitem{fagin1990logic}
Fagin, R., Halpern, J.Y., Megiddo, N.: A logic for reasoning about probabilities. Information and computation  \textbf{87}(1-2),  78--128 (1990)

\bibitem{fine1972so}
Fine, K.: In so many possible worlds. Notre Dame Journal of formal logic  \textbf{13}(4),  516--520 (1972)

\bibitem{fitting1991many}
Fitting, M.: Many-valued modal logics. Fundam. Inform.  \textbf{15}(3-4),  235--254 (1991)

\bibitem{fitting1992many}
Fitting, M.: Many-valued model logics ii. Fundam. Inform.  \textbf{17}(1-2),  55--73 (1992)

\bibitem{formica2018integrating}
Formica, A.: Integrating fuzzy formal concept analysis and rough set theory for the semantic web. Bulletin of Computational Applied Mathematics  \textbf{6}(2) (2018)

\bibitem{ganter2012formal}
Ganter, B., Wille, R.: Formal concept analysis: mathematical foundations. Springer Science \& Business Media (2012)

\bibitem{GeHa01}
Gehrke, M., Harding, J.: Bounded lattice expansions. Journal of Algebra  \textbf{238}(1),  345--371 (2001)

\bibitem{GeNaVe05}
Gehrke, M., Nagahashi, H., Venema, Y.: A {S}ahlqvist theorem for distributive modal logic. Annals of Pure and Applied Logic  \textbf{131}(1-3),  65--102 (2005)

\bibitem{Godel:Interp}
G\"{o}del, K.: Eine interpretation des intnitionistischen aussagenkalkuls. {E}rgebnisse eines mathematischen {K}olloquiums  \textbf{6},  39--40 (1933)

\bibitem{godo2003belieffunctions}
Godo, L., H{\'a}jek, P., Esteva, F.: A fuzzy modal logic for belief functions. Fundamenta Informaticae  \textbf{57}(2-4),  127--146 (2003)

\bibitem{greco2019logics}
Greco, G., Jipsen, P., Manoorkar, K., Palmigiano, A., Tzimoulis, A.: Logics for rough concept analysis. In: Proc.~ICLA 2019. pp. 144--159. Springer (2019)

\bibitem{greco2019proper}
Greco, G., Liang, F., Manoorkar, K., Palmigiano, A.: Proper multi-type display calculi for rough algebras. Electronic Notes in Theoretical Computer Science  \textbf{344},  101--118 (2019)

\bibitem{GrMaPaTzZh15}
Greco, G., Ma, M., Palmigiano, A., Tzimoulis, A., Zhao, Z.: Unified correspondence as a proof-theoretic tool. Journal of Logic and Computation  (2016, doi: 101093/logcom/exw022 ArXiv preprint 160308204)

\bibitem{kent1996rough}
Kent, R.E.: Rough concept analysis: a synthesis of rough sets and formal concept analysis. Fundamenta Informaticae  \textbf{27}(2, 3),  169--181 (1996)

\bibitem{kumar2020study}
Kumar, A.: A study of algebras and logics of rough sets based on classical and generalized approximation spaces. In: Transactions on Rough Sets XXII, pp. 123--251. Springer (2020)

\bibitem{kumar2015algebras}
Kumar, A., Banerjee, M.: Algebras of definable and rough sets in quasi order-based approximation spaces. Fundamenta Informaticae  \textbf{141}(1),  37--55 (2015)

\bibitem{liu2008axiomatic}
Liu, G.: Axiomatic systems for rough sets and fuzzy rough sets. International Journal of Approximate Reasoning  \textbf{48}(3),  857--867 (2008)

\bibitem{ma2018sequent}
Ma, M., Chakraborty, M.K., Lin, Z.: Sequent calculi for varieties of topological quasi-boolean algebras. In: International Joint Conference on Rough Sets. pp. 309--322. Springer (2018)

\bibitem{mckinsey1948some}
McKinsey, J.C., Tarski, A.: Some theorems about the sentential calculi of lewis and heyting. The journal of symbolic logic  \textbf{13}(1),  1--15 (1948)

\bibitem{moshier2016relational}
Moshier, M.: A relational category of formal contexts. Preprint  (2016)

\bibitem{orlowska1994rough}
Orlowska, E.: Rough set semantics for non-classical logics. In: Rough Sets, Fuzzy Sets and Knowledge Discovery, pp. 143--148. Springer (1994)

\bibitem{pawlak2007rudiments}
Pawlak, Z., Skowron, A.: Rudiments of rough sets. Information sciences  \textbf{177}(1),  3--27 (2007)

\bibitem{pawlak1982rough}
Pawlak, Z.: Rough sets. International journal of computer \& information sciences  \textbf{11}(5),  341--356 (1982)

\bibitem{pawlak1984rough}
Pawlak, Z.: Rough probability. Bull. Pol. Acad. Sci. Math  \textbf{32},  607--615 (1984)

\bibitem{pawlak1998rough}
Pawlak, Z.: Rough set theory and its applications to data analysis. Cybernetics \& Systems  \textbf{29}(7),  661--688 (1998)

\bibitem{ploscica1994}
Plo\v{s}\v{c}ica, M.: A natural representation of bounded lattices. Tatra Mountains Mathematical Publications  \textbf{4} (1994)

\bibitem{powers2019shouting}
Powers, E., Koliska, M., Guha, P.: “shouting matches and echo chambers”: perceived identity threats and political self-censorship on social media. International Journal of Communication  \textbf{13}, ~20 (2019)

\bibitem{radzikowska2004fuzzy}
Radzikowska, A.M., Kerre, E.E.: Fuzzy rough sets based on residuated lattices. In: Transactions on Rough Sets II, pp. 278--296. Springer (2004)

\bibitem{rainie2012social}
Rainie, L., Smith, A.: Social networking sites and politics. Washington, DC: Pew Internet \& American Life Project. Retrieved June  \textbf{12}, ~2012 (2012)

\bibitem{rasiowa1984rough}
Rasiowa, H., Skowron, A.: Rough concepts logic. In: Symposium on Computation Theory. pp. 288--297. Springer (1984)

\bibitem{rasouli2010roughness}
Rasouli, S., Davvaz, B.: Roughness in mv-algebras. Information Sciences  \textbf{180}(5),  737--747 (2010)

\bibitem{saha2014algebraic}
Saha, A., Sen, J., Chakraborty, M.K.: Algebraic structures in the vicinity of pre-rough algebra and their logics. Information Sciences  \textbf{282},  296--320 (2014)

\bibitem{saha2016algebraic}
Saha, A., Sen, J., Chakraborty, M.K.: Algebraic structures in the vicinity of pre-rough algebra and their logics ii. Information Sciences  \textbf{333},  44--60 (2016)

\bibitem{sahlqvist1975completeness}
Sahlqvist, H.: Completeness and correspondence in the first and second order semantics for modal logic. In: Studies in Logic and the Foundations of Mathematics, vol.~82, pp. 110--143. Elsevier (1975)

\bibitem{skowron1996tolerance}
Skowron, A., Stepaniuk, J.: Tolerance approximation spaces. Fundamenta Informaticae  \textbf{27}(2-3),  245--253 (1996)

\bibitem{sun2008fuzzy}
Sun, B., Gong, Z., Chen, D.: Fuzzy rough set theory for the interval-valued fuzzy information systems. Information Sciences  \textbf{178}(13),  2794--2815 (2008)

\bibitem{doi:10.1080/1369118X.2013.862563}
Thorson, K.: Facing an uncertain reception: young citizens and political interaction on facebook. Information, Communication \& Society  \textbf{17}(2),  203--216 (2014). \doi{10.1080/1369118X.2013.862563}, \url{https://doi.org/10.1080/1369118X.2013.862563}

\bibitem{vakarelov1991model}
Vakarelov, D.: A modal logic for similarity relations in {P}awlak knowledge representation systems. Fundam. Inform.  \textbf{15}(1),  61--79 (1991)

\bibitem{vakarelov2005modal}
Vakarelov, D.: A modal characterization of indiscernibility and similarity relations in pawlak’s information systems. In: International Workshop on Rough Sets, Fuzzy Sets, Data Mining, and Granular-Soft Computing. pp. 12--22. Springer (2005)

\bibitem{van1976mrp}
Van~Benthem, J.: Modal reduction principles. The Journal of Symbolic Logic  \textbf{41}(2),  301--312 (1976)

\bibitem{van1984correspondence}
Van~Benthem, J.: Correspondence theory. In: Handbook of philosophical logic, pp. 167--247. Springer (1984)

\bibitem{vanBenthem2001}
Van~Benthem, J.: Correspondence Theory, pp. 325--408. Springer Netherlands, Dordrecht (2001). \doi{10.1007/978-94-017-0454-0_4}, \url{https://doi.org/10.1007/978-94-017-0454-0_4}

\bibitem{vraga2015individual}
Vraga, E.K., Thorson, K., Kligler-Vilenchik, N., Gee, E.: How individual sensitivities to disagreement shape youth political expression on facebook. Computers in Human Behavior  \textbf{45},  281--289 (2015)

\bibitem{wu2004constructive}
Wu, W.Z., Zhang, W.X.: Constructive and axiomatic approaches of fuzzy approximation operators. Information sciences  \textbf{159}(3-4),  233--254 (2004)

\bibitem{wybraniec1989generalization}
Wybraniec-Skardowska, U.: On a generalization of approximation space. Bulletin of the Polish Academy of Sciences. Mathematics  \textbf{37}(1-6),  51--62 (1989)

\bibitem{yao2004comparative}
Yao, Y.: A comparative study of formal concept analysis and rough set theory in data analysis. In: International Conference on Rough Sets and Current Trends in Computing. pp. 59--68. Springer (2004)

\bibitem{yao2008probabilistic}
Yao, Y.: Probabilistic rough set approximations. International journal of approximate reasoning  \textbf{49}(2),  255--271 (2008)

\bibitem{yao2016rough}
Yao, Y.: Rough-set concept analysis: interpreting rs-definable concepts based on ideas from formal concept analysis. Information Sciences  \textbf{346},  442--462 (2016)

\bibitem{yao2020three}
Yao, Y.: Three-way granular computing, rough sets, and formal concept analysis. International Journal of Approximate Reasoning  \textbf{116},  106--125 (2020)

\bibitem{yao1998generalized}
Yao, Y.: Generalized rough set models. Rough sets in knowledge discovery  \textbf{1},  286--318 (1998)

\bibitem{yao2004concept}
Yao, Y.: Concept lattices in rough set theory. In: Fuzzy Information, 2004. Processing NAFIPS'04. IEEE Annual Meeting of the. vol.~2, pp. 796--801. IEEE (2004)

\bibitem{yao2006unifying}
Yao, Y.: On unifying formal concept analysis and rough set analysis. Rough Set and Concept Lattice, Xi’an Jiaotong University Press, Xi’an pp. 1--20 (2006)

\bibitem{yao1996-Sahlqvist}
Yao, Y., Lin, T.Y.: Generalization of rough sets using modal logics. Intelligent Automation \& Soft Computing  \textbf{2}(2),  103--119 (1996)

\bibitem{zakowski1983approximations}
Zakowski, W.: Approximations in the space (u, $\pi$). Demonstratio mathematica  \textbf{16}(3),  761--770 (1983)

\end{thebibliography}
\end{document}